\documentclass[a4paper,12pt]{article}
\usepackage[T1]{fontenc}
\usepackage[utf8]{inputenc}
\usepackage{textcomp}

\usepackage{indentfirst}   % 核心：section/subsection 之后第一段也缩进
\makeatletter
\@ifl@t@r\fmtversion{2020/10/01}
  {\IfFileExists{newtxtext.sty}{\usepackage[osf]{newtxtext}}{\usepackage{times}}}
  {\usepackage{times}}
\makeatother
\usepackage{mathtools,amssymb,amsthm,mathrsfs}
\usepackage{booktabs,array,tabularx,needspace,placeins}
\newcommand{\PageMargin}{25mm}
\usepackage[a4paper,margin=\PageMargin,
            headheight=15pt,headsep=7mm]{geometry}
\usepackage{microtype,titlesec,fancyhdr}
\usepackage[shortlabels]{enumitem}
\usepackage{xcolor,aliascnt}
\usepackage{tikz}
\usepackage[numbers,sort&compress]{natbib}
\definecolor{burgundy}{RGB}{128,0,32}
\definecolor{controlfill}{RGB}{110,180,228}
\definecolor{controledge}{RGB}{0,100,132}
\definecolor{phaseorange}{RGB}{190,92,18}
\definecolor{ink}{RGB}{35,45,55}
\usepackage[colorlinks=true,allcolors=burgundy,linktoc=all]{hyperref}
\IfFileExists{bookmark.sty}{\usepackage{bookmark}}{}

\DeclareSymbolFont{exsm}{OT1}{ptm}{m}{n}
\DeclareMathSymbol({\mathopen}{exsm}{`(}
\DeclareMathSymbol){\mathclose}{exsm}{`)}
\DeclareMathSymbol[{\mathopen}{exsm}{`[}
\DeclareMathSymbol]{\mathclose}{exsm}{`]}
\DeclareMathAlphabet{\mathbf}{OT1}{ptm}{bx}{n}
\DeclareMathAlphabet{\mathrm}{OT1}{ptm}{m}{n}
\DeclareMathAlphabet{\mathit}{OT1}{ptm}{m}{it}
\DeclareSymbolFont{sfoperators}{OT1}{ptm}{m}{n}
\makeatletter
\def\operator@font{\mathgroup\symsfoperators}
\makeatother

\titleformat{\section}{\centering\large\bfseries}{\thesection}{.75em}{}
\titleformat{\subsection}[runin]{\normalsize\bfseries}{\thesubsection}{.75em}{}[.]
\titlespacing*{\section}{0pt}{3ex plus 1ex minus .2ex}{2ex plus .2ex}
\titlespacing*{\subsection}{0pt}{2ex plus 1ex minus .2ex}{1em}
\allowdisplaybreaks
\newcommand{\R}{\mathbb R}
\newcommand{\C}{\mathbb C}

\newcommand{\Om}{\Omega}
\newcommand{\dd}{\,\mathrm d}
\newcommand{\ip}[2]{\langle #1,#2\rangle}
\newcommand{\norm}[1]{\lVert #1\rVert}
\newcommand{\abs}[1]{\lvert #1\rvert}
\newcommand{\supp}{\operatorname{supp}}
\newcommand{\dist}{\operatorname{dist}}

\newcommand{\Real}{\operatorname{Re}}
\newcommand{\Imag}{\operatorname{Im}}
\newcommand{\Ss}{\mathcal S_s}
\newcommand{\Ys}{\mathcal Y_s}
\newcommand{\Us}{\mathcal U_s}
\newcommand{\Hs}{\mathcal H_s}
\newcommand{\one}{\mathbf 1}

\numberwithin{equation}{section}
\newtheorem{theorem}{Theorem}[section]
\newaliascnt{lemma}{theorem}
\newtheorem{lemma}[lemma]{Lemma}
\aliascntresetthe{lemma}
\newaliascnt{proposition}{theorem}
\newtheorem{proposition}[proposition]{Proposition}
\aliascntresetthe{proposition}
\newaliascnt{corollary}{theorem}
\newtheorem{corollary}[corollary]{Corollary}
\aliascntresetthe{corollary}
\theoremstyle{definition}
\newaliascnt{definition}{theorem}
\newtheorem{definition}[definition]{Definition}
\aliascntresetthe{definition}
\newaliascnt{assumption}{theorem}
\newtheorem{assumption}[assumption]{Assumption}
\aliascntresetthe{assumption}
\newaliascnt{example}{theorem}
\newtheorem{example}[example]{Example}
\aliascntresetthe{example}
\theoremstyle{remark}
\newaliascnt{remark}{theorem}
\newtheorem{remark}[remark]{Remark}
\aliascntresetthe{remark}
\newaliascnt{problem}{theorem}

\aliascntresetthe{problem}
\newtheorem*{remark*}{Remark}
\usepackage[capitalize,noabbrev]{cleveref}
\newcolumntype{Y}{>{\raggedright\arraybackslash}X}
\setlist[enumerate]{label=\textup{(\roman*)},leftmargin=*,itemsep=3pt,topsep=5pt}
\newcommand{\TitleToDateGap}{4mm}          % 标题与日期之间的额外间距
\newcommand{\BeforeAuthorGap}{8mm}         % 姓名前的额外间距（原为 4mm）
\newcommand{\AuthorToAffiliationGap}{3mm}  % 姓名与单位之间的额外间距
\newcommand{\AffiliationToEmailGap}{2mm}  % 单位与邮箱之间的额外间距
\newcommand{\BeforeAbstractTitleExtraGap}{2mm} % Abstract 标题上方额外间距；0mm 恢复原间距
\newcommand{\AbstractTitleToTextGap}{2mm}      % Abstract 标题与摘要正文之间的间距
\newcommand{\BeforeIntroductionExtraGap}{10mm} % 摘要与 Introduction 之间额外增加的间距；0mm 恢复原间距
\newcommand{\AfterTitleBlockGap}{3mm}     % 标题信息块之后的额外间距
\title{Global exact controllability for Schr\"odinger equations with Lipschitz and logarithmic nonlinearities}
\author{Lei Zhang\textsuperscript{1,*} and Xiuting Li\textsuperscript{2}}
\date{\today}
 
\newcommand{\authorinfo}{%
  \textsuperscript{1}School of Mathematics and Statistics, Hubei Key Laboratory of\\
  Engineering Modeling and Scientific Computing, Huazhong University of Science and Technology\par
  \textsuperscript{2}Department of Mathematics and Statistics, College of Informatics,\\
  Huazhong Agricultural University\par
  \vspace{\AffiliationToEmailGap}%
  \begin{tabular}{@{}l@{\quad}l@{\quad}l@{}}
    Emails: & \texttt{lei\_zhang@hust.edu.cn} & (L. Zhang)\\
            & \texttt{xtingli@mail.hzau.edu.cn} & (X. Li)
  \end{tabular}\par
  \smallskip Corresponding author: Lei Zhang.\par}
\makeatletter
\renewcommand{\maketitle}{%
  \thispagestyle{empty}%
  \par\begingroup
  \centering
  \setlength{\parskip}{0pt}% Avoid accumulated paragraph spacing.
  {\Large\bfseries\@title\par}%
  \ifx\@date\@empty\else
    \vspace{\TitleToDateGap}%
    {\normalsize\@date\par}%
  \fi
  \vspace{\BeforeAuthorGap}%
  \begin{minipage}{.94\textwidth}
    \centering
    {\large\bfseries\@author\par}%
    \vspace{\AuthorToAffiliationGap}%
    {\small\authorinfo\par}%
  \end{minipage}\par
  \endgroup
  \vspace{\AfterTitleBlockGap}}
\makeatother
\renewenvironment{abstract}
  {\begin{center}\begin{minipage}{.94\textwidth}\small
   \vspace*{\BeforeAbstractTitleExtraGap}%
   {\centering\textbf{Abstract}\par}\vspace{\AbstractTitleToTextGap}\noindent}
  {\end{minipage}\end{center}}
\hypersetup{pdfstartview={},
  pdftitle={Global exact controllability for Schrodinger equations with Lipschitz and logarithmic nonlinearities},
  pdfauthor={Lei Zhang; Xiuting Li},
  pdfsubject={Semilinear Schrodinger control, parameter-dependent Carleman estimates, and logarithmic nonlinearities}}

\begin{document}
\maketitle
\begin{abstract} 
\qquad  In this paper, we study global exact controllability of semilinear Schr\"odinger equations by localized internal additive controls. For every globally Lipschitz nonlinearity in arbitrary dimensions, we prove that any two $L^2$ states can be connected in any prescribed positive time, without smallness assumptions on the states or the Lipschitz constant. The result holds on smooth bounded Euclidean domains under a local convex-phase condition, and on an interval with any nonempty open control region. It also extends to compact Riemannian manifolds under the corresponding phase condition. In one dimension, we further establish controllability for logarithmic nonlinearities of the form
\[
f(z)=z\phi(|z|)+\mathfrak g(z),
\]
where $\phi$ is real-valued, $\mathfrak g$ is globally Lipschitz, and $r|\phi'(r)|\leq D(1+|\log r|)^\gamma$ with $0\leq\gamma<1/2$. For the endpoint case $\gamma=1/2$, we achieve the result by assuming the leading coefficient is sufficiently small. These results resolve the global Lipschitz controllability question raised by Zuazua \cite[Section~3]{Zuazua2003} in the stated geometric settings and establish its logarithmic extension within the specified one-dimensional classes.

\vspace{1em}
\end{abstract}
\vspace{\BeforeIntroductionExtraGap}

\section{Introduction}\label{sec:introduction}

Our main goal is to establish global exact controllability of semilinear
Schr\"odinger equations between arbitrary $L^2$ states in every
prescribed positive time, using localized internal additive controls.
Let $\Om\subset\R^n$, $n\geq1$, be a bounded, connected domain with
$C^\infty$ boundary, and let $\omega\subset\Om$ be a nonempty open
set. For $T>0$, set
\[
 Q=(0,T)\times\Om,\qquad
 Q_\omega=(0,T)\times\omega,\qquad
 \Sigma=(0,T)\times\partial\Om.
\]
We consider the controlled equation
\begin{equation}\label{eq:main}
 \begin{cases}
 i\partial_t y=\Delta y+f(y)+\one_\omega h&\text{in }Q,\\
 y=0&\text{on }\Sigma,\\
 y(0)=y_0&\text{in }\Om.
 \end{cases}
\end{equation}
The state $y$ and control $h$ are complex valued, and
$\one_\omega h$ denotes the extension of $h$ by zero to $\Om$.
More precisely, for every prescribed $T>0$ and every pair
$y_0,y_1\in L^2(\Om;\C)$, we seek a control
$h\in L^2(Q_\omega;\C)$ such that a solution of \eqref{eq:main}
satisfies $y(T)=y_1$. The time $T$ is fixed independently of the
endpoint norms. We treat globally Lipschitz nonlinearities in
arbitrary dimension and nonlinearities with logarithmic growth in
one dimension.

The questions raised in \cite[Section~3]{Zuazua2003} concern the
passage from sufficiently small Lipschitz perturbations of the free
equation to general globally Lipschitz nonlinearities, and its extension
to logarithmically superlinear growth in arbitrary positive time;
see also the discussion in \cite[Section~6.2.2]{Laurent2014}.
For \eqref{eq:main}, these questions ask whether, under suitable geometric
assumptions on $(\Om,\omega)$, arbitrary $L^2(\Om;\C)$ states can
be joined in every prescribed positive time in the following cases:
\begin{enumerate}
	\item [\textbf{(a)}] $f:\mathbb{C}\to\mathbb{C}$ is globally Lipschitz, with no smallness restriction imposed on its Lipschitz constant;
	\item [\textbf{(b)}] $f$ exhibits logarithmic superlinear growth, including the example $f(z)=\lambda z\log|z|^2$ for $z\ne0$, with $f(0)=0$ and $\lambda\in\mathbb{R}$, in the sense of a suitable low-regularity solution concept.
\end{enumerate}

The difficulty in addressing the above two questions lies in closing the nonlinear control argument for arbitrary Lipschitz constants and endpoint norms. Recall that a contraction argument based on fixed unweighted estimates necessitates a small Lipschitz constant. An alternative is to freeze the nonlinearity, which yields bounded time-dependent potentials; this approach requires uniform observability together with a nonlinear fixed-point construction at the $L^2$ regularity level \cite[(3.7)--(3.8)]{Zuazua2003}. Genuinely different from parabolic equations, the unitary Schr\"odinger group provides no Sobolev regularization for $L^2$ solutions, rendering the standard compactness argument inapplicable. In Section~\ref{subsec:related-literature}, we shall compare this approach and existing controllability results with the weighted contraction argument developed herein.

The main contribution of this work is to provide an affirmative answer to the above two questions under a local convex-phase condition. More precisely, our first main result, Theorem~\ref{thm:main}, resolves question \textbf{(a)} on domains satisfying a local convex-phase condition near the uncontrolled region, and on a spatial interval with any nonempty open control region. This result eliminates the smallness restriction inherent in elementary perturbation arguments, allowing for arbitrary $L^2$ terminal states and all time horizons $T>0$. In particular, no differentiability, gauge invariance, or conservation law assumptions are imposed on the nonlinearity $f$. Theorems~\ref{thm:log-main} and~\ref{thm:log-critical-small} answer question \textbf{(b)} on a spatial interval for the logarithmic nonlinearity classes specified hereafter. In particular, these results cover $f(z)=\lambda z\log|z|^2$ for all $\lambda\in\mathbb R$, as well as arbitrary globally Lipschitz perturbations. Analogous to the Lipschitz setting, the control time is independent of the terminal-state norms.  
 
\subsection{Comparison with existing works}\label{subsec:related-literature}

We shall compare our results with closely related works on control and
observability for Schr\"odinger equations, and the comparison mainly concerns
the nonlinearities, endpoint spaces, geometric assumptions and
dependence of the control time on the data.

For the cubic Schr\"odinger equation on a periodic interval,
Illner et~al.~\cite{IllnerLangeTeismann2006} established local exact
controllability in $H^1$. In the case of a bounded interval with
periodic, Dirichlet or Neumann boundary conditions, Rosier and
Zhang~\cite{RosierZhang2009local} obtained local exact controllability
in $L^2$. They also treated polynomial nonlinearities on rectangles
in compatible Sobolev spaces, where the required regularity depends
on the dimension, degree and boundary conditions
\cite{RosierZhang2010}. In these results, the positive control time is prescribed first,
and the endpoints must lie in a sufficiently small neighborhood of zero. Moreover, the rectangular
theory allows any nonempty open internal control region, including
regions that do not satisfy the geometric control condition (GCC).
For the globally Lipschitz class considered here,
Theorem~\ref{thm:main} establishes exact controllability between
arbitrary $L^2$ states without restricting their norms or the
Lipschitz constant.

Global exact controllability has also been obtained by combining
stabilization with local control. On an
interval, Laurent~\cite{Laurent2010interval} proved global exact
controllability of the cubic equation with either sign of the
nonlinearity in $L^2$, using any nonempty open control region.
More precisely, for each $R>0$, a common control time is obtained
for all endpoints with norm at most $R$. For the defocusing equation
on compact surfaces, Dehman et~al.~\cite{DehmanGerardLebeau2006}
established the corresponding $H^1$ theory under GCC and an additional
unique-continuation assumption. In dimension three,
Laurent~\cite{Laurent2010manifolds} proved global exact
controllability for the defocusing cubic equation in $H^1$ under
these assumptions and a spectral bilinear estimate; the result
also covers cuboids by reflection.

Concerning the dependence of the control time on the size of the
data, Le Balc'h and Martin~\cite{LeBalchMartin2025} obtained uniform
exponential energy stabilization for the defocusing cubic equation
on tori of dimension at most three, provided that the control region
contains a strip in each coordinate direction. Their estimate gives
a logarithmic upper bound for the null-control time in terms of
the initial energy. Under GCC, Loyola~\cite{Loyola2025} removed
the additional nonlinear unique-continuation assumption for
gauge-invariant defocusing polynomial models on compact surfaces
and defocusing cubic models on the specified three-dimensional
manifolds, obtaining semiglobal $H^1$ exact controllability.
Thus, for each bounded set of endpoints, these results provide a
common sufficiently large control time. In Theorem~\ref{thm:main}, the time $T>0$ is prescribed
independently of the endpoint norms for every globally Lipschitz
map $f:\C\to\C$, including $f(0)\ne0$, without requiring
differentiability or a conservation law.

In addition to the nonlinear assumptions and control time, the
geometric conditions in these results differ from those used here.
The above manifold results rely on GCC, which concerns the geodesic
flow, whereas Assumption~\ref{ass:phase-geometry} requires a smooth
real phase with a nonvanishing gradient and a positive definite
Hessian near $K_\omega=\overline{\Om\setminus\omega}^{\,\overline\Om}$,
with $\partial_\nu\psi=0$ on the relevant boundary neighborhood.
Thus the phase is required to be convex and noncritical only near
the uncontrolled set. Under the corresponding covariant condition,
Proposition~\ref{prop:riemannian-control} gives exact control between
arbitrary $L^2$ states on compact Riemannian manifolds in every
prescribed positive time. On an interval,
Theorem~\ref{thm:main}(ii) allows every nonempty open control region
without a phase assumption.

To pass from linear to nonlinear control by freezing the
nonlinearity, one also needs to track the dependence of observability
estimates on the potential. Bourgain
et~al.~\cite{BourgainBurqZworski2013} proved observability and
additive controllability on two-dimensional tori for a fixed real,
time-independent $L^2$ potential, from any nonempty open set in
every positive time. On $\R$, Su et~al.~\cite{SuSunYuan2025}
obtained quantitative observability for real, time-independent
bounded potentials on thick sets, with a bound of the form
$Ce^{C/T^2}$, where the constants are controlled by the potential
norm and thickness parameters. These estimates concern fixed spatial potentials. After separating
$f(0)$, however, freezing a globally Lipschitz nonlinearity gives a
family of bounded time-dependent potentials that need not be real.
An estimate for a fixed potential therefore does not by itself give
the uniform observability required for this family.

For this approach, Zuazua~\cite[(3.7)--(3.8)]{Zuazua2003}
identified two requirements: uniform observability and compactness of
the nonlinear fixed-point map. The compactness issue arises because the
unitary Schr\"odinger group provides no Sobolev regularization of
$L^2$ data, and hence the compactness argument based on parabolic
smoothing does not apply. To address the nonlinear step, Zuazua
also suggested the global inverse function approach used
by Lasiecka and Triggiani~\cite{LasieckaTriggiani1991}.
For the observability step, Lasiecka
et~al.~\cite{LasieckaTriggianiZhang2004a,LasieckaTriggianiZhang2004b}
established Carleman estimates for nonconservative Schr\"odinger
equations at the $H^1$ and $L^2$ levels, and Triggiani and
Xu~\cite{TriggianiXu2007} developed an $H^1$ theory on Riemannian
manifolds. In the present work, a parameter-uniform zeroth-order
Carleman estimate allows us to decrease the source-to-state gain
by choosing the parameter sufficiently large. This makes the
nonlinear correction contractive in weighted $L^2$ spaces for every
finite Lipschitz constant and yields the control cost
$Ce^{C(T^{-1}+L^{2/3})}$ in
\eqref{eq:unified-lipschitz-cost}, without requiring compactness
of the fixed-point map.

For logarithmic nonlinearities, Beauchard et~al.~\cite{BCP2026}
studied $f(z)=\lambda z\log|z|^2$ for arbitrary real $\lambda$
on tori and Euclidean spaces. Using finitely many real bilinear
potentials, they established global $L^2$ approximate
controllability between arbitrary states on the unit sphere,
up to a constant phase, in arbitrarily short times.
Theorems~\ref{thm:log-main} and~\ref{thm:log-critical-small}
concern localized additive controls on an interval, where any
nonempty open control region is allowed and arbitrary $L^2$
endpoints can be joined exactly in every prescribed time $T>0$.
Moreover, these theorems apply to
$f(z)=z\phi(|z|)+\mathfrak g(z)$, where
$\phi\in W^{1,\infty}_{\mathrm{loc}}((0,\infty);\R)$ and
$\mathfrak g$ is any globally Lipschitz map, with the radial term
extended by zero at the origin. Under
$r|\phi'(r)|\leq D(1+|\log r|)^\gamma$, every finite $D$ is
allowed when $0\leq\gamma<1/2$. At the critical exponent,
the condition $r|\phi'(r)|\leq D_0+D_1\sqrt{1+|\log r|}$
requires only $D_1$ to be sufficiently small, while $D_0$ and
the Lipschitz perturbation remain unrestricted. In particular,
the canonical logarithmic term is covered for every real
coefficient, and the controlled solutions belong to
$C([0,T];L^2)\cap L^4(Q)$ with $f(y)\in L^2(Q)$.
The estimates \eqref{eq:log-main-cost} and
\eqref{eq:critical-main-cost} quantify the dependence of the
control cost on time and endpoint size; for the canonical term,
\eqref{eq:canonical-log-fast-cost} gives a linear dependence
on the endpoint norms.

\FloatBarrier
\Needspace{7\baselineskip}

\subsection{Main results}
Let us first state the assumptions on the nonlinearity and the control
geometry for \eqref{eq:main}.  
\begin{assumption}\label{ass:lip}
	The function $f:\C\to\C$ is globally Lipschitz: for some finite
	$L>0$,
	\begin{equation}\label{eq:lip}
		|f(z)-f(w)|\leq L|z-w|
		\qquad\text{for all }z,w\in\C.
	\end{equation} 
\end{assumption}

\begin{assumption}\label{ass:phase-geometry}
	Let $K_\omega=\overline{\Om\setminus\omega}^{\,\overline\Om}$.
	There exist a relatively open neighborhood $\mathcal O$ of
	$K_\omega$ in $\overline\Om$, a real-valued function $\psi$
	smooth on a neighborhood of $\overline\Om$, and constants
	$c_0,c_1>0$ such that
	\begin{equation}\label{eq:phase-geometry}
		|\nabla\psi|\geq c_0,\quad D^2\psi\geq c_1I
		\quad\text{on }\mathcal O,
		\qquad
		\partial_\nu\psi=0
		\quad\text{on }\mathcal O\cap\partial\Om.
	\end{equation}
	Here $\nu$ is the outward unit normal to $\partial\Om$, and
	the Hessian inequality is understood in the sense of quadratic forms.
\end{assumption}

For the Euclidean Dirichlet problem, set
\[
 A=-\Delta_D\quad\text{on }H:=L^2(\Om;\C),\qquad
 D(A)=H^2(\Om;\C)\cap H^1_0(\Om;\C),\qquad U(t)=e^{itA}.
\]
The boundary condition is encoded in $A$, and the Schr\"odinger
operator is given by
$
 P=i\partial_t-\Delta=i\partial_t+A.
$
A mild solution of \eqref{eq:main} is a function
$y\in C([0,T];H)$ satisfying
\begin{equation}\label{eq:mild}
 y(t)=U(t)y_0-i\int_0^tU(t-r)
             \bigl(f(y(r))+\one_\omega h(r)\bigr)\,\dd r.
\end{equation}
We use the same symbol $f$ for its Nemytskii map on $H$, defined by
$f(v)(x)=f(v(x))$. Since $\Om$ has finite measure, \eqref{eq:lip}
implies that this map is globally Lipschitz on $H$ and that
\[
 \|f(v)\|_H\leq L\|v\|_H+|\Om|^{1/2}|f(0)|.
\]
It follows that, for every $y_0\in H$ and
$h\in L^2(Q_\omega;\C)$, equation \eqref{eq:main} has a unique mild solution
$y\in C([0,T];H)$.

\begin{definition} \label{def}
The system \eqref{eq:main} is said to be \textit{globally exactly controllable} in $H$ if, for any $T>0$ and all endpoints $y_0,y_1\in H$, there exists a control $h\in L^2(Q_\omega;\mathbb{C})$ such that the corresponding mild solution $y$ of \eqref{eq:main} satisfies $y(T)=y_1$.

\end{definition}

We now state the first main result for globally Lipschitz nonlinearities.

\begin{theorem}\label{thm:main}
Let $f$ satisfy Assumption~\ref{ass:lip}, and assume that one of
the following holds:
\begin{enumerate}
 \item $\Om$ and $\omega$ satisfy
       Assumption~\ref{ass:phase-geometry};
 \item $\Om=(0,\ell)$ for some $\ell>0$, and
       $\omega\subset(0,\ell)$ is any nonempty open set.
\end{enumerate}
Then \eqref{eq:main} is globally exactly controllable in the sense
of Definition~\ref{def}. 
Moreover, the control may be chosen so that
\begin{equation}\label{eq:unified-lipschitz-cost}
 \|h\|_{L^2(Q_\omega)}+\|y\|_{C([0,T];H)}
 \leq C e^{C(T^{-1}+L^{2/3})}
       \bigl(\|y_0\|_H+\|y_1\|_H
             +T|\Om|^{1/2}|f(0)|\bigr).
\end{equation}
The constant $C$ depends only on the fixed spatial geometry,
the control region, and the phase data. It is
independent of $T$, $L$, $f$ and the endpoints.  
\end{theorem}

\begin{example}\label{ex:control-geometries}
Theorem~\ref{thm:main} applies to the following control geometries;
see Figure~\ref{fig:basic-control-geometries}.
\begin{enumerate}
\item $\Om=(0,\ell)$, $\ell>0$, and $\omega$ is any nonempty
open subset of $\Om$.

\item Let $\Om\subset\R^2$ be a bounded simply connected domain
with $C^\infty$ boundary, and let $\Gamma_c\subset\partial\Om$
be a nonempty open arc with
$\overline{\Gamma_c}\ne\partial\Om$.
Set $d(x)=\operatorname{dist}(x,\partial\Om)$, and let $\pi$
be the nearest-point projection in a tubular neighborhood
of $\partial\Om$. For sufficiently small $\delta>0$,
with $2\delta$ below the tubular radius, define
\begin{equation}\label{eq:planar-control-region}
 \omega=
 \{x\in\Om:\delta<d(x)<2\delta\}
 \cup
 \{x\in\Om:0<d(x)<2\delta,\ \pi(x)\in\Gamma_c\}.
\end{equation}
This region satisfies Assumption~\ref{ass:phase-geometry};
the verification is given in Appendix~\ref{app:planar-phase}.
\end{enumerate}
\end{example}

\begin{figure}[!htbp]
\centering
\begin{minipage}[c]{.46\textwidth}
\centering
\begin{tikzpicture}[x=.8cm,y=1.05cm,font=\footnotesize]
 \path[use as bounding box] (-.45,-2.05) rectangle (5.25,2.05);
 \fill[controlfill] (1.60,-.14) rectangle (3.20,.14);
 \draw[ink,line width=.8pt] (0,0)--(4.80,0);
 \draw[ink,line width=.8pt] (0,-.12)--(0,.12) (4.80,-.12)--(4.80,.12);
 \draw[controledge,line width=3pt] (1.83,0)--(2.96,0);
 \node[below=6pt] at (0,0) {$0$};
 \node[below=6pt] at (4.80,0) {$\ell$};
 \node[above=9pt,controledge] at (2.40,0) {$J\Subset\omega$};
 \node[align=center] at (2.40,1.30) {Open control region};
 \node[align=center] at (2.40,-.85) {Unobserved endpoints};
 \node[below] at (2.40,-1.74) {(a) An interval};
\end{tikzpicture}
\end{minipage}\hfill
\begin{minipage}[c]{.48\textwidth}
\centering
% A smooth, noncircular boundary with exact inward normal offsets.
% This is a schematic: the permitted tubular width is fixed by the proof.
\begin{tikzpicture}[x=1cm,y=1cm,font=\footnotesize]
 % Fixed coordinates of the original curves; no nested PGF functions.
 \path[use as bounding box] (-2.28,-2.05) rectangle (2.28,2.05);
 \fill[controlfill,even odd rule] plot coordinates {
    (1.606953,0.014099) (1.608712,0.057311) (1.607206,0.100546) (1.602440,0.143701) (1.594426,0.186670)
    (1.583189,0.229342) (1.568771,0.271596) (1.551228,0.313305) (1.530642,0.354336) (1.507120,0.394554)
    (1.480791,0.433830) (1.451812,0.472043) (1.420362,0.509086) (1.386635,0.544872) (1.350843,0.579337)
    (1.313205,0.612440) (1.273944,0.644167) (1.233279,0.674532) (1.191423,0.703573) (1.148580,0.731354)
    (1.104937,0.757959) (1.060662,0.783492) (1.015908,0.808071) (0.970805,0.831826) (0.925463,0.854891)
    (0.879975,0.877403) (0.834417,0.899493) (0.788847,0.921284) (0.743309,0.942884) (0.697834,0.964382)
    (0.652440,0.985843) (0.607131,1.007307) (0.561895,1.028789) (0.516707,1.050275) (0.471524,1.071727)
    (0.426292,1.093082) (0.380941,1.114258) (0.335397,1.135153) (0.289576,1.155648) (0.243399,1.175614)
    (0.196789,1.194904) (0.149683,1.213365) (0.102031,1.230833) (0.053801,1.247138) (0.004985,1.262107)
    (-0.044402,1.275563) (-0.094321,1.287332) (-0.144706,1.297244) (-0.195466,1.305138) (-0.246485,1.310864)
    (-0.297626,1.314284) (-0.348733,1.315281) (-0.399633,1.313755) (-0.450140,1.309630) (-0.500060,1.302853)
    (-0.549196,1.293396) (-0.597347,1.281254) (-0.644322,1.266448) (-0.689932,1.249021) (-0.734004,1.229037)
    (-0.776377,1.206579) (-0.816910,1.181748) (-0.855477,1.154657) (-0.891975,1.125431) (-0.926318,1.094203)
    (-0.958443,1.061113) (-0.988307,1.026307) (-1.015885,0.989933) (-1.041174,0.952139) (-1.064191,0.913076)
    (-1.084971,0.872892) (-1.103568,0.831733) (-1.120060,0.789739) (-1.134538,0.747044) (-1.147117,0.703773)
    (-1.157927,0.660038) (-1.167115,0.615937) (-1.174844,0.571553) (-1.181288,0.526953) (-1.186633,0.482186)
    (-1.191070,0.437283) (-1.194794,0.392263) (-1.198002,0.347131) (-1.200885,0.301886) (-1.203623,0.256524)
    (-1.206385,0.211042) (-1.209313,0.165442) (-1.212526,0.119731) (-1.216104,0.073924) (-1.220092,0.028038)
    (-1.224493,-0.017911) (-1.229268,-0.063912) (-1.234340,-0.109965) (-1.239597,-0.156083) (-1.244899,-0.202286)
    (-1.250082,-0.248603) (-1.254964,-0.295068) (-1.259350,-0.341709) (-1.263035,-0.388544) (-1.265809,-0.435579)
    (-1.267460,-0.482800) (-1.267777,-0.530172) (-1.266555,-0.577638) (-1.263598,-0.625117) (-1.258720,-0.672507)
    (-1.251752,-0.719684) (-1.242539,-0.766508) (-1.230950,-0.812824) (-1.216871,-0.858465) (-1.200211,-0.903255)
    (-1.180900,-0.947014) (-1.158892,-0.989558) (-1.134161,-1.030701) (-1.106703,-1.070258) (-1.076536,-1.108047)
    (-1.043700,-1.143887) (-1.008256,-1.177601) (-0.970288,-1.209018) (-0.929906,-1.237972) (-0.887240,-1.264309)
    (-0.842448,-1.287885) (-0.795709,-1.308573) (-0.747225,-1.326264) (-0.697217,-1.340872) (-0.645922,-1.352334)
    (-0.593590,-1.360619) (-0.540479,-1.365724) (-0.486848,-1.367677) (-0.432953,-1.366538) (-0.379043,-1.362396)
    (-0.325352,-1.355373) (-0.272096,-1.345612) (-0.219467,-1.333285) (-0.167632,-1.318581) (-0.116725,-1.301705)
    (-0.066850,-1.282874) (-0.018075,-1.262313) (0.029569,-1.240248) (0.076083,-1.216905) (0.121505,-1.192503)
    (0.165905,-1.167253) (0.209384,-1.141357) (0.252068,-1.115003) (0.294105,-1.088366) (0.335654,-1.061610)
    (0.376885,-1.034884) (0.417960,-1.008329) (0.459034,-0.982070) (0.500238,-0.956220) (0.541676,-0.930874)
    (0.583420,-0.906102) (0.625509,-0.881942) (0.667955,-0.858397) (0.710746,-0.835429) (0.753854,-0.812958)
    (0.797240,-0.790866) (0.840858,-0.769004) (0.884655,-0.747198) (0.928568,-0.725258) (0.972524,-0.702987)
    (1.016429,-0.680190) (1.060173,-0.656678) (1.103621,-0.632274) (1.146614,-0.606820) (1.188973,-0.580181)
    (1.230495,-0.552243) (1.270963,-0.522925) (1.310145,-0.492173) (1.347800,-0.459963) (1.383689,-0.426306)
    (1.417574,-0.391238) (1.449229,-0.354826) (1.478442,-0.317161) (1.505025,-0.278355) (1.528813,-0.238534)
    (1.549667,-0.197833) (1.567478,-0.156394) (1.582162,-0.114354) (1.593660,-0.071844) (1.601931,-0.028988)
    (1.606953,0.014099)
  } --cycle
  plot coordinates {
    (1.427506,0.028199) (1.428713,0.057838) (1.427679,0.087509) (1.424390,0.117277) (1.418808,0.147192)
    (1.410882,0.177283) (1.400551,0.207549) (1.387758,0.237954) (1.372463,0.268433) (1.354643,0.298894)
    (1.334307,0.329223) (1.311494,0.359301) (1.286272,0.389002) (1.258740,0.418212) (1.229019,0.446827)
    (1.197250,0.474765) (1.163588,0.501964) (1.128192,0.528392) (1.091226,0.554039) (1.052848,0.578922)
    (1.013212,0.603083) (0.972461,0.626582) (0.930730,0.649500) (0.888144,0.671929) (0.844818,0.693968)
    (0.800864,0.715720) (0.756391,0.737283) (0.711506,0.758747) (0.666317,0.780181) (0.620931,0.801637)
    (0.575455,0.823137) (0.529988,0.844676) (0.484620,0.866220) (0.439423,0.887711) (0.394450,0.909063)
    (0.349727,0.930178) (0.305262,0.950940) (0.261039,0.971229) (0.217024,0.990918) (0.173173,1.009878)
    (0.129439,1.027979) (0.085773,1.045093) (0.042135,1.061090) (-0.001502,1.075844) (-0.045151,1.089230)
    (-0.088805,1.101125) (-0.132435,1.111413) (-0.175994,1.119984) (-0.219413,1.126738) (-0.262601,1.131587)
    (-0.305454,1.134454) (-0.347852,1.135283) (-0.389663,1.134031) (-0.430751,1.130677) (-0.470977,1.125218)
    (-0.510203,1.117670) (-0.548299,1.108066) (-0.585143,1.096454) (-0.620630,1.082897) (-0.654667,1.067464)
    (-0.687181,1.050234) (-0.718115,1.031284) (-0.747430,1.010693) (-0.775102,0.988535) (-0.801118,0.964878)
    (-0.825480,0.939784) (-0.848194,0.913309) (-0.869275,0.885503) (-0.888740,0.856410) (-0.906613,0.826074)
    (-0.922921,0.794534) (-0.937695,0.761832) (-0.950976,0.728007) (-0.962811,0.693101) (-0.973260,0.657150)
    (-0.982393,0.620190) (-0.990295,0.582252) (-0.997065,0.543363) (-1.002818,0.503540) (-1.007680,0.462797)
    (-1.011795,0.421144) (-1.015314,0.378590) (-1.018401,0.335151) (-1.021223,0.290855) (-1.023946,0.245749)
    (-1.026730,0.199906) (-1.029715,0.153430) (-1.033016,0.106453) (-1.036713,0.059136) (-1.040840,0.011652)
    (-1.045387,-0.035821) (-1.050296,-0.083117) (-1.055469,-0.130093) (-1.060774,-0.176636) (-1.066057,-0.222668)
    (-1.071145,-0.268139) (-1.075862,-0.313028) (-1.080029,-0.357326) (-1.083469,-0.401034) (-1.086013,-0.444152)
    (-1.087501,-0.486675) (-1.087784,-0.528583) (-1.086724,-0.569846) (-1.084199,-0.610419) (-1.080102,-0.650244)
    (-1.074343,-0.689253) (-1.066846,-0.727370) (-1.057553,-0.764517) (-1.046421,-0.800614) (-1.033416,-0.835582)
    (-1.018517,-0.869349) (-1.001709,-0.901844) (-0.982981,-0.933003) (-0.962327,-0.962762) (-0.939740,-0.991056)
    (-0.915220,-1.017820) (-0.888769,-1.042980) (-0.860398,-1.066455) (-0.830130,-1.088156) (-0.798001,-1.107988)
    (-0.764070,-1.125846) (-0.728415,-1.141626) (-0.691138,-1.155226) (-0.652365,-1.166549) (-0.612243,-1.175513)
    (-0.570940,-1.182050) (-0.528637,-1.186114) (-0.485529,-1.187682) (-0.441811,-1.186756) (-0.397681,-1.183364)
    (-0.353328,-1.177560) (-0.308928,-1.169421) (-0.264643,-1.159046) (-0.220607,-1.146553) (-0.176933,-1.132073)
    (-0.133701,-1.115748) (-0.090961,-1.097730) (-0.048729,-1.078170) (-0.006989,-1.057220) (0.034310,-1.035032)
    (0.075249,-1.011750) (0.115935,-0.987516) (0.156501,-0.962469) (0.197093,-0.936746) (0.237865,-0.910490)
    (0.278965,-0.883848) (0.320526,-0.856979) (0.362647,-0.830052) (0.405384,-0.803241) (0.448740,-0.776722)
    (0.492665,-0.750655) (0.537055,-0.725175) (0.581767,-0.700373) (0.626637,-0.676289) (0.671491,-0.652907)
    (0.716167,-0.630158) (0.760521,-0.607927) (0.804428,-0.586065) (0.847782,-0.564405) (0.890490,-0.542768)
    (0.932459,-0.520977) (0.973597,-0.498866) (1.013798,-0.476287) (1.052945,-0.453112) (1.090905,-0.429241)
    (1.127530,-0.404601) (1.162662,-0.379152) (1.196133,-0.352884) (1.227777,-0.325820) (1.257432,-0.298013)
    (1.284947,-0.269541) (1.310192,-0.240507) (1.333061,-0.211028) (1.353479,-0.181228) (1.371401,-0.151232)
    (1.386816,-0.121155) (1.399742,-0.091090) (1.410219,-0.061104) (1.418301,-0.031231) (1.424047,-0.001474)
    (1.427506,0.028199)
  } --cycle;
 \fill[controlfill] plot coordinates {
    (1.755497,0.281401) (1.748258,0.304138) (1.740343,0.326684) (1.731764,0.349021) (1.722537,0.371135)
    (1.712675,0.393012) (1.702193,0.414638) (1.691109,0.435999) (1.679440,0.457083) (1.667203,0.477879)
    (1.654417,0.498377) (1.641101,0.518565) (1.627275,0.538436) (1.612958,0.557981) (1.598172,0.577194)
    (1.582937,0.596067) (1.567275,0.614597) (1.551206,0.632777) (1.534752,0.650606) (1.517935,0.668080)
    (1.500777,0.685199) (1.483299,0.701961) (1.465523,0.718366) (1.447470,0.734417) (1.429161,0.750115)
  }
  --plot coordinates {
    (1.197250,0.474765) (1.210727,0.463211) (1.223862,0.451532) (1.236647,0.439734) (1.249070,0.427821)
    (1.261120,0.415799) (1.272788,0.403676) (1.284065,0.391457) (1.294942,0.379151) (1.305411,0.366766)
    (1.315466,0.354311) (1.325099,0.341793) (1.334307,0.329223) (1.343085,0.316610) (1.351428,0.303961)
    (1.359335,0.291286) (1.366803,0.278594) (1.373834,0.265893) (1.380425,0.253190) (1.386580,0.240492)
    (1.392299,0.227807) (1.397585,0.215139) (1.402442,0.202493) (1.406873,0.189874) (1.410882,0.177283)
  } --cycle;
 \draw[ink,line width=.8pt] plot coordinates {
    (1.786400,0.000000) (1.788712,0.056785) (1.786734,0.113582) (1.780490,0.170125) (1.770043,0.226148)
    (1.755497,0.281401) (1.736990,0.335644) (1.714697,0.388656) (1.688822,0.440238) (1.659596,0.490214)
    (1.627275,0.538436) (1.592131,0.584784) (1.554451,0.629169) (1.514530,0.671533) (1.472668,0.711847)
    (1.429161,0.750115) (1.384300,0.786370) (1.338365,0.820672) (1.291621,0.853107) (1.244313,0.883785)
    (1.196662,0.912835) (1.148864,0.940401) (1.101086,0.966642) (1.053466,0.991723) (1.006108,1.015815)
    (0.959086,1.039086) (0.912442,1.061703) (0.866187,1.083822) (0.820301,1.105587) (0.774737,1.127127)
    (0.729426,1.148548) (0.684274,1.169939) (0.639170,1.191358) (0.593990,1.212840) (0.548599,1.234391)
    (0.502856,1.255987) (0.456620,1.277576) (0.409754,1.299076) (0.362128,1.320379) (0.313624,1.341350)
    (0.264140,1.361829) (0.213594,1.381637) (0.161927,1.400575) (0.109105,1.418432) (0.055122,1.434983)
    (0.000000,1.450000) (-0.056208,1.463250) (-0.113418,1.474504) (-0.171519,1.483538) (-0.230368,1.490141)
    (-0.289798,1.494114) (-0.349614,1.495278) (-0.409603,1.493479) (-0.469529,1.488583) (-0.529144,1.480488)
    (-0.588188,1.469122) (-0.646396,1.454442) (-0.703500,1.436442) (-0.759234,1.415145) (-0.813340,1.390609)
    (-0.865574,1.362925) (-0.915705,1.332213) (-0.963525,1.298622) (-1.008848,1.262327) (-1.051518,1.223528)
    (-1.091406,1.182443) (-1.128419,1.139305) (-1.162495,1.094362) (-1.193608,1.047867) (-1.221768,1.000077)
    (-1.247020,0.951249) (-1.269442,0.901633) (-1.289143,0.851470) (-1.306265,0.800988) (-1.320974,0.750396)
    (-1.333461,0.699885) (-1.343935,0.649621) (-1.352623,0.599744) (-1.359759,0.550367) (-1.365585,0.501575)
    (-1.370344,0.453422) (-1.374274,0.405935) (-1.377603,0.359110) (-1.380546,0.312917) (-1.383301,0.267300)
    (-1.386040,0.222178) (-1.388912,0.177454) (-1.392035,0.133008) (-1.395496,0.088711) (-1.399345,0.044424)
    (-1.403600,0.000000) (-1.408241,-0.044706) (-1.413211,-0.089838) (-1.418419,-0.135529) (-1.423741,-0.181903)
    (-1.429019,-0.229068) (-1.434066,-0.277109) (-1.438671,-0.326092) (-1.442601,-0.376054) (-1.445605,-0.427005)
    (-1.447418,-0.478925) (-1.447770,-0.531761) (-1.446386,-0.585430) (-1.442997,-0.639815) (-1.437338,-0.694769)
    (-1.429161,-0.750115) (-1.418233,-0.805646) (-1.404347,-0.861131) (-1.387321,-0.916316) (-1.367005,-0.970929)
    (-1.343282,-1.024680) (-1.316074,-1.077271) (-1.285340,-1.128398) (-1.251079,-1.177755) (-1.213332,-1.225038)
    (-1.172180,-1.269954) (-1.127743,-1.312222) (-1.080179,-1.351580) (-1.029682,-1.387788) (-0.976479,-1.420631)
    (-0.920826,-1.449925) (-0.863003,-1.475521) (-0.803312,-1.497303) (-0.742069,-1.515194) (-0.679601,-1.529155)
    (-0.616241,-1.539188) (-0.552320,-1.545334) (-0.488167,-1.547672) (-0.424095,-1.546320) (-0.360405,-1.541429)
    (-0.297376,-1.533185) (-0.235263,-1.521804) (-0.174292,-1.507524) (-0.114657,-1.490609) (-0.056518,-1.471337)
    (0.000000,-1.450000) (0.054811,-1.426896) (0.107866,-1.402327) (0.159154,-1.376590) (0.208699,-1.349974)
    (0.256562,-1.322757) (0.302833,-1.295199) (0.347636,-1.267538) (0.391117,-1.239987) (0.433444,-1.212730)
    (0.474804,-1.185920) (0.515394,-1.159678) (0.555421,-1.134088) (0.595092,-1.109200) (0.634611,-1.085027)
    (0.674174,-1.061548) (0.713963,-1.038709) (0.754143,-1.016421) (0.794856,-0.994569) (0.836218,-0.973009)
    (0.878313,-0.951575) (0.921195,-0.930082) (0.964881,-0.908331) (1.009354,-0.886111) (1.054558,-0.863207)
    (1.100399,-0.839404) (1.146749,-0.814490) (1.193443,-0.788261) (1.240283,-0.760528) (1.287040,-0.731120)
    (1.333461,-0.699885) (1.379265,-0.666698) (1.424156,-0.631461) (1.467824,-0.594107) (1.509947,-0.554598)
    (1.550202,-0.512934) (1.588266,-0.469144) (1.623824,-0.423294) (1.656572,-0.375482) (1.686225,-0.325835)
    (1.712518,-0.274512) (1.735214,-0.221699) (1.754105,-0.167604) (1.769018,-0.112456) (1.779816,-0.056502)
    (1.786400,0.000000)
  } --cycle;
 \draw[controledge,line width=.45pt] plot coordinates {
    (1.427506,0.028199) (1.428713,0.057838) (1.427679,0.087509) (1.424390,0.117277) (1.418808,0.147192)
    (1.410882,0.177283) (1.400551,0.207549) (1.387758,0.237954) (1.372463,0.268433) (1.354643,0.298894)
    (1.334307,0.329223) (1.311494,0.359301) (1.286272,0.389002) (1.258740,0.418212) (1.229019,0.446827)
    (1.197250,0.474765) (1.163588,0.501964) (1.128192,0.528392) (1.091226,0.554039) (1.052848,0.578922)
    (1.013212,0.603083) (0.972461,0.626582) (0.930730,0.649500) (0.888144,0.671929) (0.844818,0.693968)
    (0.800864,0.715720) (0.756391,0.737283) (0.711506,0.758747) (0.666317,0.780181) (0.620931,0.801637)
    (0.575455,0.823137) (0.529988,0.844676) (0.484620,0.866220) (0.439423,0.887711) (0.394450,0.909063)
    (0.349727,0.930178) (0.305262,0.950940) (0.261039,0.971229) (0.217024,0.990918) (0.173173,1.009878)
    (0.129439,1.027979) (0.085773,1.045093) (0.042135,1.061090) (-0.001502,1.075844) (-0.045151,1.089230)
    (-0.088805,1.101125) (-0.132435,1.111413) (-0.175994,1.119984) (-0.219413,1.126738) (-0.262601,1.131587)
    (-0.305454,1.134454) (-0.347852,1.135283) (-0.389663,1.134031) (-0.430751,1.130677) (-0.470977,1.125218)
    (-0.510203,1.117670) (-0.548299,1.108066) (-0.585143,1.096454) (-0.620630,1.082897) (-0.654667,1.067464)
    (-0.687181,1.050234) (-0.718115,1.031284) (-0.747430,1.010693) (-0.775102,0.988535) (-0.801118,0.964878)
    (-0.825480,0.939784) (-0.848194,0.913309) (-0.869275,0.885503) (-0.888740,0.856410) (-0.906613,0.826074)
    (-0.922921,0.794534) (-0.937695,0.761832) (-0.950976,0.728007) (-0.962811,0.693101) (-0.973260,0.657150)
    (-0.982393,0.620190) (-0.990295,0.582252) (-0.997065,0.543363) (-1.002818,0.503540) (-1.007680,0.462797)
    (-1.011795,0.421144) (-1.015314,0.378590) (-1.018401,0.335151) (-1.021223,0.290855) (-1.023946,0.245749)
    (-1.026730,0.199906) (-1.029715,0.153430) (-1.033016,0.106453) (-1.036713,0.059136) (-1.040840,0.011652)
    (-1.045387,-0.035821) (-1.050296,-0.083117) (-1.055469,-0.130093) (-1.060774,-0.176636) (-1.066057,-0.222668)
    (-1.071145,-0.268139) (-1.075862,-0.313028) (-1.080029,-0.357326) (-1.083469,-0.401034) (-1.086013,-0.444152)
    (-1.087501,-0.486675) (-1.087784,-0.528583) (-1.086724,-0.569846) (-1.084199,-0.610419) (-1.080102,-0.650244)
    (-1.074343,-0.689253) (-1.066846,-0.727370) (-1.057553,-0.764517) (-1.046421,-0.800614) (-1.033416,-0.835582)
    (-1.018517,-0.869349) (-1.001709,-0.901844) (-0.982981,-0.933003) (-0.962327,-0.962762) (-0.939740,-0.991056)
    (-0.915220,-1.017820) (-0.888769,-1.042980) (-0.860398,-1.066455) (-0.830130,-1.088156) (-0.798001,-1.107988)
    (-0.764070,-1.125846) (-0.728415,-1.141626) (-0.691138,-1.155226) (-0.652365,-1.166549) (-0.612243,-1.175513)
    (-0.570940,-1.182050) (-0.528637,-1.186114) (-0.485529,-1.187682) (-0.441811,-1.186756) (-0.397681,-1.183364)
    (-0.353328,-1.177560) (-0.308928,-1.169421) (-0.264643,-1.159046) (-0.220607,-1.146553) (-0.176933,-1.132073)
    (-0.133701,-1.115748) (-0.090961,-1.097730) (-0.048729,-1.078170) (-0.006989,-1.057220) (0.034310,-1.035032)
    (0.075249,-1.011750) (0.115935,-0.987516) (0.156501,-0.962469) (0.197093,-0.936746) (0.237865,-0.910490)
    (0.278965,-0.883848) (0.320526,-0.856979) (0.362647,-0.830052) (0.405384,-0.803241) (0.448740,-0.776722)
    (0.492665,-0.750655) (0.537055,-0.725175) (0.581767,-0.700373) (0.626637,-0.676289) (0.671491,-0.652907)
    (0.716167,-0.630158) (0.760521,-0.607927) (0.804428,-0.586065) (0.847782,-0.564405) (0.890490,-0.542768)
    (0.932459,-0.520977) (0.973597,-0.498866) (1.013798,-0.476287) (1.052945,-0.453112) (1.090905,-0.429241)
    (1.127530,-0.404601) (1.162662,-0.379152) (1.196133,-0.352884) (1.227777,-0.325820) (1.257432,-0.298013)
    (1.284947,-0.269541) (1.310192,-0.240507) (1.333061,-0.211028) (1.353479,-0.181228) (1.371401,-0.151232)
    (1.386816,-0.121155) (1.399742,-0.091090) (1.410219,-0.061104) (1.418301,-0.031231) (1.424047,-0.001474)
    (1.427506,0.028199)
  } --cycle;
 \draw[controledge,line width=.45pt] plot coordinates {
    (1.313205,0.612440) (1.273944,0.644167) (1.233279,0.674532) (1.191423,0.703573) (1.148580,0.731354)
    (1.104937,0.757959) (1.060662,0.783492) (1.015908,0.808071) (0.970805,0.831826) (0.925463,0.854891)
    (0.879975,0.877403) (0.834417,0.899493) (0.788847,0.921284) (0.743309,0.942884) (0.697834,0.964382)
    (0.652440,0.985843) (0.607131,1.007307) (0.561895,1.028789) (0.516707,1.050275) (0.471524,1.071727)
    (0.426292,1.093082) (0.380941,1.114258) (0.335397,1.135153) (0.289576,1.155648) (0.243399,1.175614)
    (0.196789,1.194904) (0.149683,1.213365) (0.102031,1.230833) (0.053801,1.247138) (0.004985,1.262107)
    (-0.044402,1.275563) (-0.094321,1.287332) (-0.144706,1.297244) (-0.195466,1.305138) (-0.246485,1.310864)
    (-0.297626,1.314284) (-0.348733,1.315281) (-0.399633,1.313755) (-0.450140,1.309630) (-0.500060,1.302853)
    (-0.549196,1.293396) (-0.597347,1.281254) (-0.644322,1.266448) (-0.689932,1.249021) (-0.734004,1.229037)
    (-0.776377,1.206579) (-0.816910,1.181748) (-0.855477,1.154657) (-0.891975,1.125431) (-0.926318,1.094203)
    (-0.958443,1.061113) (-0.988307,1.026307) (-1.015885,0.989933) (-1.041174,0.952139) (-1.064191,0.913076)
    (-1.084971,0.872892) (-1.103568,0.831733) (-1.120060,0.789739) (-1.134538,0.747044) (-1.147117,0.703773)
    (-1.157927,0.660038) (-1.167115,0.615937) (-1.174844,0.571553) (-1.181288,0.526953) (-1.186633,0.482186)
    (-1.191070,0.437283) (-1.194794,0.392263) (-1.198002,0.347131) (-1.200885,0.301886) (-1.203623,0.256524)
    (-1.206385,0.211042) (-1.209313,0.165442) (-1.212526,0.119731) (-1.216104,0.073924) (-1.220092,0.028038)
    (-1.224493,-0.017911) (-1.229268,-0.063912) (-1.234340,-0.109965) (-1.239597,-0.156083) (-1.244899,-0.202286)
    (-1.250082,-0.248603) (-1.254964,-0.295068) (-1.259350,-0.341709) (-1.263035,-0.388544) (-1.265809,-0.435579)
    (-1.267460,-0.482800) (-1.267777,-0.530172) (-1.266555,-0.577638) (-1.263598,-0.625117) (-1.258720,-0.672507)
    (-1.251752,-0.719684) (-1.242539,-0.766508) (-1.230950,-0.812824) (-1.216871,-0.858465) (-1.200211,-0.903255)
    (-1.180900,-0.947014) (-1.158892,-0.989558) (-1.134161,-1.030701) (-1.106703,-1.070258) (-1.076536,-1.108047)
    (-1.043700,-1.143887) (-1.008256,-1.177601) (-0.970288,-1.209018) (-0.929906,-1.237972) (-0.887240,-1.264309)
    (-0.842448,-1.287885) (-0.795709,-1.308573) (-0.747225,-1.326264) (-0.697217,-1.340872) (-0.645922,-1.352334)
    (-0.593590,-1.360619) (-0.540479,-1.365724) (-0.486848,-1.367677) (-0.432953,-1.366538) (-0.379043,-1.362396)
    (-0.325352,-1.355373) (-0.272096,-1.345612) (-0.219467,-1.333285) (-0.167632,-1.318581) (-0.116725,-1.301705)
    (-0.066850,-1.282874) (-0.018075,-1.262313) (0.029569,-1.240248) (0.076083,-1.216905) (0.121505,-1.192503)
    (0.165905,-1.167253) (0.209384,-1.141357) (0.252068,-1.115003) (0.294105,-1.088366) (0.335654,-1.061610)
    (0.376885,-1.034884) (0.417960,-1.008329) (0.459034,-0.982070) (0.500238,-0.956220) (0.541676,-0.930874)
    (0.583420,-0.906102) (0.625509,-0.881942) (0.667955,-0.858397) (0.710746,-0.835429) (0.753854,-0.812958)
    (0.797240,-0.790866) (0.840858,-0.769004) (0.884655,-0.747198) (0.928568,-0.725258) (0.972524,-0.702987)
    (1.016429,-0.680190) (1.060173,-0.656678) (1.103621,-0.632274) (1.146614,-0.606820) (1.188973,-0.580181)
    (1.230495,-0.552243) (1.270963,-0.522925) (1.310145,-0.492173) (1.347800,-0.459963) (1.383689,-0.426306)
    (1.417574,-0.391238) (1.449229,-0.354826) (1.478442,-0.317161) (1.505025,-0.278355) (1.528813,-0.238534)
    (1.549667,-0.197833) (1.567478,-0.156394) (1.582162,-0.114354) (1.593660,-0.071844) (1.601931,-0.028988)
    (1.606953,0.014099) (1.608712,0.057311) (1.607206,0.100546) (1.602440,0.143701) (1.594426,0.186670)
    (1.583189,0.229342)
  };
 \draw[controledge,line width=.45pt] (1.755497,0.281401)--(1.583189,0.229342)
  (1.429161,0.750115)--(1.313205,0.612440);
 \draw[phaseorange,line width=1.7pt] plot coordinates {
    (1.755497,0.281401) (1.748258,0.304138) (1.740343,0.326684) (1.731764,0.349021) (1.722537,0.371135)
    (1.712675,0.393012) (1.702193,0.414638) (1.691109,0.435999) (1.679440,0.457083) (1.667203,0.477879)
    (1.654417,0.498377) (1.641101,0.518565) (1.627275,0.538436) (1.612958,0.557981) (1.598172,0.577194)
    (1.582937,0.596067) (1.567275,0.614597) (1.551206,0.632777) (1.534752,0.650606) (1.517935,0.668080)
    (1.500777,0.685199) (1.483299,0.701961) (1.465523,0.718366) (1.447470,0.734417) (1.429161,0.750115)
  };
 \node at (-.05,-.10) {$K_0$};
 \node at (-1.75,1.37) {$K_1$};
 \draw[ink!65,line width=.45pt] (-1.54,1.26)--(-1.195848,0.843718);
 \node[controledge] at (-1.68,-1.39) {$\omega$};
 \draw[controledge,line width=.55pt] (-1.47,-1.28)--(-1.080300,-0.945701);
 \node[phaseorange] at (1.84,1.23) {$\Gamma_c$};
 \draw[phaseorange,line width=.6pt] (1.74,1.07)--(1.627275,0.538436);
 \node[below] at (0,-1.78) {(b) A planar domain};
\end{tikzpicture}
\end{minipage}
\caption{Schematic control geometries in
(a) one dimension and (b) two dimensions.
The control regions are shown in blue.
The sets $K_0,K_1$ and the planar phase construction
are given in Appendix~\ref{app:planar-phase}.}
\label{fig:basic-control-geometries}
\end{figure}
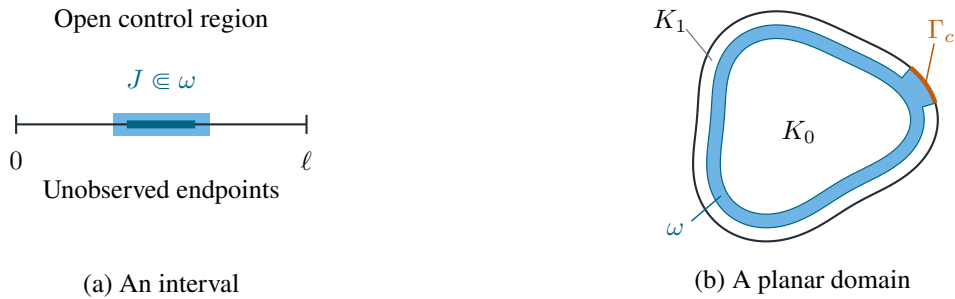

On an interval, we also consider nonlinearities that need not be
locally Lipschitz at zero and may grow faster than linearly.
We make the following assumption.

\begin{assumption}
\label{ass:radial-log-class}\label{eq:radial-log-class}
There exist constants $0\leq\gamma<1/2$, $0\leq D<\infty$, a real-valued function $\phi\in W^{1,\infty}_{\mathrm{loc}}((0,\infty))$, and a globally Lipschitz mapping $\mathfrak g:\mathbb{C}\to\mathbb{C}$ such that
\[
f(z)=z\phi(|z|)+\mathfrak g(z)\quad(z\ne0),\qquad
f(0)=\mathfrak g(0),
\]
and
\begin{equation}\label{eq:log-growth-assumption}
	r|\phi'(r)|\leq D(1+|\log r|)^\gamma
	\quad\text{for almost every }r>0.
\end{equation}
We use the locally Lipschitz representative of $\phi$.
The value $\mathfrak g(0)\in\mathbb{C}$ is arbitrary, and no sign
condition is imposed on the imaginary part of $\mathfrak g(z)\bar z$.

\end{assumption}

\begin{remark}
(1) The growth assumption gives
\[
|\phi(r)|\leq|\phi(1)|+C D(1+|\log r|)^{1+\gamma},
\]
so the radial term extends continuously by zero at the origin.
When $\gamma=0$, the map $q\mapsto\phi(e^q)$ is globally Lipschitz.
The reality assumption concerns only the radial coefficient;
the full nonlinearity may be complex-valued.

(2)  Assumption~\ref{ass:radial-log-class} applies to the canonical logarithmic nonlinearity as well as to
\[
\lambda z\log(1+|z|^2),\quad
\lambda z\bigl(\log(2+|z|^2)\bigr)^\alpha,
\quad \lambda\in\mathbb{R},~ 1<\alpha<3/2,
\]
and to their arbitrary globally Lipschitz perturbations.

\end{remark}
The second main result concerns this logarithmic class.

\begin{theorem}\label{thm:log-main}
	Let $\Omega=(0,\ell)$ with $\ell>0$, and let $\omega\subset\Omega$ be nonempty and open. Suppose that Assumption~\ref{ass:radial-log-class} holds. For every $T>0$ and $y_0,y_1\in H$, there exists a control $h\in L^2(Q_\omega;\mathbb{C})$ such that \eqref{eq:main} admits a Dirichlet mild solution in the sense of \eqref{eq:mild}, satisfying
	\[
	\begin{gathered}
		y\in C([0,T];H)\cap L^4(Q),\qquad f(y)\in L^2(Q),\\
		y(0)=y_0,\qquad y(T)=y_1.
	\end{gathered}
	\]
	For this control and initial state, the solution is unique within $C([0,T];H)\cap L^4(Q)$. When $\gamma=0$, uniqueness also holds among mild solutions in $C([0,T];H)$ with $f(y)\in L^1(0,T;H)$.
	
	Set $d=\|y_0\|_H+\|y_1\|_H$. The following estimates hold:
	
	\smallskip
	\textup{(i)} Let $q=2(1+\gamma)/3<1$. The control may be selected such that
	\begin{equation}\label{eq:log-main-cost}
		\|h\|_{L^2(Q_\omega)}+\|y\|_{C([0,T];H)}
		\leq C e^{C/T}\left((1+d)
		\exp\!\left(C[\log(2+d)]^q\right)+\sqrt T\right).
	\end{equation}
	Consequently, for any $\varepsilon>0$,
	\begin{equation}\label{eq:log-almost-linear-cost}
		\|h\|_{L^2(Q_\omega)}+\|y\|_{C([0,T];H)}
		\leq C_\varepsilon e^{C/T}
		\bigl((1+d)^{1+\varepsilon}+\sqrt T\bigr).
	\end{equation}
	Here $C$ depends only on $\Omega$, $\omega$, $\gamma$, $D$, $|\phi(1)|$, $\operatorname{Lip}(\mathfrak g)$ and $|\mathfrak g(0)|$; $C_\varepsilon$ additionally depends on $\varepsilon$. In particular, $C$ is independent of $\varepsilon$, and both constants are independent of $T,y_0,y_1$.
	
	\smallskip
	\textup{(ii)} If $f(z)=\lambda z\log|z|^2$ for $z\ne0$, with $f(0)=0$ and $\lambda\in\mathbb{R}$, then the control can instead be chosen so that
	\begin{equation}\label{eq:canonical-log-fast-cost}
		\|h\|_{L^2(Q_\omega)}+\|y\|_{C([0,T];H)}
		\leq C_\lambda e^{C_\lambda/T}d,
	\end{equation}
	where $C_\lambda$ depends only on $\Omega$, $\omega$ and $\lambda$.
\end{theorem}

At the   exponent $\gamma=1/2$, the argument requires
smallness only of the leading coefficient.

\begin{theorem}\label{thm:log-critical-small}
Let $\Om=(0,\ell)$, $\ell>0$, and let $\omega\subset\Om$ be
nonempty and open. There are constants $D_*,c_*>0$, depending
only on $\Om$ and $\omega$, with the following property.
Suppose that
\[
 f(z)=z\phi(|z|)+\mathfrak g(z)\quad(z\ne0),\qquad
 f(0)=\mathfrak g(0),
\]
where $\phi\in W^{1,\infty}_{\mathrm{loc}}((0,\infty);\R)$
is taken in its locally Lipschitz representative and
$\mathfrak g:\C\to\C$ is globally Lipschitz. Assume that
\begin{equation}\label{eq:critical-log-two-coefficients}
 r|\phi'(r)|\leq D_0+D_1\sqrt{1+|\log r|}
 \quad\text{for a.e. }r>0,
 \qquad D_0\geq0,\quad 0\leq D_1\leq D_*.
\end{equation}
For every $T>0$ and $y_0,y_1\in H$, there is a control
$h\in L^2(Q_\omega;\C)$ such that \eqref{eq:main} admits a Dirichlet
mild solution in the sense of \eqref{eq:mild}, with
\[
 \begin{gathered}
 y\in C([0,T];H)\cap L^4(Q),\qquad f(y)\in L^2(Q),\\
 y(0)=y_0,\qquad y(T)=y_1.
 \end{gathered}
\]
For this control and initial state, the solution is unique in
$C([0,T];H)\cap L^4(Q)$. Writing
$d=\|y_0\|_H+\|y_1\|_H$, the control can be chosen, for every
$\varepsilon>0$, so that
\begin{equation}\label{eq:critical-main-cost}
 \|h\|_{L^2(Q_\omega)}+\|y\|_{C([0,T];H)}
 \leq C_\varepsilon e^{C/T}
       \left((1+d)^{1+c_*D_1^{2/3}+\varepsilon}+\sqrt T\right).
\end{equation}
Here $C$ depends only on $\Om$, $\omega$, $D_0$, $D_1$,
$|\phi(1)|$, $\operatorname{Lip}(\mathfrak g)$ and
$|\mathfrak g(0)|$; $C_\varepsilon$ may also depend on
$\varepsilon$. In particular, $C$ is independent of $\varepsilon$,
and both constants are independent of $T,y_0,y_1$.
No smallness is required of $D_0$ or $\mathfrak g$.
\end{theorem}

\begin{remark}\label{cor:log-asymptotic-critical}
Under the notation and regularity assumptions of
Theorem~\ref{thm:log-critical-small}, its conclusions remain valid
if \eqref{eq:critical-log-two-coefficients} is replaced by
\[
 \max\left\{
 \operatorname*{ess\,limsup}_{r\downarrow0}
       \frac{r|\phi'(r)|}{\sqrt{1+|\log r|}},\;
 \operatorname*{ess\,limsup}_{r\to\infty}
       \frac{r|\phi'(r)|}{\sqrt{1+|\log r|}}
 \right\}<D_*.
\]
Indeed, choose $D_1<D_*$ above both limits; local Lipschitz
regularity allows the contribution from the remaining compact
amplitude range to be absorbed into a finite $D_0$ in
\eqref{eq:critical-log-two-coefficients}.
\end{remark}

\subsection{Extension to Riemannian manifolds}\label{subsec:riemannian-extension}

The weighted control argument extends to compact Riemannian
manifolds under a covariant version of the local phase condition.
The result again allows arbitrary $L^2$ endpoints in every
prescribed positive time, without restricting the Lipschitz constant.
Riemannian Carleman estimates for Schr\"odinger equations were developed
by Triggiani and Xu~\cite{TriggianiXu2007} at the $H^1$ level.
In our setting, the control acts in the interior, and the phase
condition uses the covariant Hessian and tangency on the unobserved boundary.

\begin{proposition}\label{prop:riemannian-control}
Let $(M,g)$ be a smooth compact connected Riemannian manifold, either with smooth
boundary or without boundary, and let $\omega\subset M^\circ$ be
nonempty and open, where $M^\circ$ denotes the interior. Suppose that
there are a relatively open neighborhood
$\mathcal O$ of $K_\omega=\overline{M^\circ\setminus\omega}^{\,M}$, a real
$\psi\in C^\infty(M)$, and $c_0,c_1>0$ such that
\begin{equation}\label{eq:riemannian-phase}
 \begin{gathered}
 |\nabla_g\psi|_g\geq c_0,\qquad
 \operatorname{Hess}_g\psi\geq c_1g\quad\text{on }\mathcal O,\\
 \partial_{\nu_g}\psi=0\quad\text{on }\mathcal O\cap\partial M.
 \end{gathered}
\end{equation}
The last condition is omitted if $\partial M=\varnothing$.
If $f$ satisfies Assumption~\ref{ass:lip}, then for every $T>0$
and $y_0,y_1\in L^2(M;\C)$, the equation
\[
 i\partial_t y=\Delta_g y+f(y)+\one_\omega h
\]
has a control $h\in L^2((0,T)\times\omega;\C)$ whose unique mild
solution satisfies $y\in C([0,T];L^2(M;\C))$, $y(0)=y_0$ and
$y(T)=y_1$. Homogeneous Dirichlet conditions are imposed when
$\partial M\ne\varnothing$. Moreover,
\begin{equation}\label{eq:riemannian-control-cost}
 \begin{aligned}
 &\|h\|_{L^2((0,T)\times\omega)}
       +\|y\|_{C([0,T];L^2(M))}\\
 &\qquad\leq C e^{C(T^{-1}+L^{2/3})}
 \bigl(\|y_0\|_{L^2(M)}+\|y_1\|_{L^2(M)}
          +T\operatorname{Vol}_g(M)^{1/2}|f(0)|\bigr).
 \end{aligned}
\end{equation}
All spatial norms use the Riemannian volume measure. The constant
$C$ depends only on the fixed manifold, metric, control region and
phase data; it is independent of $T,L,f,y_0,y_1$.
\end{proposition}

The proof is presented in Section~\ref{subsec:riemannian-proof}.
The local phase condition can be illustrated by an interior
geodesic ball $B_g(p,R)$, chosen sufficiently small that
$r(x)^2=d_g(p,x)^2$ is smooth and $\operatorname{Hess}_g(r^2)\geq g$
on the ball. If $K_\omega$ lies in $\{a\leq r\leq b\}$ with
$0<a<b<R$, a smooth extension of $r^2$ near $K_\omega$
satisfies \eqref{eq:riemannian-phase}.

\begin{corollary}\label{cor:one-dimensional-manifolds}
Let $(M,g)$ be a smooth, compact, connected one-dimensional
Riemannian manifold, possibly with boundary, and let
$\omega\subset M^\circ$ be nonempty and open. The conclusions of
Theorems~\ref{thm:main}, \ref{thm:log-main},
and~\ref{thm:log-critical-small} hold for
\[
 i\partial_t y-\Delta_g y=f(y)+\one_\omega h,
 \qquad y(0)=y_0,\quad y(T)=y_1,
\]
with homogeneous Dirichlet conditions when $\partial M\ne\varnothing$.
Here $H=L^2(M;\C)$, all spatial integrals use the Riemannian
volume, and the respective nonlinear assumptions are unchanged.
The dependence of the constants on the interval and control
region is replaced by dependence on $(M,g)$ and $\omega$;
all other stated dependencies remain unchanged.
\end{corollary}

%In one dimension, arclength coordinates identify the metric with
%that of an interval or a circle. The proof in
%Section~\ref{subsec:riemannian-proof} uses this reduction and
%the periodic $L^4$ estimate. In higher dimensions,
%the logarithmic argument still requires a suitable derivative-free
%space--time estimate. General Strichartz estimates on compact manifolds
%involve a loss of derivatives; see Burq et~al.~\cite{BurqGerardTzvetkov2004}
%in the boundaryless case and Blair et~al.~\cite{BlairSmithSogge2012}
%for manifolds with boundary. They do not directly supply
%the bound needed here for arbitrary $L^2$ data and $L^2$ controls.

\subsection{Main idea of the proof}

The proof relies on a linear control operator whose source-to-state
gain decreases as the Carleman parameter increases. For each prescribed $T>0$, this gain can
be made sufficiently small to obtain a contraction for any
finite Lipschitz constant and arbitrary endpoint data.
The logarithmic extension then follows by combining the
resulting control-cost estimate with quantitative bounds on the
truncation error. We describe the construction in three steps.

\medskip 
\emph{1. A zeroth-order Carleman estimate with a free parameter.}
We use the parameter-dependent weighted contraction strategy
developed by Hern\'andez-Santamar\'ia
et~al.~\cite{HernandezSantamariaLeBalchPeralta2023}
for semilinear stochastic parabolic equations, where
their Carleman construction adapts the time weights of
Badra et~al.~\cite{BadraErvedozaGuerrero2016}
for nonhomogeneous incompressible Navier--Stokes equations.
To the best of our knowledge, this strategy appears to be new for
nonlinear Schr\"odinger equations. For Schr\"odinger equations, Lasiecka
et~al.~\cite{LasieckaTriggianiZhang2004a,LasieckaTriggianiZhang2004b}
developed Carleman estimates for observability.
To apply the weighted contraction argument to the present
exact controllability problem, we use weights singular at
both time endpoints and establish a zeroth-order state estimate
with a zeroth-order observation. The nonlinear argument requires both zeroth-order norms;
a gradient estimate alone does not provide them.

To this end, we lower the order of a strong Carleman estimate
by applying a Bessel regularizer with parameter
$\kappa=s/[t(T-t)]$.
The positive term has size $\kappa^3$, while the squared
commutator errors are bounded by $C\kappa^2$.
Under the local phase condition, the transport field is tangent
to the unobserved boundary. The boundary flux therefore vanishes,
and the Dirichlet resolvent commutator can be used.
The interval case is treated separately by reflecting
only the linear test functions.
By tracking the constants in these estimates, we obtain
the threshold $s\geq aT+bT^2$.

\medskip 
\emph{2. Global exact controllability from the parameter gain.}
Using the Carleman estimate and weighted Hilbert-space
duality, we construct a linear solution operator with zero
initial and final states. Its norm from the weighted source
space to the weighted state space is uniform in $s$.
For fixed $T$, the embedding of the state space into the
source space has norm $O(s^{-3/2})$, which provides the
small factor needed for the nonlinear correction.

To accommodate arbitrary endpoints, we first join a forward
uncontrolled trajectory starting from $y_0$ to a backward
uncontrolled trajectory ending at $y_1$.
The resulting residual is supported away from the time endpoints
and hence belongs to the weighted source space,
including when $f(0)\ne0$.
The nonlinear correction is then a contraction provided that
$CLs^{-3/2}<1$, with no restriction on the endpoint norms.
By choosing $s$ quantitatively and subdividing the time
interval into finitely many pieces, we obtain the control-cost
factor $e^{C(T^{-1}+L^{2/3})}$ without using a conservation law.
The same construction gives a uniform control estimate
for the shifted nonlinearity $f(u+z)-f(u)$ along a prescribed
reference function $u$. We use this estimate to control the
shifted logarithmic truncations in the next step.

\medskip 
\emph{3. The logarithmic extension by Lipschitz truncation.}
We first clamp $\log|z|$ to $[-K,K]$ to obtain a globally
Lipschitz nonlinearity $f_K$ with Lipschitz constant
$O(K^{1+\gamma})$.
The preceding control estimate gives a cost with exponent
$C(\tau^{-1}+K^{2(1+\gamma)/3})$, while the truncation error
satisfies
\[
|f(z)-f_K(z)|
\leq C(1+K)^\gamma e^{-K}(1+|z|^2).
\]
We compare the original and truncated equations through a common
truncation, so the comparison introduces only the factor
$e^{C\tau K^\gamma}$.
The remaining tail is controlled by the one-dimensional
space--time $L^4$ estimate.
Since $2(1+\gamma)/3<1$ when $\gamma<1/2$,
the truncation error decays faster than the control cost grows.

These estimates allow us to construct controls on shrinking
time intervals whose squared norms are summable.
The resulting controls drive the true solution toward an
uncontrolled reference trajectory ending at the desired target.
By tracking the first truncation level, we obtain a control cost
with nearly linear dependence on the data.

At the endpoint $\gamma=1/2$, the leading exponent in the control cost
is linear in $K$, and the argument closes when its coefficient
is sufficiently small.
By keeping the lower-order term separate, we obtain the
two-coefficient critical theorem.
Separating a compact interval of amplitudes from the two
tails then yields the asymptotic condition.
For the canonical logarithmic nonlinearity, amplitude scaling
gives the stronger control-cost bound.

\subsection{Organization of the paper}\label{sec:organization}

The paper is organized as follows.
Section~\ref{sec:estimates} establishes the parameter-dependent
Carleman estimates and their Riemannian extension.
Section~\ref{sec:controllability} constructs the weighted linear
control operator and proves global exact controllability for
globally Lipschitz nonlinearities, together with the quantitative
control-cost estimates.
Section~\ref{sec:logarithmic} proves the logarithmic results
in the subcritical and critical regimes.
The planar phase condition for the control geometry in
Example~\ref{ex:control-geometries} is verified in
Appendix~\ref{app:planar-phase}.

\FloatBarrier

\section{Parameter-dependent Carleman estimates}\label{sec:estimates}

\subsection{Weights and regularization}\label{sec:weights}

All complex scalar products are linear in the first variable.
Unless otherwise specified, the spatial $L^2$ norms in
full-space estimates are taken over $\R^n$. The constants may depend
on the fixed domain, control region, cutoffs, spatial phase, and
$T$. They are independent of $s$ and $\kappa\geq1$. Uniformity
for all $T>0$ is established in
Proposition~\ref{prop:uniform-short-time-carleman}.

For the full-space estimates below and the collar estimate in
Section~\ref{sec:carleman}, we choose the following quadratic phase
independently of the local phase in \eqref{eq:phase-geometry}.
Choose a bounded open set $U$ with $\overline\Om\subset U$, a point $x_0\notin\overline U$, and a constant $M$ such that
\begin{equation}\label{eq:psi}
 \psi(x)=|x-x_0|^2-M<-1\quad(x\in\overline U).
\end{equation}
There are constants $\alpha,\beta,c_*>0$ such that
\begin{equation*}
 1<\alpha\le-\psi\le\beta,\qquad
 |\nabla\psi|\ge c_*,\qquad D^2\psi=2I
 \quad\text{on }\overline U.
\end{equation*}
Define
\begin{equation*}
 a(t)=\frac1{t(T-t)},\qquad
 \kappa(t)=sa(t),\qquad
 \theta(t,x)=e^{\kappa(t)\psi(x)}.
\end{equation*}
The dependence of $\theta$ on $s$ is left implicit.

\begin{lemma}\label{lem:timeweights}
For $s>0$ and $0<t<T$,
\begin{equation}\label{eq:kappabounds}
 \kappa\ge\frac{4s}{T^2},\qquad
 |\kappa'|\le\frac{T}{s}\kappa^2,\qquad
 0<\kappa''\le\frac{2T^2}{s^2}\kappa^3.
\end{equation}
In particular,
\begin{equation*}
 |\kappa''|+\frac{|\kappa'|^2}{\kappa}
 \le\frac{3T^2}{s^2}\kappa^3.
\end{equation*}
For each fixed $s>0$, the product of any polynomial in $\kappa$ with $\theta$ tends to zero at both temporal endpoints, uniformly on $\overline U$.
\end{lemma}
\begin{proof}
Put $d=t(T-t)$. Direct differentiation gives
\[
 \kappa'=-s(T-2t)d^{-2},\qquad
 \kappa''=2s d^{-2}+2s(T-2t)^2d^{-3}.
\]
It follows from the identity $d+(T-2t)^2=T^2-3t(T-t)\le T^2$ that
\eqref{eq:kappabounds} holds. Finally, $\theta\le e^{-\alpha\kappa}$
shows that exponential decay dominates every polynomial in $\kappa$
at the temporal endpoints.
\end{proof}

\begin{lemma}\label{lem:bessel}
For $\kappa\ge1$, let
\[
 S_\kappa=\kappa(\kappa^2-\Delta_{\R^n})^{-1/2}.
\]
Then
\begin{equation}\label{eq:multbounds}
 \norm{S_\kappa}_{2\to2}\le1,\qquad
 \norm{S_\kappa\operatorname{div}}_{L^2(\R^n;\C^n)\to L^2(\R^n)}\le\kappa,
 \qquad
 \norm{\partial_\kappa S_\kappa}_{2\to2}\le C\kappa^{-1}.
\end{equation}
For $b\in C_b^\infty(\R^n)$,
\begin{equation}\label{eq:commbounds}
 \norm{[b,S_\kappa]}_{2\to2}
 \le C\kappa^{-1}\norm{\nabla b}_\infty,
 \qquad
 \norm{[b,S_\kappa]\partial_j}_{2\to2}
 \le C\norm{\nabla b}_\infty.
\end{equation}
Moreover, for every $v\in L^2(\R^n)$,
\begin{equation}\label{eq:energyidentity}
 \kappa\norm{\nabla S_\kappa v}_2^2
 +\kappa^3\norm{S_\kappa v}_2^2
 =\kappa^3\norm v_2^2.
\end{equation}
The commutators are initially defined on $C_c^\infty(\R^n)$
and extended by continuity. The constants depend only on the
dimension, apart from the displayed coefficient norm.
\end{lemma}
\begin{proof}
The multiplier of $S_\kappa$ is
\[
 m_\kappa(\xi)=\frac{\kappa}{(\kappa^2+|\xi|^2)^{1/2}}.
\]
The bounds in \eqref{eq:multbounds} and the identity \eqref{eq:energyidentity} follow from Plancherel's theorem. For the parameter derivative, we use
\[
 \partial_\kappa m_\kappa(\xi)
 =\frac{|\xi|^2}{(\kappa^2+|\xi|^2)^{3/2}}.
\]
To prove \eqref{eq:commbounds}, we use the positive convolution kernel
\begin{equation}\label{eq:kernel}
 K_\kappa(z)=\frac{\kappa}{\sqrt\pi}
 \int_0^\infty r^{-1/2}e^{-\kappa^2r}
 (4\pi r)^{-n/2}e^{-|z|^2/(4r)}\dd r.
\end{equation}
The heat-kernel moment estimates give
\begin{equation}\label{eq:moments}
 \int_{\R^n}K_\kappa\dd z=1,\qquad
 \int_{\R^n}|z|K_\kappa(z)\dd z\le C\kappa^{-1},\qquad
 \int_{\R^n}|z||\nabla K_\kappa(z)|\dd z\le C.
\end{equation}
Indeed, the first moment of the heat kernel is $C_n\sqrt r$, whereas
the integral of $|z|$ times the magnitude of its gradient is uniform
in $r$. Substituting these bounds into \eqref{eq:kernel} gives
\eqref{eq:moments}.

The kernel of $[b,S_\kappa]$ is
\[
 (b(x)-b(y))K_\kappa(x-y).
\]
The first commutator bound follows from $|b(x)-b(y)|\le\norm{\nabla b}_\infty|x-y|$ and Young's inequality.
To obtain the second bound, we take $v\in C_c^\infty(\R^n)$ and integrate by parts in $y_j$. The resulting kernel is the sum of
\[
 (\partial_jb)(y)K_\kappa(x-y),\qquad
 (b(x)-b(y))(\partial_jK_\kappa)(x-y).
\]
The moment bounds in \eqref{eq:moments} control both kernels on
$L^2$. We justify the integration by parts by first truncating the
heat-kernel integral away from $r=0$ and then removing the truncation
using the same bounds. Density gives the required extensions.
\end{proof}

\begin{lemma}\label{lem:tails}
Let $K\Subset U$, and let $\zeta\in C_c^\infty(U)$ equal one
on a fixed neighborhood of $K$. For every integer $N\geq1$,
there is $C_N>0$, depending only on $n$, $N$, $K$, and $\zeta$,
such that
\begin{equation}\label{eq:tails}
 \norm{(1-\zeta)S_\kappa v}_{H^1(\R^n)}
 \le C_N\kappa^{-N}\norm v_2
 \quad\bigl(v\in L^2(\R^n),\ \supp v\subset K,\ \kappa\ge1\bigr).
\end{equation}
If $\eta\in C_c^\infty(\R^n)$ has support at a positive
distance from $K$, the same bound holds for
$\|\eta S_\kappa v\|_2+\|\eta\nabla S_\kappa v\|_2$,
with a constant that may also depend on $\eta$.
\end{lemma}
\begin{proof}
For $d>0$ and $|\gamma|\le1$, the kernel representation \eqref{eq:kernel} implies
\[
 \int_{\{z\in\R^n:\,|z|\ge d\}}|\partial_z^\gamma K_\kappa(z)|\dd z
 \le C_{N,d}\kappa^{-N}.
\]
The heat-kernel tail contributes $e^{-d^2/(8r)}$, and
$\kappa^2r+d^2/(8r)$ is bounded below by a positive multiple of
$\kappa d$. The resulting exponential decay absorbs the polynomial
factors from the remaining integrals. Since $x\notin\{\zeta=1\}$
and $y\in K$ are separated by a fixed positive distance, Young's
inequality gives \eqref{eq:tails}, both before and after one spatial
derivative.
\end{proof}

\subsection{A boundary collar}\label{sec:carleman}

We first consider a control region containing a full boundary collar:
\begin{equation}\label{eq:collar}
 \{x\in\Om:\dist(x,\partial\Om)<\delta\}\subset\omega
 \quad\text{for some }\delta>0.
\end{equation}
This geometry satisfies Assumption~\ref{ass:phase-geometry}.
Indeed, $K_\omega$ has a neighborhood compactly contained in $\Om$, on which
$\psi(x)=|x-x_*|^2$, with $x_*\notin\overline\Om$, is strictly convex
and has no critical point. The boundary condition on the phase is
then vacuous.
We retain the weights of Section~\ref{sec:weights} and
set $P=i\partial_t-\Delta$. All integrals below are taken with respect to Lebesgue measure.
Constants may depend on the fixed sets, cutoffs, weight, and $T$.
They do not depend on $s$ once its stated lower bound is imposed.

\begin{proposition}\label{thm:carleman}
Assume \eqref{eq:collar}. There are $s_0>0$ and $C>0$ such that, for every $s\geq s_0$ and every
$p\in C^\infty([0,T]\times\overline\Omega;\C)$,
\begin{equation}\label{eq:carleman}
 \int_Q\kappa^3\theta^2\abs p^2
 \leq C\left(
   \int_Q\theta^2\abs{Pp}^2+
   \int_{Q_\omega}\kappa^3\theta^2\abs p^2
 \right).
\end{equation}
The test function $p$ is unrestricted at $t=0,T$ and on the spatial
boundary.
\end{proposition}

We first prove an estimate for functions supported away from the
spatial boundary. Direct conjugation gives
\begin{equation*}
 L=\theta P\theta^{-1}
 =i\partial_t-\Delta+2\kappa\nabla\psi\cdot\nabla
  +\kappa\Delta\psi-\kappa^2\abs{\nabla\psi}^2-i\kappa'\psi.
\end{equation*}

\begin{lemma}\label{lem:strong}
Fix $T>0$ and the phase \eqref{eq:psi}.
There exist $s_0>0$ and $C>0$, depending only on $T$, $U$,
and the fixed phase, such that, for every $s\geq s_0$,
\begin{equation}\label{eq:strong-collar}
 \int_0^T\!\int_U
       \bigl(\kappa\abs{\nabla w}^2+\kappa^3\abs w^2\bigr)
       \,\dd x\,\dd t
 \leq C\int_0^T\!\int_U\abs{Lw}^2\,\dd x\,\dd t
\end{equation}
for every $w\in C_c^\infty((0,T)\times U;\C)$.
The same estimate holds for smooth functions with a fixed compact
spatial support in $U$ that are flat at both temporal endpoints.
Here flatness means smooth extension to $[0,T]\times U$ with
every temporal derivative equal to zero at $t=0,T$.
\end{lemma}

\begin{proof}
Put $g=\nabla\psi$ and decompose $L=A+B$, where
\[
 A=i\partial_t-\Delta-\kappa^2\abs g^2,\qquad
 B=2\kappa g\cdot\nabla+\kappa\Delta\psi-i\kappa'\psi.
\]
On the stated test class, integration by parts gives
$A^*=A$ and $B^*=-B$. Thus
\[
 \norm{Lw}_{L^2}^2
 =\norm{Aw}_{L^2}^2+\norm{Bw}_{L^2}^2
   +2\Real\int_{(0,T)\times U}Aw\,\overline{Bw}\,\dd x\,\dd t.
\]
The $L^2$ norms in this proof are taken over $(0,T)\times U$.
A direct calculation of the commutator gives
\[
 [A,B]
 =-8\kappa\Delta+8\kappa^3\abs g^2+\kappa''\psi
   +4i\kappa'g\cdot\nabla+2i\kappa'\Delta\psi .
\]
The spatial terms follow from
\[
 [-\Delta,\,2\kappa g\cdot\nabla+\kappa\Delta\psi]
 =-8\kappa\Delta,
 \qquad
 [-\kappa^2\abs g^2,\,2\kappa g\cdot\nabla]
 =8\kappa^3\abs g^2,
\]
because $D^2\psi=2I$ and $\Delta\psi=2n$. The remaining terms follow
from
\[
 [i\partial_t,\,2\kappa g\cdot\nabla+\kappa\Delta\psi]
 =2i\kappa'g\cdot\nabla+i\kappa'\Delta\psi,
\]
\[
 [-\Delta,\,-i\kappa'\psi]
 =2i\kappa'g\cdot\nabla+i\kappa'\Delta\psi,
 \qquad
 [i\partial_t,\,-i\kappa'\psi]=\kappa''\psi .
\]
Taking the real quadratic form of the commutator gives
\begin{equation}\label{eq:strong-cross-collar}
\begin{aligned}
 2\Real\int_{(0,T)\times U}Aw\,\overline{Bw}\,\dd x\,\dd t
 ={}&8\int_{(0,T)\times U}\kappa\abs{\nabla w}^2\,\dd x\,\dd t\\
    &+8\int_{(0,T)\times U}\kappa^3\abs g^2\abs w^2\,\dd x\,\dd t
     +\int_{(0,T)\times U}\kappa''\psi\abs w^2\,\dd x\,\dd t\\
    &-4\int_{(0,T)\times U}
       \kappa'\Imag(g\cdot\nabla w\,\overline w)\,\dd x\,\dd t.
\end{aligned}
\end{equation}
The support conditions eliminate the boundary terms, and
$2i\kappa'\Delta\psi$ contributes nothing to the real quadratic form.

Let
\[
 d=\inf_U\abs g>0,\qquad
 G=\sup_U\abs g,\qquad H=\sup_U\abs\psi.
\]
Young's inequality gives
\[
 4\abs{\kappa'}\,G\abs{\nabla w}\abs w
 \leq2\kappa\abs{\nabla w}^2
       +2G^2\frac{\abs{\kappa'}^2}{\kappa}\abs w^2.
\]
Lemma~\ref{lem:timeweights} implies
\[
 \frac{\abs{\kappa'}^2}{\kappa}
 \leq\frac{T^2}{s^2}\kappa^3,\qquad
 \abs{\kappa''}\leq\frac{2T^2}{s^2}\kappa^3.
\]
Combining these bounds, we find that the right-hand side of \eqref{eq:strong-cross-collar} is bounded below by
\[
 6\int_{(0,T)\times U}\kappa\abs{\nabla w}^2\,\dd x\,\dd t+
 \left(8d^2-\frac{2T^2(G^2+H)}{s^2}\right)
       \int_{(0,T)\times U}\kappa^3\abs w^2\,\dd x\,\dd t.
\]
We choose $s$ so that the coefficient in parentheses is at least
$4d^2$. The decomposition of $\norm{Lw}_{L^2}^2$ then gives
\eqref{eq:strong-collar}.

To extend the estimate to flat functions, we multiply $w$ by temporal
cutoffs equal to one on $[2\varepsilon,T-2\varepsilon]$ and zero
near the endpoints, with first derivatives of order
$O(\varepsilon^{-1})$. Flatness makes the commutator with
$i\partial_t$ tend to zero in $L^2$ and controls all polynomial
factors in $\kappa$, $\kappa'$, and $\kappa''$ above. The estimate
therefore passes to the limit.
\end{proof}

To obtain a zeroth-order observation term, we regularize the test
function before applying the strong estimate. The resulting source
norm controls the spatial cutoff errors without differentiating the
observed state. For this purpose, we use the spatial operator $S_\kappa$ from
Lemma~\ref{lem:bessel}.

\begin{lemma}\label{lem:lower}
Fix $T>0$, a compact set $K\subset U$, and
$\chi_2\in C_c^\infty(U;[0,1])$ equal to one on a neighborhood
of $K$. The spatial coefficients of $L$ admit fixed smooth bounded
extensions to $\R^n$, independent of $s$ and $t$.
Let $\widetilde L_s$ denote the resulting differential operator on
$(0,T)\times\R^n$, chosen to agree with $L$ near
$(0,T)\times\supp\chi_2$. There exist $s_0>0$ and $C>0$,
depending only on the fixed spatial data and $T$, such that,
for every $s\geq s_0$, the following estimate holds, with the
subscript $s$ on the operators suppressed:
\begin{equation}\label{eq:lower-collar}
 \int_0^T\kappa^3\norm{v(t)}_{L^2(\R^n)}^2\,\dd t
 \leq C\int_0^T
       \norm{S_\kappa\widetilde L v(t)}_{L^2(\R^n)}^2\,\dd t
\end{equation}
whenever $v\in C_c^\infty((0,T)\times\R^n;\C)$ has spatial support
in $K$. The estimate also holds when $v$ is smooth with that spatial
support and flat at $t=0,T$.
\end{lemma}

\begin{proof}
We choose fixed functions
$a_1,\ldots,a_n,b,c,d\in C_b^\infty(\R^n)$ which agree, respectively,
with
\[
 \partial_1\psi,\ldots,\partial_n\psi,\quad
 \Delta\psi,\quad \abs{\nabla\psi}^2,\quad \psi
\]
on a neighborhood of $\supp\chi_2$. These extensions are independent
of $s$ and $t$. Set
\[
 \widetilde L=i\partial_t-\Delta
      +2\kappa a\cdot\nabla+\kappa b-\kappa^2c-i\kappa'd.
\]
All spatial norms of the fixed coefficients are finite.
We may increase the lower bound on $s$ so that
$\kappa(t)\geq1$ throughout $(0,T)$.

For each fixed $t$, set $z=S_\kappa v$ and $w=\chi_2z$. We use the energy
\[
 E_\kappa(z)=\kappa\norm{\nabla z}_{L^2(\R^n)}^2
                   +\kappa^3\norm z_{L^2(\R^n)}^2.
\]
Lemma~\ref{lem:bessel} gives the exact identity
\begin{equation}\label{eq:energy-identity-collar}
 E_\kappa(S_\kappa v)=\kappa^3\norm v_{L^2(\R^n)}^2.
\end{equation}
Let $r=(1-\chi_2)z$ and $E=\supp\nabla\chi_2$. The distance from $K$
to both $\supp(1-\chi_2)$ and $E$ is positive. By
Lemma~\ref{lem:tails}, for every integer $N\geq1$,
\begin{equation}\label{eq:tails-collar}
\begin{aligned}
 &\norm r_{L^2}+\norm{\nabla r}_{L^2}
   +\norm{1_Ez}_{L^2}+\norm{1_E\nabla z}_{L^2}\\
 &\hspace{35mm}\leq C_N\kappa^{-N}\norm v_{L^2}.
\end{aligned}
\end{equation}
The fixed separation makes these constants independent of $t$ and
$s$. The product rule also controls the terms in which a derivative falls
on $1-\chi_2$. Consequently,
\[
 E_\kappa(r)\leq C_N\kappa^{3-2N}\norm v_{L^2}^2.
\]
Since $z=w+r$, \eqref{eq:energy-identity-collar} implies
\[
 \kappa^3\norm v_{L^2}^2
 \leq2E_\kappa(w)+C_N\kappa^{3-2N}\norm v_{L^2}^2.
\]
With $N=2$ fixed, increasing $\inf_t\kappa(t)$ absorbs the last term.
In particular,
\begin{equation}\label{eq:energy-comparison-collar}
 \frac14\kappa^3\norm v_{L^2}^2\leq E_\kappa(w).
\end{equation}

We next estimate the conjugated equation. Since $\widetilde L=L$ near
$\supp\chi_2$, we have
\begin{equation}\label{eq:lower-expansion-collar}
 Lw=\chi_2S_\kappa\widetilde Lv
       +\chi_2[\widetilde L,S_\kappa]v
       +[\widetilde L,\chi_2]S_\kappa v.
\end{equation}
The first commutator is
\[
\begin{aligned}
 [\widetilde L,S_\kappa]v={}&
   i(\partial_tS_\kappa)v
   +2\kappa\sum_{j=1}^n[a_j,S_\kappa]\partial_jv
   +\kappa[b,S_\kappa]v\\
   &-\kappa^2[c,S_\kappa]v-i\kappa'[d,S_\kappa]v.
\end{aligned}
\]
Applying Lemma~\ref{lem:bessel} to these terms gives
\[
 \norm{[\widetilde L,S_\kappa]v}_{L^2}
 \leq C\left(\kappa+1+\frac{\abs{\kappa'}}{\kappa}\right)
                  \norm v_{L^2}.
\]
Here the time derivative contributes
$\norm{\partial_tS_\kappa}_{L^2\to L^2}
\leq C\abs{\kappa'}/\kappa$.
By Lemma~\ref{lem:timeweights},
$\abs{\kappa'}/\kappa\leq(T/s)\kappa$. Hence, once $s\geq T$,
\begin{equation}\label{eq:commutator-collar}
 \norm{[\widetilde L,S_\kappa]v}_{L^2}
 \leq C\kappa\norm v_{L^2}.
\end{equation}
Thus the term arising from the singular time coefficient has the
same order as the spatial commutator terms.

The second commutator has the explicit form
\[
 [\widetilde L,\chi_2]z
 =-2\nabla\chi_2\cdot\nabla z-(\Delta\chi_2)z
       +2\kappa(a\cdot\nabla\chi_2)z.
\]
Its coefficients are supported away from $K$.
The same off-diagonal estimate as in \eqref{eq:tails-collar}
(with $\supp\Delta\chi_2$ included, if necessary) gives
\begin{equation}\label{eq:cutoff-tail-collar}
 \norm{[\widetilde L,\chi_2]S_\kappa v}_{L^2}
 \leq C_N\kappa^{1-N}\norm v_{L^2}.
\end{equation}
By applying Lemma~\ref{lem:strong} to $w$ and combining the energy
comparison \eqref{eq:energy-comparison-collar} with
\eqref{eq:lower-expansion-collar}--\eqref{eq:cutoff-tail-collar}, we obtain
\[
 \int_0^T\kappa^3\norm v_{L^2}^2\,\dd t
 \leq C_1\int_0^T\norm{S_\kappa\widetilde Lv}_{L^2}^2\,\dd t
       +C_2\int_0^T\kappa^2\norm v_{L^2}^2\,\dd t.
\]
The constants $C_1,C_2$ are independent of $s$. Since
$\inf_t\kappa(t)=4s/T^2$, we can absorb the last term by choosing $s$
so that $\inf_t\kappa(t)\geq2C_2$. We obtain
\eqref{eq:lower-collar}.

Finally, let $v$ be flat at both endpoints. We apply the result to
$v_\varepsilon=\rho_\varepsilon v$, with the temporal cutoffs used
in Lemma~\ref{lem:strong}. Since $S_\kappa$ acts only in space,
\[
 S_\kappa\widetilde L v_\varepsilon
 =\rho_\varepsilon S_\kappa\widetilde Lv
       +i\rho_\varepsilon'S_\kappa v.
\]
The last term tends to zero in $L^2$ by
$\norm{S_\kappa}_{L^2\to L^2}\leq1$ and flatness of $v$.
All remaining terms converge in the displayed norms, which proves
the extension to flat functions. No bound on $\sup_t\kappa$ is needed.
\end{proof}

\begin{proof}[Proof of Proposition~\ref{thm:carleman}]
The collar assumption allows us to choose a cutoff $\chi\in C_c^\infty(\Omega)$
such that
\[
 0\leq\chi\leq1,\qquad \chi=1\ \hbox{on }\Omega\setminus\omega,
 \qquad
 \supp\nabla\chi\cup\supp\Delta\chi\subset\omega.
\]
Choose $\chi_2$ as in Lemma~\ref{lem:lower}, with
$K=\supp\chi$. Extend
\[
 v=\theta\chi p
\]
by zero from $\Omega$ to $\R^n$. The compact support of $\chi$
makes this extension smooth in space. Since $\psi<-1$ and $p$ is
smooth up to $t=0,T$, the extension is also flat at both temporal
endpoints: every derivative is bounded by a polynomial in $\kappa$
times a decaying exponential. Thus Lemma~\ref{lem:lower} applies.

On the support of $v$, the extended operator agrees with the
conjugated operator. Hence
\[
 \widetilde Lv
 =\theta P(\chi p)
 =\theta\chi Pp+\theta[P,\chi]p.
\]
The first term satisfies, at every $t$,
\[
 \norm{S_\kappa(\theta\chi Pp)}_{L^2(\R^n)}
 \leq\norm{\theta Pp}_{L^2(\Omega)}.
\]
To estimate the cutoff term, we write its derivative in divergence
form and apply the regularizer. With $q=\theta p$ in $\Omega$,
\begin{equation}\label{eq:source-divergence-collar}
 \theta[P,\chi]p
 =-2\operatorname{div}(q\nabla\chi)
   +\bigl(\Delta\chi+2\kappa\nabla\psi\cdot\nabla\chi\bigr)q.
\end{equation}
The products on the right have compact support in $\omega$ and
are extended by zero to $\R^n$. By Lemma~\ref{lem:bessel},
\[
 \norm{S_\kappa\operatorname{div}(q\nabla\chi)}_{L^2(\R^n)}
 \leq C\kappa\norm q_{L^2(\omega)}.
\]
The multiplication term in \eqref{eq:source-divergence-collar}
satisfies the same bound, since
$\norm{S_\kappa}\leq1$ and $\kappa\geq1$. Thus
\[
 \norm{S_\kappa\widetilde Lv}_{L^2(\R^n)}^2
 \leq C\left(
      \norm{\theta Pp}_{L^2(\Omega)}^2
          +\kappa^2\norm{\theta p}_{L^2(\omega)}^2
       \right).
\]
Inserting this bound into Lemma~\ref{lem:lower}, we obtain
\[
 \int_Q\kappa^3\theta^2\chi^2\abs p^2
 \leq C\left(
        \int_Q\theta^2\abs{Pp}^2+
        \int_{Q_\omega}\kappa^2\theta^2\abs p^2
       \right).
\]
Finally, $\chi=1$ outside $\omega$ implies pointwise that
$\abs p^2\leq\chi^2\abs p^2+1_\omega\abs p^2$.
Adding the resulting observation term and using $\kappa^2\leq\kappa^3$, we obtain
\eqref{eq:carleman}. The argument requires no integration by parts at
$\partial\Omega$.
\end{proof}

\begin{remark}\label{rem:gradient-obstruction}
The zeroth-order left-hand side in \eqref{eq:carleman} cannot in general
be replaced by a left-hand side that also controls
$\int_Q\kappa\theta^2\abs{\nabla p}^2$, while keeping the same
zeroth-order source and observation terms. To see this, fix $s$
and let $e_j$ be normalized Dirichlet eigenfunctions,
$-\Delta e_j=\lambda_je_j$, with $\lambda_j\to\infty$.
Then $p_j(t,x)=e^{i\lambda_jt}e_j(x)$ satisfies $Pp_j=0$.
On $J=[T/4,3T/4]$ the minimum of $\kappa\theta^2$ over
$J\times\overline\Omega$ is a positive number $a_s$.
Therefore
\[
 \int_Q\kappa\theta^2\abs{\nabla p_j}^2
 \geq a_s\abs J\,\lambda_j.
\]
In contrast, if
$b_s=\sup_{Q}\kappa^3\theta^2<\infty$, then
\[
 \int_{Q_\omega}\kappa^3\theta^2\abs{p_j}^2\leq b_sT.
\]
This contradicts such a strengthened estimate as $j\to\infty$.
Lemma~\ref{lem:strong} remains valid, since its test functions have
compact spatial support. The order reduction is therefore essential
to obtaining an observation term without spatial derivatives.
\end{remark}

\subsection{A local boundary phase}\label{sec:geometry}

We next establish the estimate under the Euclidean local phase
condition of Assumption~\ref{ass:phase-geometry}. We use a Dirichlet
regularizer so that the spatial localization can reach the unobserved
boundary. The interval case of Theorem~\ref{thm:main} is treated
separately in Section~\ref{subsec:interval-geometry}.

Assume \eqref{eq:phase-geometry}, and retain the notation
$K_\omega$ and $\mathcal O$ from the introduction.
We subtract a constant from $\psi$ so that $\psi<-1$ on
$\overline\Om$. All supports and neighborhoods in this subsection
are relative to $\overline\Om$.

\begin{proposition}\label{thm:tangent-carleman}
Assume \eqref{eq:phase-geometry}, and fix a constant shift of
$\psi$ such that $\psi<-1$ on $\overline\Om$.
For every $T>0$, set
$\kappa=s/[t(T-t)]$ and $\theta=e^{\kappa\psi}$.
There exist $s_0>0$ and $C>0$, depending only on $T$
and the fixed spatial data, such that
\[
 \int_Q\kappa^3\theta^2|p|^2
 \leq C\left(\int_Q\theta^2|Pp|^2
             +\int_{Q_\omega}\kappa^3\theta^2|p|^2\right)
\]
for every $s\geq s_0$ and every
$p\in C^\infty([0,T]\times\overline\Om;\C)$ satisfying
$p=0$ on $\Sigma$. No condition is imposed on $p(0)$ or $p(T)$.
Here the fixed spatial data comprise the domain, control region,
phase, phase neighborhood, and the constants in
\eqref{eq:phase-geometry}. In particular, $s_0$ and $C$ are
independent of $s$ and $p$.
\end{proposition}

Tangency enters the proof in two places. It cancels the boundary
term in the strong estimate and ensures that the transport field
preserves the Dirichlet boundary condition in the resolvent commutator below.
A favorable boundary sign alone does not ensure the latter
property. Thus the normal condition in
Assumption~\ref{ass:phase-geometry} is a restriction of this
order-reduction argument, not a necessary condition for exact
controllability.

Put $A=-\Delta_D$ on $H=L^2(\Omega)$ and
\[
 R_\mu=(\mu+A)^{-1},\qquad
 S_\kappa=\kappa(\kappa^2+A)^{-1/2}.
\]
In this subsection, $S_\kappa$ denotes this Dirichlet
operator. Its functional-calculus representation is
\begin{equation}\label{eq:dirichlet-bessel-resolvent}
 S_\kappa=\frac\kappa\pi\int_0^\infty
       r^{-1/2}R_{\kappa^2+r}\,\mathrm dr.
\end{equation}
The integral converges in operator norm on $H$.

\begin{lemma}\label{lem:dirichlet-commutators}
Let $b\in C^\infty(\overline\Omega)$, and let
$X=a\cdot\nabla$, where $a\in C^\infty(\overline\Omega;\mathbb R^n)$
satisfies $a\cdot\nu=0$ on $\partial\Omega$.
For every $\kappa\geq1$ and $F\in L^2(\Om;\C^n)$,
with $\operatorname{div}F$ defined in the conjugate-linear dual
of $H^1_0(\Om;\C)$, one has
\begin{align*}
 &\|S_\kappa\|_{H\to H}\leq1,
 \qquad
 \|S_\kappa\operatorname{div}F\|_H\leq\kappa\|F\|_{L^2(\Omega)^n},\\
 &\|[b,S_\kappa]\|_{H\to H}\leq C\kappa^{-1},
 \qquad
 \|[X,S_\kappa]\|_{H\to H}\leq C,\\
 &\|\partial_tS_{\kappa(t)}\|_{H\to H}
       \leq C|\kappa'|/\kappa.
\end{align*}
The commutator with $X$, initially defined on smooth Dirichlet
functions, extends to the displayed bounded operator.
For every $v\in H$, one also has
\begin{equation}\label{eq:dirichlet-bessel-energy}
 \kappa\|\nabla S_\kappa v\|_2^2
       +\kappa^3\|S_\kappa v\|_H^2
 =\kappa^3\|v\|_H^2.
\end{equation}
The constants depend only on the fixed domain and coefficients
$b,a$, and are independent of $\kappa$, $v$, and $F$.
\end{lemma}

\begin{proof}
The contraction and time derivative estimates, together with
\eqref{eq:dirichlet-bessel-energy}, follow from the spectral theorem.
The divergence estimate follows by duality from
\[
 \|\nabla S_\kappa u\|_2^2
 =\sum_j\frac{\kappa^2\lambda_j}{\kappa^2+\lambda_j}
          |(u,e_j)_H|^2
 \leq\kappa^2\|u\|_H^2.
\]
Here the divergence is interpreted in $H^{-1}(\Omega)$, the dual
of $H^1_0(\Omega)$. Thus no boundary flux of $F$ is introduced.

For $\mu\geq1$ the spectral theorem and duality give
\[
 \|R_\mu\|_{H\to H}\leq\mu^{-1},\qquad
 \|R_\mu\|_{H\to H^1_0}
       +\|R_\mu\|_{H^{-1}\to H}\leq C\mu^{-1/2}.
\]
Tangency ensures that $X$ preserves the zero boundary trace.
In particular, for $v\in H$, we have $R_\mu v\in D(A)$ and
$XR_\mu v\in H^1_0$, because tangential derivatives of a Dirichlet
function have zero trace. Interpreting $X:H\to H^{-1}$
distributionally, we obtain the weak resolvent identity
\[
 [X,R_\mu]=R_\mu[A,X]R_\mu,
\]
where
\[
 [A,X]u=-2\sum_{i,j}\partial_i
             \bigl((\partial_i a_j)\partial_j u\bigr)
                  +\sum_j(\Delta a_j)\partial_j u.
\]
In this divergence form, the commutator defines a bounded map
$[A,X]:H^1_0\to H^{-1}$, whose norm depends only on fixed
derivatives of $a$. To verify the resolvent identity, note that
$(\mu+A)XR_\mu v=Xv+[A,X]R_\mu v$ in $H^{-1}$;
applying its Dirichlet resolvent yields the identity. This
argument does not require $X$ to preserve $D(A)$.
Consequently,
\[
 \|[X,R_\mu]\|_{H\to H}\leq C\mu^{-1}.
\]
Similarly,
\[
 [b,R_\mu]=R_\mu[A,b]R_\mu,\qquad
 [A,b]=-2\nabla b\cdot\nabla-\Delta b,
\]
so that
\[
 \|[b,R_\mu]\|_{H\to H}\leq C\mu^{-3/2}.
\]
Integration in \eqref{eq:dirichlet-bessel-resolvent}, using
\[
 \int_0^\infty\frac{r^{-1/2}}{\kappa^2+r}\,\mathrm dr
      =\frac\pi\kappa,
 \qquad
 \int_0^\infty\frac{r^{-1/2}}{(\kappa^2+r)^{3/2}}\,\mathrm dr
      =\frac2{\kappa^2},
\]
gives the two commutator bounds. We first justify interchanging the integral and the unbounded
operator $X$ for $v\in D(A)$.
The truncated resolvent integrals converge in the graph norm
of $D(A)$, because $A$ commutes with $R_\mu$. Thus $X$ can be
applied before passing to the limit. The integrals of the
commutators converge in operator norm on $H$, and the resulting
commutator extends to $H$ by density.
\end{proof}

\begin{lemma}\label{lem:dirichlet-tails}
Let $K,E\subset\overline\Om$ be closed sets with
$\operatorname{dist}(K,E)>0$. For every integer $N\geq1$,
there exists $C_N>0$ such that, for all
$v\in H$ supported in $K$ and all $\kappa\geq1$,
\[
 \|1_E S_\kappa v\|_H
       +\|1_E\nabla S_\kappa v\|_2
 \leq C_N\kappa^{-N}\|v\|_H.
\]
The same estimate holds for $\|\eta S_\kappa v\|_{H^1}$ when
$\eta$ is a fixed smooth cutoff supported away from $K$.
The constants depend only on the domain, the fixed sets, $N$,
and, in the cutoff version, $\eta$.
\end{lemma}

\begin{proof}
Let $u=R_\mu v$, with $\mu\geq1$. A real cutoff
$\eta\in C^\infty(\overline\Omega)$ that vanishes near $K$ is
admissible in the Dirichlet energy identity, even when its support
meets the boundary. Testing $(\mu-\Delta)u=v$ with $\eta^2u$
and taking the real part gives
\begin{equation}\label{eq:dirichlet-caccioppoli}
 \mu\|\eta u\|_H^2+\|\nabla(\eta u)\|_2^2
       =\||\nabla\eta|u\|_H^2.
\end{equation}
We iterate \eqref{eq:dirichlet-caccioppoli} using finitely many nested
cutoff neighborhoods that separate $E$ from $K$. Each step gains a
factor $\mu^{-1/2}$ in the localized $L^2$ norm. Starting from
$\|u\|_H\leq\mu^{-1}\|v\|_H$, iteration gives, for any integer
$m\geq1$,
\[
 \|1_E R_\mu v\|_H\leq C_m\mu^{-1-m/2}\|v\|_H,
 \qquad
 \|1_E\nabla R_\mu v\|_2
       \leq C_m\mu^{-1/2-m/2}\|v\|_H.
\]
To obtain the gradient estimate, we use the gradient term
in \eqref{eq:dirichlet-caccioppoli} at the last step; the preceding $m-1$ steps use
the $L^2$ term. The constants can depend on $m$ and the separation,
which are fixed independently of $\mu$.

We insert these estimates into \eqref{eq:dirichlet-bessel-resolvent}.
The substitution $r=\kappa^2 z$ gives
\[
 \|1_E S_\kappa v\|_H\leq C_m\kappa^{-m}\|v\|_H,
 \qquad
 \|1_E\nabla S_\kappa v\|_2
       \leq C_m\kappa^{1-m}\|v\|_H.
\]
The integrals converge for $m\geq1$. Taking $m=N+1$ proves the
claim. The product rule proves the cutoff version.
\end{proof}

\begin{proof}[Proof of Proposition~\ref{thm:tangent-carleman}]
If $K_\omega=\varnothing$, then $\omega=\Omega$ and the estimate follows
from its observation term. We henceforth assume $K_\omega\ne\varnothing$.
Write $g=\nabla\psi$ and
\[
 L=\theta P\theta^{-1}
   =i\partial_t-\Delta+2\kappa g\cdot\nabla+\kappa\Delta\psi
          -\kappa^2|g|^2-i\kappa'\psi.
\]
We first consider a smooth Dirichlet function $w$ whose spatial
support is a compact subset of $\mathcal O$ and which is
flat at both temporal endpoints. Decompose $L=\mathcal A+\mathcal B$,
where
\[
 \mathcal A=i\partial_t-\Delta-\kappa^2|g|^2,
 \qquad
 \mathcal B=2\kappa g\cdot\nabla+\kappa\Delta\psi-i\kappa'\psi.
\]
Integration by parts gives
\begin{align*}
 2\Real\int_Q\mathcal A w\,\overline{\mathcal B w}
 ={}&4\int_Q\kappa\,D^2\psi(\nabla w,\nabla\overline w)
       +4\int_Q\kappa^3 D^2\psi(g,g)|w|^2\\
 &-\int_Q\kappa\Delta^2\psi\,|w|^2
       +\int_Q\kappa''\psi|w|^2
       -4\int_Q\kappa'\Imag(g\cdot\nabla w\,\overline w)\\
 &-2\int_0^T\!\int_{\partial\Omega}
                  \kappa\partial_\nu\psi\,|\partial_\nu w|^2
                  \,\mathrm d\sigma\,\dd t.
\end{align*}
The boundary integral vanishes because $\partial_\nu\psi=0$ on
$\mathcal O\cap\partial\Omega$ and $w$ vanishes near the rest of
the boundary. Strict convexity on $\mathcal O$ and the lower bound
on $|g|$ make the first two volume terms coercive. More precisely,
Young's inequality and Lemma~\ref{lem:timeweights} give constants
$c,C>0$, depending only on the spatial data, such that
\[
 2\Real(\mathcal Aw,\mathcal Bw)_{L^2(Q)}
 \geq c\int_Q\kappa|\nabla w|^2
 +\int_Q\left(c-C\frac{T^2}{s^2}
                    -C\kappa^{-2}\right)\kappa^3|w|^2.
\]
Here the last term accounts for $-\kappa\Delta^2\psi$.
We therefore take $T/s$ sufficiently small and $\inf_t\kappa(t)$
sufficiently large to obtain
\begin{equation}\label{eq:tangent-strong}
 \int_0^T\left(\kappa\|\nabla w\|_2^2
                        +\kappa^3\|w\|_H^2\right)\,\mathrm dt
       \leq C\|Lw\|_{L^2(Q)}^2.
\end{equation}
We also need the estimate for functions in
$L^2(0,T;D(A))\cap H^1(0,T;H)$ with compact temporal support
and fixed spatial support compactly contained in $\mathcal O$.
We first apply Dirichlet spectral truncation and time smoothing.
We then multiply the approximants by a fixed smooth cutoff supported in
$\mathcal O$ and equal to one near the original support.
The graph norm of $D(A)$ is equivalent to the $H^2$ norm,
so this multiplication preserves both convergences.
Since the coefficients of $L$ are bounded on the temporal support,
the estimate passes to the limit. The estimate also extends to functions in
$C^\infty([0,T];D(A))$ that are flat at both endpoints in the
graph norm. We insert temporal cutoffs and then remove them:
flatness dominates every polynomial in the time weights,
and the error $i\rho_\varepsilon'w$ tends to zero in $L^2(Q)$.

We choose $\chi,\chi_2,\rho\in C^\infty(\overline\Omega)$, taking
values in $[0,1]$, with successively nested supports in $\mathcal O$.
We require that $\chi=1$ on a neighborhood of $K_\omega$, that
$\chi_2=1$ on a neighborhood of $\supp\chi$, and that $\rho=1$
on a neighborhood of $\supp\chi_2$. In particular,
\[
 \chi=1\ \hbox{on }\Omega\setminus\omega,
 \qquad
 (\supp\nabla\chi\cup\supp\Delta\chi)\cap\Omega\subset\omega.
\]
Such cutoffs exist because $K_\omega$ is compact and is contained in the
relatively open set $\mathcal O$. Put
\[
 a=\rho\nabla\psi,
 \qquad
 \widetilde L=i\partial_t-\Delta+2\kappa a\cdot\nabla
             +\kappa\Delta\psi-\kappa^2|\nabla\psi|^2-i\kappa'\psi.
\]
Then $a\cdot\nu=0$ on the entire boundary, since the field is tangent
on $\mathcal O\cap\partial\Omega$ and vanishes on the rest. Moreover,
$\widetilde L=L$ near $\supp\chi_2$.

Let $v\in C^\infty([0,T];D(A))$ be supported spatially in
$K=\supp\chi$ and flat at both temporal endpoints in the graph
norm of $D(A)$. Set $z=S_\kappa v$ and $w=\chi_2z$.
Both functions are smooth and flat at the temporal endpoints
in the graph norm. Indeed, spectral calculus gives
\[
 \|\partial_\kappa^rS_\kappa\|_{D(A)\to D(A)}
       \leq C_r\kappa^{-r}\qquad(r=0,1,2,\ldots).
\]
Time differentiation introduces only polynomial factors in
$\kappa$ and its derivatives, which are dominated by the
flatness of $v$. Multiplication by $\chi_2$ is bounded on $D(A)$.
Lemma~\ref{lem:dirichlet-tails}
and \eqref{eq:dirichlet-bessel-energy} imply
\[
 \tfrac14\kappa^3\|v\|_H^2
       \leq\kappa\|\nabla w\|_2^2+\kappa^3\|w\|_H^2
\]
for sufficiently large $\inf_t\kappa$. In fact the energy of
$(1-\chi_2)z$ is at most $C_N\kappa^{3-2N}\|v\|_H^2$.
We have the identity
\[
 Lw=\chi_2S_\kappa\widetilde Lv
          +\chi_2[\widetilde L,S_\kappa]v
          +[\widetilde L,\chi_2]S_\kappa v.
\]
Lemma~\ref{lem:dirichlet-commutators} gives
\[
 \|[\widetilde L,S_\kappa]v\|_H
       \leq C\left(\kappa+1+\frac{|\kappa'|}{\kappa}\right)\|v\|_H
       \leq C\kappa\|v\|_H,
\]
once $s\geq T$ and $\kappa\geq1$. The tangent transport and the multiplier $-\kappa^2|\nabla\psi|^2$
each contribute $\mathcal O(\kappa)$; the multiplier $-i\kappa'\psi$
contributes $\mathcal O(|\kappa'|/\kappa)$. The cutoff commutator
is supported a positive distance from $K$, so
Lemma~\ref{lem:dirichlet-tails} gives
\[
 \|[\widetilde L,\chi_2]S_\kappa v\|_H
       \leq C_N\kappa^{1-N}\|v\|_H.
\]
By applying \eqref{eq:tangent-strong} to $w$ and absorbing the resulting
$C\int_0^T\kappa^2\|v\|_H^2\,\mathrm dt$, we obtain
\begin{equation}\label{eq:tangent-lowered}
 \int_0^T\kappa^3\|v\|_H^2\,\mathrm dt
       \leq C\int_0^T\|S_\kappa\widetilde Lv\|_H^2\,\mathrm dt.
\end{equation}

Take $v=\theta\chi p$. It has zero Dirichlet trace and is flat at
both temporal endpoints. On its support,
\[
 \widetilde Lv=\theta\chi Pp+\theta[P,\chi]p.
\]
With $q=\theta p$, the last term is
\[
 \theta[P,\chi]p
       =-2\operatorname{div}(q\nabla\chi)
          +(\Delta\chi+2\kappa\nabla\psi\cdot\nabla\chi)q.
\]
The contraction and divergence bounds of
Lemma~\ref{lem:dirichlet-commutators} imply
\[
 \|S_\kappa\widetilde Lv\|_H^2
       \leq C\left(\|\theta Pp\|_H^2
                         +\kappa^2\|\theta p\|_{L^2(\omega)}^2\right).
\]
Inserting this estimate into \eqref{eq:tangent-lowered} and using
$|p|^2\leq\chi^2|p|^2+1_\omega|p|^2$, we obtain the Carleman inequality.
Together with the strict negativity of $\psi$, this inequality supplies
the hypotheses for the weighted control construction in
Section~\ref{sec:linear} and the nonlinear argument in
Section~\ref{sec:nonlinear}, with $\vartheta=\theta$.
\end{proof}

\subsection{One dimension case}
\label{subsec:interval-geometry}

On an interval, reflection allows us to treat any nonempty open
control set. We apply the full-space estimate to two reflected pieces
and fold the resulting weights back onto the interval. We use separate
state and observation weights to avoid comparing exponential weights
across the partition of unity.

\begin{proposition}\label{prop:interval-carleman}
Let $\Omega=(0,\ell)$, where $\ell>0$, and let
$\omega\subset(0,\ell)$ be a nonempty open set. Fix $T>0$ and set
$\kappa=s/[t(T-t)]$. There are positive weights $W_s,V_s$ on $Q$,
constants $\alpha,\beta,C>0$, and $s_0>0$ such that
\begin{equation}\label{eq:interval-weight-bounds}
 \frac14e^{-2\beta\kappa}\leq W_s\leq V_s
 \leq e^{-2\alpha\kappa}
\end{equation}
and
\begin{equation}\label{eq:interval-carleman}
 \int_Q\kappa^3W_s|p|^2
 \leq C\left(
   \int_QW_s|Pp|^2+
   \int_{Q_\omega}\kappa^3V_s|p|^2
 \right)
\end{equation}
for every $s\geq s_0$ and every
$p\in C^\infty([0,T]\times[0,\ell];\C)$ satisfying
$p(t,0)=p(t,\ell)=0$. The constants depend only on $T$, the interval, and a fixed
subinterval compactly contained in $\omega$, together with the
associated cutoffs and phase. They are independent of $s$ and $p$.
No temporal endpoint condition is imposed on $p$.
\end{proposition}

\begin{proof}
We choose $0<a<b<\ell$ such that $[a,b]\subset\omega$.
We take $\chi_-\in C^\infty([0,\ell];[0,1])$ with
\[
 \chi_-=1\quad\hbox{on }[0,a],\qquad
 \chi_-=0\quad\hbox{on }[b,\ell],
\]
and set $\chi_+=1-\chi_-$. The cutoffs localize near opposite endpoints,
and their nonzero derivatives are supported in $[a,b]\subset\omega$.
In particular,
\begin{equation}\label{eq:interval-cutoff-cover}
 \chi_-^2+\chi_+^2\geq\frac12.
\end{equation}

We reflect the two localized functions across the left and right
endpoints, respectively. The reflection maps are $R_-x=-x$ and $R_+x=2\ell-x$.
We define functions $q_-,q_+$ on the real line by
\[
 q_-(x)=
 \begin{cases}
  \chi_-(x)p(x),&0\leq x\leq\ell,\\
  -\chi_-(-x)p(-x),&-\ell\leq x<0,\\
  0,&\text{otherwise},
 \end{cases}
\]
and
\[
 q_+(x)=
 \begin{cases}
  \chi_+(x)p(x),&0\leq x\leq\ell,\\
  -\chi_+(2\ell-x)p(2\ell-x),&\ell<x\leq2\ell,\\
  0,&\text{otherwise}.
 \end{cases}
\]
The time variable is suppressed in these formulas.
The supports of these functions are contained in $[-b,b]$ and
$[a,2\ell-a]$, respectively. Both functions depend smoothly on time with values in $H^2(\R)$.
At the reflecting point, the Dirichlet condition makes the odd
extension continuous, and odd reflection also matches its first
derivative. The distributional second derivative therefore contains
no interface measure. This $H^2$ regularity is sufficient; higher
smoothness across the interface is unnecessary.

We choose a bounded interval $U$ containing $[-\ell,2\ell]$ in its
interior. We then choose $x_0\notin\overline U$ and $M$ so large that
\[
 \psi(x)=|x-x_0|^2-M<-1\quad\hbox{on }\overline U.
\]
Write $\vartheta_0(t,x)=e^{\kappa(t)\psi(x)}$ and
\[
 \alpha=\min_{\overline U}(-\psi)>0,\qquad
 \beta=\max_{\overline U}(-\psi).
\]
We choose a fixed compact set $K'\Subset U$ whose interior contains
both reflected supports, and fix the cutoff and coefficient
extensions in Lemma~\ref{lem:lower} for $K'$.
Set $v_j=\vartheta_0q_j$.
After applying a smooth temporal cutoff supported in
$(0,T)$, we mollify in space. The mollified functions converge in
$H^1_tL^2_x\cap L^2_tH^2_x$. All sufficiently small mollifications
have support in $K'$, so the constant in the lowering estimate is unchanged.
The coefficients of $L$ are bounded on the temporal support;
hence the images under $L$ converge in $L^2$, as do their images
under $S_\kappa L$.
Thus the lowering estimate extends to the present $H^2$ functions.
For fixed $(T,s)$, $v_j$ is flat at both temporal endpoints
in $H^2(\R)$. We can therefore remove the temporal cutoff using
$\|S_\kappa\|_{2\to2}\leq1$ and exponential flatness,
with the same estimate constant. We obtain
\begin{equation}\label{eq:interval-piece-lowered}
 \int_0^T\kappa^3\|\vartheta_0q_j\|_{L^2(\R)}^2\,\dd t
 \leq C\int_0^T
   \|S_\kappa(\vartheta_0Pq_j)\|_{L^2(\R)}^2\,\dd t,
 \qquad j\in\{-,+\}.
\end{equation}

We next estimate the cutoff commutators. Let $\mathcal O_j$ denote
odd reflection of a function on $(0,\ell)$ across the corresponding
endpoint, followed by zero extension. This notation is applied
below only to functions vanishing near the opposite endpoint.
Define $F_j,G_j$ on $(0,\ell)$ by
\[
 F_j=\chi_j'p,\qquad G_j=\chi_j''p.
\]
Extend $F_j$ evenly and $G_j$ oddly across the reflecting endpoint,
and extend them by zero elsewhere. These extensions are smooth because $\chi_j'$ and
$\chi_j''$ are supported away from both endpoints. Direct differentiation gives
\begin{equation*}
 Pq_j=\mathcal O_j(\chi_jPp)-2\partial_xF_j+G_j.
\end{equation*}
On the original interval, the last two terms equal
$-2\chi_j'p_x-\chi_j''p$. On the reflected interval, differentiating
the even extension of $F_j$ gives the sign required by odd
reflection. Thus the identity holds in $L^2$ and in distributions.
Although $\mathcal O_j(\chi_jPp)$ may jump at the reflecting point,
continuity of $q_j$ and its first derivative excludes an interface
measure. No boundary condition on $Pp$ is needed.

Multiplication by the exponential weight yields
\[
 \vartheta_0(-2\partial_xF_j+G_j)
 =-2\partial_x(\vartheta_0F_j)
     +2\kappa\psi'\vartheta_0F_j+\vartheta_0G_j.
\]
Since $\|S_\kappa\|\leq1$ and
$\|S_\kappa\partial_x\|\leq\kappa$, we can estimate the commutator in
terms of $p$, without any derivative of $p$.
We now fold the full-space integrals back onto the interval.
For this purpose, introduce the weights on $(0,\ell)$,
\[
 w_-(t,x)=\vartheta_0(t,x)^2+\vartheta_0(t,-x)^2,\qquad
 w_+(t,x)=\vartheta_0(t,x)^2
                  +\vartheta_0(t,2\ell-x)^2.
\]
All original and reflected arguments lie in $U$. Thus the bounds
determined by $\alpha$ and $\beta$ apply to every exponential factor.
At each time,
\[
 \|\vartheta_0q_j\|_2^2
       =\int_0^\ell\chi_j^2w_j|p|^2\,\dd x,\qquad
 \|\vartheta_0\mathcal O_j(\chi_jPp)\|_2^2
       =\int_0^\ell\chi_j^2w_j|Pp|^2\,\dd x.
\]
Moreover,
\[
 \|S_\kappa[
    \vartheta_0(-2\partial_xF_j+G_j)]\|_2^2
 \leq C\kappa^2\int_\omega w_j|p|^2\,\dd x
\]
for $\kappa\geq1$. Inserting these estimates in
\eqref{eq:interval-piece-lowered} gives
\[
 \int_Q\kappa^3\chi_j^2w_j|p|^2
 \leq C\left(
  \int_Q\chi_j^2w_j|Pp|^2+
  \int_{Q_\omega}\kappa^2w_j|p|^2
 \right).
\]
Finally, set
\begin{equation*}
 W_s=\frac14(\chi_-^2w_-+\chi_+^2w_+),\qquad
 V_s=\frac14(w_-+w_+).
\end{equation*}
Adding the two estimates, dividing by four, and using
$\kappa^2\leq\kappa^3$ gives
\eqref{eq:interval-carleman}. Each $w_j$ lies between
$2e^{-2\beta\kappa}$ and $2e^{-2\alpha\kappa}$.
Together with \eqref{eq:interval-cutoff-cover}, this proves
\eqref{eq:interval-weight-bounds}.
In particular, $\theta=W_s^{1/2}$ and $\vartheta=V_s^{1/2}$
satisfy \eqref{eq:abstract-weight-envelopes} with $c_*=1/2$.
For fixed $s$, these weights and their inverses are bounded
on compact subsets of $Q$.
\end{proof}

The reflection construction applies only to linear test functions
and therefore imposes no parity assumption on the nonlinearity.
The interval case of Theorem~\ref{thm:main} follows by using
$\theta=W_s^{1/2}$ and $\vartheta=V_s^{1/2}$ in the control
construction of Section~\ref{sec:controllability}.

\begin{proposition}
\label{prop:uniform-short-time-carleman}
Fix either the Euclidean local-phase geometry of
Assumption~\ref{ass:phase-geometry} or an interval with any
nonempty open control set, together with the spatial phase and cutoffs. There are constants
$a,b,C_{\rm Car},c_*,\alpha,\beta>0$, depending only on these
spatial data, with the following property. For every $T>0$, set
$\kappa(t)=s/[t(T-t)]$. If
\begin{equation}\label{eq:uniform-admissible-parameter}
 s\geq s_0(T):=aT+bT^2,
\end{equation}
there are positive weights $\theta,\vartheta$ satisfying
the envelopes \eqref{eq:abstract-weight-envelopes} and the
Carleman estimate \eqref{eq:carleman-hum-form}.
The constants are independent of $T$ and $s$. In particular,
the threshold has order $T$ as $T\downarrow0$; no additional
condition $s\geq1$ is imposed. On an interval use
$\theta=W_s^{1/2}$ and $\vartheta=V_s^{1/2}$.
\end{proposition}

\begin{proof}
We keep all spatial data fixed and track the two dimensionless
quantities that enter the preceding proofs:
\[
 \delta=\frac Ts,\qquad m=\inf_t\kappa(t)=\frac{4s}{T^2}.
\]
We show that all preceding estimates hold once $\delta\leq\delta_0$ and
$m\geq M_0$, for suitably chosen constants $\delta_0>0$ and
$M_0\geq1$ depending only on the spatial data. By
Lemma~\ref{lem:timeweights},
\[
 |\kappa'|\leq\delta\kappa^2,\qquad
 |\kappa''|\leq2\delta^2\kappa^3,\qquad
 |\kappa'|^2/\kappa\leq\delta^2\kappa^3.
\]
In the strong estimate, the mixed time term is at most
$\varepsilon\kappa|\nabla w|^2+C_\varepsilon\delta^2
\kappa^3|w|^2$. The $\kappa''$ term has the same order.
We choose $\delta_0\leq1$ sufficiently small to absorb these
remainders, with a coercivity constant depending only on the phase. For a nonquadratic phase, the additional
$\kappa\Delta^2\psi$ term is absorbed by a fixed lower bound
on $m$. The boundary flux vanishes by tangency.

The spatial regularizer and off-diagonal estimates are uniform
for $\kappa\geq1$. Their time-dependent commutators satisfy
\[
 \|\partial_tS_{\kappa(t)}\|_{2\to2}
 +\|\kappa'[\psi,S_\kappa]\|_{2\to2}
 \leq C\frac{|\kappa'|}{\kappa}
 \leq C\delta_0\kappa.
\]
Every other regularizer commutator has norm at most $C\kappa$.
Thus, with constants now independent of $T$, the lowering
argument gives the remainder
$C_2\int_0^T\!\int_D\kappa^2|v|^2\,\dd x\,\dd t$,
where $D=\R^n$ in the full-space argument and $D=\Om$ in the
Dirichlet argument (with $D=\R$ for the reflected interval).
This term is absorbed when $m\geq2C_2$. The localized energy comparison and the cutoff source estimates
require only further fixed lower bounds on $m$.
We include all these bounds in $M_0$.
Reflection and folding on an interval involve only fixed
spatial maps and introduce no additional time dependence.

The temporal cutoff and density arguments do not introduce further
thresholds. For each fixed positive pair $(T,s)$, flatness ensures that the
approximation errors tend to zero with the same estimate constants. Hence the estimates remain valid even when
$s<1$, provided $T/s\leq\delta_0$ and $4s/T^2\geq M_0$.
By taking $a\geq\delta_0^{-1}$ and $b\geq M_0/4$, we obtain
\eqref{eq:uniform-admissible-parameter}. The weight envelopes
depend only on the extrema of the spatial phases.
\end{proof}

\subsection{The Laplace--Beltrami operator}
\label{subsec:riemannian-proof}

\begin{proof}[Proof of Proposition~\ref{prop:riemannian-control}]
We extend Proposition~\ref{thm:tangent-carleman} to the Riemannian
setting, with constants independent of the time horizon. All derivatives and inner products below are taken with
respect to $g$. Put $P_g=i\partial_t-\Delta_g$ and let $A=-\Delta_g$
have its Dirichlet realization when $\partial M\ne\varnothing$.
Otherwise $A$ is the self-adjoint Laplacian on $M$.
Its group $e^{itA}$ is unitary on $L^2(M)$ in both cases.
If $K_\omega=\varnothing$, the Carleman estimate below follows
from its observation term. We may therefore assume
$K_\omega\ne\varnothing$.
Shift $\psi$ so that $\psi<-1$, and set
$\kappa=s/[t(T-t)]$, $\theta=e^{\kappa\psi}$ and
$V=\nabla_g\psi$. Then
\[
 \theta P_g\theta^{-1}=\mathcal A+\mathcal B,\qquad
 \mathcal A=i\partial_t-\Delta_g-\kappa^2|V|_g^2,
 \quad
 \mathcal B=2\kappa V\cdot\nabla_g+\kappa\Delta_g\psi-i\kappa'\psi.
\]
For a smooth function $w$ with spatial support compactly contained
in $\mathcal O$, flat at the
temporal endpoints and satisfying the Dirichlet condition where
applicable, integration by parts gives
\begin{align*}
 2\Real(\mathcal Aw,\mathcal Bw)_{L^2}
 ={}&4\int_0^T\!\int_M\kappa\operatorname{Hess}_g\psi
               (\nabla_gw,\nabla_g\bar w)
       +4\int_0^T\!\int_M\kappa^3
                \operatorname{Hess}_g\psi(V,V)|w|^2\\
 &-\int_0^T\!\int_M\kappa\Delta_g^2\psi|w|^2
       +\int_0^T\!\int_M\kappa''\psi|w|^2
       -4\int_0^T\!\int_M\kappa'
                   \Imag(V\cdot\nabla_gw\,\bar w)\\
 &-2\int_0^T\!\int_{\partial M}
           \kappa\partial_{\nu_g}\psi\,|\partial_{\nu_g}w|^2.
\end{align*}
Here and below, volume and boundary integrals are taken with respect to
$\dd V_g\,\dd t$ and $\dd S_g\,\dd t$, respectively.
The boundary term is absent on a closed manifold. No separate curvature term occurs in this scalar
identity. Indeed,
\[
 [-\Delta_g,2V\cdot\nabla_g+\Delta_g\psi]
 =-4\operatorname{div}_g
       (\operatorname{Hess}_g\psi\,\nabla_g\,\cdot\,)
       -\Delta_g^2\psi,
\]
where the Hessian is viewed as a symmetric endomorphism. This
divergence form incorporates the curvature terms produced by
commuting covariant derivatives.

The boundary term vanishes by tangency. Using
$|\kappa'|\leq(T/s)\kappa^2$ and
$|\kappa''|\leq2(T/s)^2\kappa^3$, we obtain
\begin{equation}\label{eq:riemannian-strong}
 \int_0^T\bigl(\kappa\|\nabla_gw\|_2^2
                         +\kappa^3\|w\|_2^2\bigr)\,\dd t
 \leq C\|\theta P_g\theta^{-1}w\|_{L^2}^2
\end{equation}
once $T/s\leq\delta_0$ and $\inf_t\kappa\geq M_0$, with constants
depending only on the fixed geometric data.

The order reduction also has an intrinsic form. For $\mu\geq1$ and
$\kappa\geq1$, set $R_\mu=(\mu+A)^{-1}$ and
$S_\kappa=\kappa(\kappa^2+A)^{-1/2}$.
For a smooth vector field $X$ tangent to $\partial M$, let
$E_X$ denote the symmetric part of its covariant derivative.
The identity
\[
 [A,X]u=-2\operatorname{div}_g(E_X\nabla_g u)
              +\langle\nabla_g\operatorname{div}_gX,
                                      \nabla_g u\rangle_g
\]
defines a bounded map $H^1_0(M)\to H^{-1}(M)$.
On a closed manifold these spaces are $H^1(M)$, with its full norm,
and its dual.
Tangency ensures that $XR_\mu u$ has zero boundary trace.
Consequently, the weak resolvent identity
$[X,R_\mu]=R_\mu[A,X]R_\mu$ gives
$\|[X,R_\mu]\|_{2\to2}\leq C\mu^{-1}$.
For a smooth multiplier $b$,
$[A,b]=-2\nabla_gb\cdot\nabla_g-\Delta_gb$ similarly gives
$\|[b,R_\mu]\|_{2\to2}\leq C\mu^{-3/2}$.
Integrating the resolvents as in
\eqref{eq:dirichlet-bessel-resolvent} yields
\[
 \|[X,S_\kappa]\|_{2\to2}\leq C,
 \qquad \|[b,S_\kappa]\|_{2\to2}\leq C\kappa^{-1},
 \qquad
 \|S_\kappa\operatorname{div}_g F\|_2\leq\kappa\|F\|_2.
\]
The last bound follows from spectral calculus and duality, which also
give
\[
 \kappa\|\nabla_gS_\kappa v\|_2^2
       +\kappa^3\|S_\kappa v\|_2^2=\kappa^3\|v\|_2^2,
 \qquad
 \|\partial_tS_\kappa\|_{2\to2}\leq C|\kappa'|/\kappa.
\]
These spectral estimates use only $A\geq0$ and also hold on
$\ker A$ in the closed case.
The same cutoff argument gives the off-diagonal estimates of
Lemma~\ref{lem:dirichlet-tails}. Indeed, if
$u=R_\mu v$ and $\eta$ vanishes near $\operatorname{supp}v$,
\[
 \mu\|\eta u\|_2^2+\|\nabla_g(\eta u)\|_2^2
       =\||\nabla_g\eta|_g u\|_2^2.
\]

We choose the nested cutoffs $\chi,\chi_2,\rho$ as in the proof of
Proposition~\ref{thm:tangent-carleman}, and replace $V$ in the
transport term by the globally tangent field $\rho V$.
We apply \eqref{eq:riemannian-strong} to $\chi_2S_\kappa v$,
where $v=\theta\chi p$. The preceding commutator bounds give a
remainder $C\int_0^T\kappa^2\|v\|_2^2\,\dd t$, which we
absorb by increasing $M_0$. For the localized source, we use
\[
 \theta[P_g,\chi]p
 =-2\operatorname{div}_g(\theta p\nabla_g\chi)
       +(\Delta_g\chi+2\kappa\langle V,\nabla_g\chi\rangle_g)
              \theta p.
\]
The cutoff derivatives in this identity are supported in $\omega$.
The density and temporal cutoff arguments of the Euclidean proof
apply through the spectral resolution of $A$ and elliptic regularity.
Thus, for every smooth test function $p$ with unrestricted temporal
endpoint values and homogeneous Dirichlet trace when $\partial M\ne\varnothing$,
\begin{equation}\label{eq:riemannian-carleman}
 \int_0^T\!\int_M\kappa^3\theta^2|p|^2
 \leq C\left(\int_0^T\!\int_M\theta^2|P_gp|^2
       +\int_0^T\!\int_\omega\kappa^3\theta^2|p|^2\right).
\end{equation}
This estimate holds for $s\geq aT+bT^2$, with $a,b,C$ independent of $T,s$:
take $a\geq\delta_0^{-1}$ and $4b\geq M_0$.
The weight envelopes in \eqref{eq:abstract-weight-envelopes} follow
from $\psi<-1$, with $\vartheta=\theta$.

The argument in Section~\ref{sec:controllability} now applies with $\dd x$
replaced by $\dd V_g$ and $c_f=\operatorname{Vol}_g(M)^{1/2}|f(0)|$.
The duality, bridge and contraction arguments use only this Carleman
estimate, finite volume and the unitary group of $A$.
The same argument gives the shifted estimate
\eqref{eq:shifted-lipschitz-cost}, with a constant independent of
the reference trajectory, since
$z\mapsto f(u+z)-f(u)$ has Lipschitz constant $L$ and vanishes at zero.
Using the uniform threshold, we obtain
\eqref{eq:riemannian-control-cost} by choosing the parameter
and subdividing time as in Section~\ref{subsec:uniform-cost}.
This completes the proof of the proposition.
\end{proof}

\begin{proof}[Proof of Corollary~\ref{cor:one-dimensional-manifolds}]
Write $\ell=\operatorname{Vol}_g(M)$. If $\partial M\ne\varnothing$,
arc length identifies $M$ isometrically with $[0,\ell]$.
Indeed, in a coordinate with $g=a(x)^2\,\mathrm dx^2$, the change
$s=\int_0^x a(r)\,\mathrm dr$ gives
\[
 \mathrm dV_g=\mathrm ds,\qquad
 \Delta_g=a^{-1}\partial_x(a^{-1}\partial_x)=\partial_s^2.
\]
The change of variable preserves the pointwise nonlinearity,
Dirichlet conditions, control supports, and every norm in the three
theorems. Their interval statements therefore apply directly.

If $\partial M=\varnothing$, arc length identifies $M$ with
the circle $\R/\ell\mathbb Z$. Choose a cut point inside
$\omega$, so that $K=\overline{M\setminus\omega}$ is a compact
subset of the coordinate interval $(0,\ell)$. We choose a smooth
periodic function that agrees with $(s+1)^2$ near
$K$, modifying it only near the cut inside $\omega$.
Its derivative is nonzero and its second derivative equals
$2$ near $K$. Thus the intrinsic phase hypothesis holds,
and Proposition~\ref{prop:riemannian-control} gives the Lipschitz
conclusion and the shifted estimate
\eqref{eq:shifted-lipschitz-cost} on the circle.

The remaining dispersive estimate follows from the periodic
calculation in Lemma~\ref{lem:interval-l4}. For completeness,
put $P_*=\ell^2/(2\pi)$ and write
$v(s)=\sum_n c_n e^{2\pi i n s/\ell}$. For a Fourier polynomial,
the resonance relations
\[
 n_1+n_2=n_3+n_4,\qquad n_1^2+n_2^2=n_3^2+n_4^2
\]
imply $\{n_1,n_2\}=\{n_3,n_4\}$ and hence
\[
 \int_0^{P_*}\!\int_0^\ell |e^{-it\Delta_g}v|^4
             \,\mathrm ds\,\mathrm dt
 =\ell P_*\left[2\left(\sum_n|c_n|^2\right)^2
                       -\sum_n|c_n|^4\right]
 \leq\frac{2P_*}{\ell}\|v\|_2^4.
\]
Density and a finite covering by temporal periods give the
free $L^4$ estimate uniformly for $0<\tau\leq T_*$, for each fixed
$T_*>0$.
Minkowski's inequality gives the corresponding Duhamel estimate
\eqref{eq:interval-l4} for $L^1_tL^2_x$ forcing.

The remaining arguments in Section~\ref{sec:logarithmic}
use only this estimate, finite spatial volume, unitarity of the
free group, and the shifted Lipschitz control estimate.
The truncation and mass estimates rely on pointwise and Hilbert-space
arguments, and their volume factors remain $\ell^{1/2}$.
The forced well-posedness proof and both control cascades
therefore apply on the circle, with the same exponents,
regularity, and uniqueness classes. The canonical logarithmic
cost follows from the same scalar gauge transformation.
\end{proof}

\section{Lipschitz exact controllability}\label{sec:controllability}

First, we use the weighted duality to construct a linear control operator
with zero endpoint states. To handle arbitrary endpoints, we join
two uncontrolled trajectories and correct the resulting bridge by
a contraction argument. Then, we subdivide the time interval to obtain
the control cost in terms of the control time and the Lipschitz constant.

\subsection{Weighted linear control}\label{sec:linear}

Fix $T>0$, set $Q=(0,T)\times\Om$ and
$Q_\omega=(0,T)\times\omega$, and write $H=L^2(\Om;\C)$.
The inner product $\ip{\cdot}{\cdot}$ on $H$ is linear in its first
argument. We use the Dirichlet realization
$A=-\Delta_D$ and its unitary group $U(t)=e^{itA}$, so that
$P=i\partial_t+A=i\partial_t-\Delta$.
We formulate the geometric assumptions using a common weight for
the state and source, together with a separate observation weight.
This formulation includes the folded interval weights of
Proposition~\ref{prop:interval-carleman}; the observation weight
determines the control weight in the dual construction.

Set
\[
 \kappa(t)=\frac{s}{t(T-t)},\qquad s\geq s_0>0.
\]
Suppose that positive measurable weights $\theta=\theta_s$ and
$\vartheta=\vartheta_s$ on $Q$ satisfy
\begin{equation}\label{eq:abstract-weight-envelopes}
 c_*e^{-\beta\kappa}\leq\theta\leq e^{-\alpha\kappa},
 \qquad 0<\vartheta\leq e^{-\alpha\kappa},
\end{equation}
where the inequalities hold almost everywhere and
$c_*,\alpha,\beta>0$ are independent of $s$.
No lower envelope for $\vartheta$ or comparison between the two
weights is needed. Suppose, in addition, that
\begin{equation}\label{eq:carleman-hum-form}
 \int_Q\kappa^3\theta^2\abs{q}^2\,\dd x\,\dd t
 \le C_{\rm Car}\left(
 \int_Q\theta^2\abs{Pq}^2\,\dd x\,\dd t
 +\int_{Q_\omega}\kappa^3\vartheta^2\abs{q}^2\,\dd x\,\dd t
 \right)
\end{equation}
for smooth Dirichlet test functions $q$, with a constant
$C_{\rm Car}>0$ independent of $s\geq s_0$. The test functions have
unrestricted values at $t=0$ and $t=T$. The freedom to choose these
values independently enforces both endpoint conditions on the controlled
state.
Proposition~\ref{thm:carleman},
Proposition~\ref{thm:tangent-carleman}, and
Proposition~\ref{prop:interval-carleman} ensure these hypotheses
in the geometric configurations considered here. The argument
below uses only
\eqref{eq:abstract-weight-envelopes} and
\eqref{eq:carleman-hum-form}.

For this fixed value of $s$, define the weighted Hilbert spaces
\[
\begin{aligned}
 \Ss&=L^2(Q;\kappa^{-3}\theta^{-2}\,\dd x\,\dd t),\\
 \Ys&=L^2(Q;\theta^{-2}\,\dd x\,\dd t),\\
 \Us&=L^2(Q_\omega;\kappa^{-3}\vartheta^{-2}\,\dd x\,\dd t).
\end{aligned}
\]
These spaces consist of complex-valued functions. We equip
intersections of Banach spaces with the sum norm. We also set
\begin{equation*}
 m_s=\inf_{0<t<T}\kappa(t)=\frac{4s}{T^2},\qquad
 B_s=\mathop{\rm ess\,sup}_{(t,x)\in Q}
   \kappa(t)^3\max\{\theta(t,x)^2,\vartheta(t,x)^2\}<\infty.
\end{equation*}
The finiteness of $B_s$ follows from
\eqref{eq:abstract-weight-envelopes}, since
$\sup_{a>0}a^3e^{-2\alpha a}<\infty$.
Consequently,
\begin{equation}\label{eq:weighted-embeddings}
 \norm{F}_{L^2(Q)}\le B_s^{1/2}\norm{F}_{\Ss},\qquad
 \norm{h}_{L^2(Q_\omega)}\le B_s^{1/2}\norm{h}_{\Us},\qquad
 \norm{v}_{L^2(Q)}\le\norm{v}_{\Ys}.
\end{equation}
In formulas on $Q$, a control on $Q_\omega$ is extended by zero and
denoted by $\one_\omega h$.

\begin{proposition}\label{thm:linear}
Assume \eqref{eq:abstract-weight-envelopes} and
\eqref{eq:carleman-hum-form}.
For every $s\ge s_0$ there is a bounded linear map
\[
 \Lambda_s:\Ss\longrightarrow
 \bigl(\Ys\cap C([0,T];H)\bigr)\times\Us,
 \qquad F\longmapsto(v_F,h_F),
\]
such that
\[
 Pv_F=F+\one_\omega h_F,\qquad
 v_F(0)=v_F(T)=0.
\]
The equation and the homogeneous Dirichlet boundary condition are
understood in the mild sense associated with $A$. Moreover,
\begin{equation}\label{eq:linearbound}
 \norm{v_F}_{\Ys}^2+\norm{h_F}_{\Us}^2
 \le C_{\rm Car}\norm{F}_{\Ss}^2,
\end{equation}
and
\begin{equation}\label{eq:linear-continuous-bound}
 \norm{v_F}_{C([0,T];H)}
 \le \sqrt{T}\,B_s^{1/2}(1+\sqrt{C_{\rm Car}})
       \norm{F}_{\Ss}.
\end{equation}
\end{proposition}

\begin{proof}
Let
\[
 \mathcal D=
 \{q\in C^\infty([0,T]\times\overline\Om;\C):
              q=0\text{ on }[0,T]\times\partial\Om\}.
\]
On this space, define the Hermitian form
\begin{equation}\label{eq:hum-form}
 a_s(p,q)=
 \int_Q\theta^2Pp\,\overline{Pq}\,\dd x\,\dd t
 +\int_{Q_\omega}\kappa^3\vartheta^2p\bar q\,\dd x\,\dd t.
\end{equation}
Estimate~\eqref{eq:carleman-hum-form} shows that $a_s$ is positive
definite. Let $\Hs$ be the completion of $\mathcal D$ with respect to
the norm $a_s(q,q)^{1/2}$.

We first realize this completion as a space of functions. Let
$(p_j)$ be an $a_s$-Cauchy sequence. The sequences
$\kappa^{3/2}\theta p_j$ and $\theta Pp_j$ converge in $L^2(Q)$;
for the first sequence, this follows from~\eqref{eq:carleman-hum-form}.
The sequence $\kappa^{3/2}\vartheta p_j$ also converges in $L^2(Q_\omega)$
by the definition of $a_s$. On each compact subset of $Q$,
$\theta$ and $\kappa^{3/2}\theta$ are bounded above and below by
positive constants in view of \eqref{eq:abstract-weight-envelopes}.
Hence $p_j\to p$ in $L^2_{\rm loc}(Q)$, and the local $L^2$ limit
of $Pp_j$ agrees with $Pp$ in the sense of distributions. The local
boundedness of $\kappa^{3/2}\vartheta$ identifies the local limit of
the observation sequence as $\kappa^{3/2}\vartheta p$. This agrees
with its global $L^2(Q_\omega)$ limit. The same argument shows that
these identifications do not depend on the approximating sequence. Therefore,
$\theta Pp$, $\kappa^{3/2}\theta p$, and
$\kappa^{3/2}\vartheta p|_{Q_\omega}$ are well-defined $L^2$
functions, and both \eqref{eq:hum-form} and
\eqref{eq:carleman-hum-form} extend to $\Hs$.
This realization is injective. Indeed, if the local representative $p$
vanishes, then $Pp=0$ in the sense of distributions. Both weighted
$L^2$ limits defining $a_s$ therefore vanish, so the corresponding
completed element has zero $\Hs$ norm. We have thus obtained a
continuous embedding of $\Hs$ into $L^2_{\rm loc}(Q)$.

For $F\in\Ss$, the conjugate-linear functional
\[
 \ell_F(q)=\int_QF\bar q\,\dd x\,\dd t
\]
is continuous on $\Hs$, because
\[
 \abs{\ell_F(q)}
 \le\norm{F}_{\Ss}
       \left(\int_Q\kappa^3\theta^2\abs{q}^2\,\dd x\,\dd t\right)^{1/2}
 \le\sqrt{C_{\rm Car}}\norm{F}_{\Ss}a_s(q,q)^{1/2}.
\]
By the Riesz representation theorem, there exists a unique $p_F\in\Hs$
such that
\begin{equation}\label{eq:riesz-hum}
 a_s(p_F,q)=\ell_F(q)\qquad(q\in\Hs).
\end{equation}
Define
\begin{equation}\label{eq:hum-state-control}
 v_F=\theta^2Pp_F,\qquad
 h_F=-\kappa^3\vartheta^2p_F\big|_{Q_\omega}.
\end{equation}
Taking $q=p_F$ in~\eqref{eq:riesz-hum} gives
\[
 \norm{v_F}_{\Ys}^2+\norm{h_F}_{\Us}^2
 =a_s(p_F,p_F)
 \le\sqrt{C_{\rm Car}}\norm{F}_{\Ss}a_s(p_F,p_F)^{1/2},
\]
which proves~\eqref{eq:linearbound}. The construction is linear in
$F$.

It remains to show that the state constructed by duality is a mild
solution with zero initial and final states. Set $G=F+\one_\omega h_F$.
It follows from \eqref{eq:weighted-embeddings} that $G\in L^2(0,T;H)$ and
$v_F\in L^2(0,T;H)$. Combining \eqref{eq:riesz-hum} with
\eqref{eq:hum-state-control}, we obtain
\begin{equation}\label{eq:hum-transposition}
 \int_Qv_F\overline{Pq}\,\dd x\,\dd t
 =\int_QG\bar q\,\dd x\,\dd t
 \qquad(q\in\mathcal D).
\end{equation}
For smooth Dirichlet functions $v,q$, the underlying Green
identity is
\[
 \int_Qv\overline{Pq}\,\dd x\,\dd t
 =\int_Q(Pv)\bar q\,\dd x\,\dd t
       -i\bigl[\ip{v(t)}{q(t)}\bigr]_0^T.
\]
Since the temporal traces of $v_F$ are not yet available,
we recover them from \eqref{eq:hum-transposition}, coefficient by
coefficient in a Dirichlet eigenbasis.
Let $(e_k)_{k\ge1}$ be an orthonormal Dirichlet eigenbasis of $H$,
with $Ae_k=\lambda_ke_k$. Smoothness of the domain gives
$e_k\in C^\infty(\overline\Om)$, so
$q(t,x)=\zeta(t)e_k(x)$ is admissible for every
$\zeta\in C^\infty([0,T])$. Write
$v_k(t)=\ip{v_F(t)}{e_k}$ and $G_k(t)=\ip{G(t)}{e_k}$. Then
\begin{equation}\label{eq:spectral-transposition}
 \int_0^Tv_k(t)\bigl(-i\overline{\zeta'(t)}
                   +\lambda_k\overline{\zeta(t)}\bigr)\,\dd t
 =\int_0^TG_k(t)\overline{\zeta(t)}\,\dd t.
\end{equation}
Taking $\zeta\in C_c^\infty(0,T)$ first, we obtain
\[
 iv_k'+\lambda_kv_k=G_k
\]
in distributions. Since both $v_k$ and $G_k$ belong to $L^2(0,T)$,
we have $v_k\in H^1(0,T)$. We can therefore integrate by parts in
\eqref{eq:spectral-transposition} for arbitrary
$\zeta\in C^\infty([0,T])$. The differential equation cancels the
interior terms, leaving
\[
 -i\bigl[v_k(t)\overline{\zeta(t)}\bigr]_{t=0}^{t=T}=0.
\]
As the two endpoint values of $\zeta$ can be chosen independently,
\begin{equation}\label{eq:spectral-zero-endpoints}
 v_k(0)=v_k(T)=0\qquad(k\ge1).
\end{equation}

Define the Bochner integral
\begin{equation}\label{eq:linear-duhamel}
 w(t)=-i\int_0^tU(t-\tau)G(\tau)\,\dd\tau.
\end{equation}
Strong continuity and unitarity of $U$ imply that $w\in C([0,T];H)$
and
\[
 \norm{w}_{C([0,T];H)}
 \le\int_0^T\norm{G(t)}_H\,\dd t
 \le\sqrt{T}\norm{G}_{L^2(Q)}.
\]
For each $k$, the coefficient $\ip{w(t)}{e_k}$ and $v_k$ solve
the same scalar initial-value problem with zero initial data.
Scalar uniqueness and completeness of the eigenbasis give
$v_F=w$ almost everywhere on $Q$. We henceforth use this continuous
representative of $v_F$. Its terminal coefficients vanish by
\eqref{eq:spectral-zero-endpoints}; completeness of the
eigenbasis then gives $v_F(T)=0$. Formula~\eqref{eq:linear-duhamel}
also yields the Dirichlet mild formulation. Finally,
\eqref{eq:weighted-embeddings} and~\eqref{eq:linearbound} give
\[
 \norm{G}_{L^2(Q)}
 \le B_s^{1/2}(1+\sqrt{C_{\rm Car}})\norm{F}_{\Ss},
\]
which proves~\eqref{eq:linear-continuous-bound}.
\end{proof}

\begin{remark}
For an $L^2$ state, the Dirichlet boundary condition is encoded by
the group $U(t)$ and need not be understood as a classical spatial trace.
The temporal endpoint conditions follow from the scalar $H^1$ regularity
and the unrestricted endpoint values of the test functions
in~\eqref{eq:spectral-transposition}.
\end{remark}

\subsection{Arbitrary endpoints}
\label{sec:nonlinear}

Assume \eqref{eq:lip}, with Lipschitz constant $L\ge0$. As above, $f$ also
denotes the induced Nemytskii map on $H$. Since $\Om$ has finite measure,
\begin{equation}\label{eq:nemytskii-bounds}
 \norm{f(u)-f(v)}_H\le L\norm{u-v}_H,
 \qquad
 \norm{f(u)}_H\le L\norm{u}_H+c_f,
 \qquad c_f=\abs{\Om}^{1/2}\abs{f(0)}.
\end{equation}
In particular, no condition $f(0)=0$ is imposed.

\begin{lemma}
\label{lem:nonlinear-bridge}
Assume \eqref{eq:lip}, and fix $T>0$.
For every $y_0,y_1\in H$, there exist
$u^0,u^1\in C([0,T];H)$ satisfying $Pu^j=f(u^j)$ in the Dirichlet
mild sense and
\[
 u^0(0)=y_0,\qquad u^1(T)=y_1.
\]
There is a bridge $r\in C([0,T];H)$ with $r(0)=y_0$, $r(T)=y_1$
such that
\begin{equation}\label{eq:bridge-residual}
 R=f(r)-Pr\in\Ss
 \qquad\text{for every fixed }s\ge s_0.
\end{equation}
More precisely, $R\in C([0,T];H)$ and $R(t)=0$ for
$t\in[0,T/3]\cup[2T/3,T]$.
\end{lemma}

\begin{proof}
For data $z\in H$ at any time $t_0\in[0,T]$, the mild equation is
\begin{equation*}
 u(t)=U(t-t_0)z
      -i\int_{t_0}^tU(t-\tau)f(u(\tau))\,\dd\tau,
\end{equation*}
where the integral for $t<t_0$ is interpreted with the usual orientation.
By \eqref{eq:nemytskii-bounds}, the right-hand side is a contraction
on sufficiently short intervals in $C(H)$. We iterate this construction
in either time direction. The linear growth bound prevents blow-up
on $[0,T]$. More precisely, set
\begin{equation*}
 \Phi_L(a)=
 \begin{cases}
 (e^{La}-1)/L,&L>0,\\
 a,&L=0,
 \end{cases}
 \qquad a\ge0,
\end{equation*}
and apply the integral estimate and Gronwall's inequality to obtain
\begin{equation}\label{eq:flow-growth}
 \norm{u(t)}_H
 \le e^{L\abs{t-t_0}}\norm{z}_H
       +c_f\Phi_L(\abs{t-t_0}).
\end{equation}
This proves existence of both the forward trajectory $u^0$ and the
backward trajectory $u^1$.

Choose a real-valued function $\eta\in C^\infty([0,T])$ such that
\begin{equation}\label{eq:bridge-cutoff}
 0\le\eta\le1,\qquad
 \eta=1\text{ on }[0,T/3],\qquad
 \eta=0\text{ on }[2T/3,T],\qquad
 \norm{\eta'}_{L^\infty}\le c_\eta/T,
\end{equation}
where $c_\eta$ is an absolute constant determined by a fixed cutoff.
Set
\begin{equation*}
 r=\eta u^0+(1-\eta)u^1.
\end{equation*}
By the product rule, interpreted in the mild equations or their
weak formulation against $D(A)$, we obtain
\[
 Pr=\eta f(u^0)+(1-\eta)f(u^1)+i\eta'(u^0-u^1).
\]
Thus the residual in~\eqref{eq:bridge-residual} is the
$C([0,T];H)$ function
\begin{equation}\label{eq:residual-formula}
 R=f(r)-\eta f(u^0)-(1-\eta)f(u^1)-i\eta'(u^0-u^1).
\end{equation}
It vanishes for $t\le T/3$ and for $t\ge2T/3$. On the intervening
closed time interval, the weight $\kappa^{-3}\theta^{-2}$ is
bounded for each fixed $s$. It follows that $R\in\Ss$.
\end{proof}

\begin{proof}[Proof of exact controllability in Theorem~\ref{thm:main}]
Under the geometric assumption of the theorem, the corresponding
Carleman estimate gives weights satisfying \eqref{eq:abstract-weight-envelopes}
and \eqref{eq:carleman-hum-form}. In particular, for an arbitrary
nonempty open control set on an interval, take
$\theta=W_s^{1/2}$ and $\vartheta=V_s^{1/2}$ from
Proposition~\ref{prop:interval-carleman}; then $c_*=1/2$.
Fix $y_0,y_1\in H$ and construct $r$ and $R$ as in
Lemma~\ref{lem:nonlinear-bridge}. For $F\in\Ss$, let
$(v_F,h_F)=\Lambda_sF$ be the pair given by
Proposition~\ref{thm:linear}. To correct the bridge, define
\begin{equation}\label{eq:nonlinear-fixed-map}
 \mathcal N_s(F)=R+f(r+v_F)-f(r).
\end{equation}
The pointwise Lipschitz estimate gives
\begin{align*}
 \norm{f(r+v_F)-f(r)}_{\Ss}^2
 &\le L^2\int_Q\kappa^{-3}\theta^{-2}\abs{v_F}^2\,\dd x\,\dd t\\
 &\le L^2m_s^{-3}\norm{v_F}_{\Ys}^2.
\end{align*}
Since $R\in\Ss$, it follows that $\mathcal N_s$ maps $\Ss$ into
itself. Applying the same estimate to the difference of two sources
and using the linearity of $\Lambda_s$ together with
\eqref{eq:linearbound}, we obtain
\begin{equation}\label{eq:contraction-rate}
 \norm{\mathcal N_s(F)-\mathcal N_s(\widetilde F)}_{\Ss}
 \le q_s\norm{F-\widetilde F}_{\Ss},\qquad
 q_s=L\sqrt{C_{\rm Car}}\,m_s^{-3/2}.
\end{equation}
Since $C_{\rm Car}$ is independent of $s$, we can choose
$s\ge s_0$ such that $q_s<1$. This choice imposes no smallness
condition on $L$ or the endpoint states.

By the Banach fixed-point theorem, there exists a unique fixed point
$F_*\in\Ss$ for this choice of bridge and control operator, satisfying
\begin{equation*}
 F_*=R+f(r+v_{F_*})-f(r).
\end{equation*}
Since $\mathcal N_s(0)=R$, the fixed point satisfies
\begin{equation}\label{eq:fixed-source-bound}
 \norm{F_*}_{\Ss}\le\frac{\norm{R}_{\Ss}}{1-q_s}.
\end{equation}
Define
\begin{equation*}
 y=r+v_{F_*},\qquad h=h_{F_*}.
\end{equation*}
The bridge and the endpoint conditions for the linear correction give
$y\in C([0,T];H)$, $y(0)=y_0$, and $y(T)=y_1$. Moreover,
\begin{align*}
 Py
 &=Pr+F_*+\one_\omega h\\
 &=Pr+R+f(r+v_{F_*})-f(r)+\one_\omega h\\
 &=f(y)+\one_\omega h.
\end{align*}
The right-hand side belongs to $L^2(0,T;H)$ and hence to
$L^1(0,T;H)$. Adding the Dirichlet mild formulations for $r$
and $v_{F_*}$ therefore yields
\begin{equation*}
 y(t)=U(t)y_0-i\int_0^tU(t-\tau)
          \bigl(f(y(\tau))+\one_\omega h(\tau)\bigr)\,\dd\tau.
\end{equation*}
Finally, $h\in L^2(Q_\omega)$ by
\eqref{eq:weighted-embeddings}. For this fixed control, uniqueness
of the mild solution follows from~\eqref{eq:nemytskii-bounds} and
Gronwall's inequality. This proves the asserted exact
controllability for arbitrary $L^2$ endpoints.
The uniform bound \eqref{eq:unified-lipschitz-cost} is proved in
Section~\ref{subsec:uniform-cost}.
\end{proof}

\begin{proposition}
\label{prop:iteration}
Under the assumptions of
Theorem~\ref{thm:main}, let
$r$, $R$, $\Lambda_s$, and
$\mathcal N_s$ be the bridge, residual, control operator, and nonlinear
source map constructed above, respectively.
Fix $s\geq s_0$ with
$q_s=L\sqrt{C_{\rm Car}}m_s^{-3/2}<1$.
Let $F_*$ be the unique fixed point of $\mathcal N_s$, and set
$y=r+v_{F_*}$ and $h=h_{F_*}$. Define
\[
 F^{(0)}=0,\qquad F^{(n+1)}=\mathcal N_s(F^{(n)}),\qquad
 (v^{(n)},h^{(n)})=\Lambda_sF^{(n)},\qquad
 y^{(n)}=r+v^{(n)}.
\]
Then $F^{(n)}\to F_*$ in $\Ss$, $v^{(n)}\to v_{F_*}$ in
$\Ys\cap C([0,T];H)$, and $h^{(n)}\to h$ in $\Us$ and in
$L^2(Q_\omega)$. More precisely, for every $n\ge0$, with $q_s^0=1$,
\begin{equation}\label{eq:iteration-source-rate}
 \norm{F^{(n)}-F_*}_{\Ss}
 \le\frac{q_s^n}{1-q_s}\norm{R}_{\Ss},
\end{equation}
and
\begin{equation}\label{eq:iteration-pair-rate}
 \left(\norm{v^{(n)}-v_{F_*}}_{\Ys}^2
       +\norm{h^{(n)}-h}_{\Us}^2\right)^{1/2}
 \le\frac{\sqrt{C_{\rm Car}}q_s^n}{1-q_s}\norm{R}_{\Ss}.
\end{equation}
The convergence in $C([0,T];H)$ satisfies the estimate
\begin{equation}\label{eq:iteration-continuous-rate}
 \norm{y^{(n)}-y}_{C([0,T];H)}
 \le\sqrt T\,B_s^{1/2}(1+\sqrt{C_{\rm Car}})
       \frac{q_s^n}{1-q_s}\norm{R}_{\Ss}.
\end{equation}
Every iterate satisfies the endpoint conditions $y^{(n)}(0)=y_0$ and
$y^{(n)}(T)=y_1$ exactly, and its nonlinear equation has the residual
\begin{equation}\label{eq:iteration-residual}
 Py^{(n)}-f(y^{(n)})-\one_\omega h^{(n)}
 =F^{(n)}-F^{(n+1)},\qquad
 \norm{F^{(n)}-F^{(n+1)}}_{\Ss}
 \le q_s^n\norm{R}_{\Ss}.
\end{equation}
\end{proposition}

\begin{proof}
Because $F^{(1)}=R$, estimate~\eqref{eq:contraction-rate} gives
\[
 \norm{F^{(n+1)}-F^{(n)}}_{\Ss}
 \le q_s^n\norm{R}_{\Ss}.
\]
Summing this geometric bound from $n$ to infinity yields
\eqref{eq:iteration-source-rate}. By the linearity of $\Lambda_s$ and
\eqref{eq:linearbound}, we obtain~\eqref{eq:iteration-pair-rate}, while
\eqref{eq:linear-continuous-bound} yields
\eqref{eq:iteration-continuous-rate}. The control also converges in the
unweighted space by~\eqref{eq:weighted-embeddings}. Each iterate has the
required endpoint values by Proposition~\ref{thm:linear}. Substituting
\eqref{eq:bridge-residual} and
\eqref{eq:nonlinear-fixed-map} gives the identity in
\eqref{eq:iteration-residual}.
\end{proof}

\begin{proposition}\label{prop:cost}
Under the assumptions of
Theorem~\ref{thm:main}, let $r$ and
$R$ be the bridge and residual in its proof.
For each $s\geq s_0$ with $q_s<1$, the control constructed
from these data satisfies
\begin{equation}\label{eq:weighted-residual-cost}
 \norm{h}_{L^2(Q_\omega)}
 \le\frac{\sqrt{C_{\rm Car}B_s}}{1-q_s}\norm{R}_{\Ss}.
\end{equation}
There is a choice of $s$, depending on $L$ and on the fixed
geometric and temporal parameters, for which
\begin{equation}\label{eq:explicit-control-cost}
 \begin{aligned}
 \norm{h}_{L^2(Q_\omega)}
 \le{}& C_T(1+L)
       \exp\!\left(LT+C_T(1+L^{2/3})\right)\\
 &\quad\times
 \left(\norm{y_0}_H+\norm{y_1}_H
       +T\abs{\Om}^{1/2}\abs{f(0)}\right).
 \end{aligned}
\end{equation}
Here $C_T$ can depend on $\Om$, $\omega$, $T$ and the fixed
constants $s_0,C_{\rm Car},c_*,\alpha,\beta$ in
\eqref{eq:abstract-weight-envelopes}--\eqref{eq:carleman-hum-form}.
It is independent of $f$, $L$, $y_0$ and $y_1$.
\end{proposition}

\begin{proof}
Estimate~\eqref{eq:weighted-residual-cost} follows by combining
\eqref{eq:weighted-embeddings}, \eqref{eq:linearbound},
and~\eqref{eq:fixed-source-bound}.

To obtain the explicit bound, use $\alpha,\beta,c_*$ from
\eqref{eq:abstract-weight-envelopes} and set
\[
 m_0=\max\left\{\frac{4s_0}{T^2},1,\frac{3}{2\alpha}\right\}.
\]
Choose
\begin{equation}\label{eq:cost-parameter-choice}
 m=\max\{m_0,(2L\sqrt{C_{\rm Car}})^{2/3}\},\qquad
 s=\frac{T^2m}{4}.
\end{equation}
Then $s\ge s_0$, $m_s=m$, and $q_s\le1/2$. Moreover,
$m\le C_T(1+L^{2/3})$. Since
$a\mapsto a^3e^{-2\alpha a}$ is decreasing for
$a\ge3/(2\alpha)$,
\begin{equation}\label{eq:Bs-cost-bound}
 B_s^{1/2}\le m^{3/2}e^{-\alpha m}.
\end{equation}
On the support of $R$, the choice of cutoff gives
\[
 m\le\kappa(t)\le\frac98m.
\]
It follows that
\begin{equation}\label{eq:weighted-R-cost-bound}
 \norm{R}_{\Ss}
 \le c_*^{-1}m^{-3/2}e^{9\beta m/8}\norm{R}_{L^2(Q)}.
\end{equation}

We next estimate the unweighted residual. Set
$d=\norm{y_0}_H+\norm{y_1}_H$. By \eqref{eq:flow-growth},
we have, for $0\le t\le T$,
\begin{align*}
 \norm{u^0(t)-u^1(t)}_H
 &\le e^{Lt}\norm{y_0}_H+e^{L(T-t)}\norm{y_1}_H
       +c_f\bigl(\Phi_L(t)+\Phi_L(T-t)\bigr)\\
 &\le e^{LT}d+c_f\Phi_L(T)
 \le e^{LT}(d+Tc_f).
\end{align*}
For $L>0$, the middle inequality follows from
$(e^{Lt}-1)(e^{L(T-t)}-1)\ge0$; for $L=0$, it is immediate.
For the nonlinear part of the residual, the pointwise Lipschitz
condition gives
\begin{align*}
 &\abs{f(\eta u^0+(1-\eta)u^1)
          -\eta f(u^0)-(1-\eta)f(u^1)}\\
 &\qquad\le
 \eta L(1-\eta)\abs{u^0-u^1}
 +(1-\eta)L\eta\abs{u^0-u^1}
 \le\frac L2\abs{u^0-u^1}.
\end{align*}
Using~\eqref{eq:residual-formula}, the support of $R$, and
\eqref{eq:bridge-cutoff}, we obtain
\begin{equation}\label{eq:unweighted-R-cost-bound}
 \norm{R}_{L^2(Q)}
 \le\sqrt{T/3}\left(\frac L2+\frac{c_\eta}{T}\right)
       e^{LT}(d+Tc_f)
 \le C_T(1+L)e^{LT}(d+Tc_f).
\end{equation}
Combining~\eqref{eq:weighted-residual-cost}--
\eqref{eq:unweighted-R-cost-bound} and using $(1-q_s)^{-1}\le2$,
we cancel the factors $m^{3/2}$ and $m^{-3/2}$ and obtain
\[
 \norm{h}_{L^2(Q_\omega)}
 \le C_T(1+L)e^{LT}e^{(9\beta/8-\alpha)m}(d+Tc_f).
\]
Using the bound on $m$ following~\eqref{eq:cost-parameter-choice}
and increasing $C_T$ if necessary, we obtain
\eqref{eq:explicit-control-cost}.
\end{proof}

\begin{remark}

The term involving $f(0)$ in~\eqref{eq:explicit-control-cost}
accounts for the forcing present even at the zero state. It cannot
in general be omitted from a bound valid for all globally
Lipschitz nonlinearities and all endpoint data. For example, if
$f\equiv c\ne0$ and $y_0=y_1=0$, the uncontrolled final state is
$-ic\int_0^TU(T-t)\one_\Om\,\dd t$, which is nonzero at
nonresonant times.

The nonlinear endpoint trajectories in
Lemma~\ref{lem:nonlinear-bridge} ensure that the residual vanishes near
the endpoints. For a general $f$, replacing these trajectories by
solutions of $Pu=f(0)$ leaves the residual $f(r)-f(0)$ near the
endpoints. This residual need not belong to $\Ss$ because of the
exponentially growing weight. If $f$
is real-affine, its affine flow provides the required
endpoint trajectories. For example, when
$f(z)=az+b\bar z+c$, the real cutoff satisfies
$f(r)=\eta f(u^0)+(1-\eta)f(u^1)$, so the residual reduces exactly
to $R=-i\eta'(u^0-u^1)$.
\end{remark}

\begin{remark}
The Lipschitz contraction argument alone does not apply to a general
locally Lipschitz or superlinear nonlinearity. For example, if
$g(z)=|z|^{p-1}z$ with $p>1$, the available difference estimate is
\[
 |g(z)-g(w)|\leq C_p(|z|+|w|)^{p-1}|z-w|.
\]
The factor involving the states is not controlled by the
$\Ys$ norm. A weighted $L^2$ estimate does not imply a pointwise bound,
and $g$ does not map every $L^2(\Om)$ state into $L^2(\Om)$.
Thus a large Carleman parameter alone does not make this map
contractive on the source space used here. Truncation gives
a controlled solution of the truncated equation, but passing to
the original equation requires an additional bound uniform in the
truncation level.

For the logarithmic class in Section~\ref{sec:logarithmic}, we use
the truncation and stability argument given below.
\end{remark}

\subsection{Uniform control costs}\label{subsec:uniform-cost}

\begin{proof}[Proof of the cost bound in Theorem~\ref{thm:main}]
We first derive a control estimate on a single interval of length
$0<\tau\leq1$. We then concatenate the resulting controls, choosing
the interval length to limit the growth of the uncontrolled trajectories.
Put $c_f=|\Om|^{1/2}|f(0)|$, $a_j=\tau c_f$, and
$d_j=\|z_0\|_2+\|z_1\|_2$ for the two endpoints.
The bridge in Lemma~\ref{lem:nonlinear-bridge}, together with
\eqref{eq:flow-growth} and \eqref{eq:unweighted-R-cost-bound}, gives
\[
 \|r\|_{C L^2}\leq e^{L\tau}(d_j+a_j),\qquad
 \|R\|_{L^2}\leq C\sqrt\tau(L+\tau^{-1})
                         e^{L\tau}(d_j+a_j).
\]
For the linear map in Proposition~\ref{thm:linear}, define
\[
 \mathcal N_s(F)=R+f(r+v_F)-f(r).
\]
Writing $m=4s/\tau^2$, the weighted embedding and
\eqref{eq:linearbound} give
\[
 \|\mathcal N_s(F)-\mathcal N_s(G)\|_{\Ss}
 \leq L\sqrt{C_{\rm Car}}\,m^{-3/2}\|F-G\|_{\Ss}.
\]
The nonlinear difference belongs to the weighted source space,
and $\mathcal N_s(0)=R$. When the factor above is less than one,
this map has a unique fixed point, which gives the controlled pair
$y=r+v_F$, $h=h_F$.
The growth bound \eqref{eq:nemytskii-bounds} ensures that the
controlled state satisfies the mild equation \eqref{eq:mild}.
Using the uniform Carleman threshold, choose
\[
 m=\max\left\{\frac{4a}{\tau}+4b,\,1,\,
     \frac3{2\alpha},\,(2L\sqrt{C_{\rm Car}})^{2/3}\right\},
 \qquad s=\frac{\tau^2m}{4}.
\]
Then the contraction factor is at most $1/2$ and
$m\leq C(\tau^{-1}+L^{2/3})$. Combining \eqref{eq:Bs-cost-bound} with
\eqref{eq:weighted-R-cost-bound}, we cancel the factors $m^{3/2}$
and $m^{-3/2}$. Using the resulting control bound and
\eqref{eq:linear-continuous-bound}, we obtain
\[
 \begin{aligned}
 \|h\|_2&\leq C\sqrt\tau(L+\tau^{-1})
                      e^{Cm+L\tau}(d_j+a_j),\\
 \|v\|_{C L^2}&\leq C\tau(L+\tau^{-1})
                      e^{Cm+L\tau}(d_j+a_j).
 \end{aligned}
\]
Since $y=r+v$, we obtain, after absorbing these polynomial factors,
\begin{equation}\label{eq:uniform-segment-cost}
 \|h\|_2+\|y\|_{C L^2}
 \leq C e^{C(\tau^{-1}+L^{2/3})+L\tau}(d_j+a_j).
\end{equation}
We now let $T>0$ be arbitrary. To remove the factor $e^{LT}$, set
\[
 J=\max\{1,\lceil T(1+L)^{1/3}\rceil\},\qquad \tau=T/J.
\]
Then $0<\tau\leq1$, $\tau^{-1}\leq T^{-1}+(1+L)^{1/3}$ and
$L\tau\leq(1+L)^{2/3}$. Hence the prefactor in each segment estimate
is bounded by the same constant $A=C e^{C(T^{-1}+L^{2/3})}$.
If $J=1$, apply \eqref{eq:uniform-segment-cost} directly.
If $J\geq2$, join $y_0$ to zero on the first interval, zero
to zero on all intervening intervals, and zero to $y_1$ on
the last interval. The single-interval construction supplies the
controls on the intervening intervals even when $f(0)\ne0$.
The corresponding trajectories have zero endpoint states, although
they need not vanish between the endpoints.

On each segment, the forcing belongs to $L^1_tH$ and the
control belongs to $L^2_tL^2(\omega)$. The states agree at every junction.
Thus, by repeated use of the group identity $U(t-r)U(r-q)=U(t-q)$,
we obtain the mild equation on the whole interval, with no distributional
jump term. The global control norm is the Euclidean norm
of the segment control norms. Since $\sum_j a_j=Tc_f$ and
the first and last $d_j$ are the only ones that can be nonzero,
the triangle inequality yields
\[
 \|h\|_2\leq A\Bigl(\sum_j(d_j+a_j)^2\Bigr)^{1/2}
 \leq A(\|y_0\|_2+\|y_1\|_2+Tc_f).
\]
The supremum norm of the state is bounded by $A\max_j(d_j+a_j)$,
which satisfies the same final bound. Increasing $C$ yields
\eqref{eq:unified-lipschitz-cost} for every $T>0$.
For a prescribed control, the global Lipschitz bound and Gronwall's
inequality give uniqueness of the mild solution in $C([0,T];H)$.
\end{proof}

\begin{remark}
No conservation law is required in
Theorem~\ref{thm:main}. If $f$ is
mass preserving, its uncontrolled trajectories have constant norm.
The single-interval estimate then avoids $e^{L\tau}$.
For $T\leq1$, it gives the uniform cost directly.
For $T>1$, we follow the uncontrolled flow until $T-1$ and
apply the single-interval estimate on the last unit interval.
The norm is unchanged during the uncontrolled evolution, so
the same uniform bound holds without concatenating several
controlled intervals. For continuous scalar $f$,
$\operatorname{Im}(f(z)\bar z)=0$ for every $z$ also implies
$f(0)=0$, as seen by testing $z=r$ and $z=ir$ and letting $r\to0$.
\end{remark}

For the logarithmic truncations, we need a control estimate for
a nonlinear difference along a prescribed reference function.

\begin{lemma}\label{lem:shifted-control}
Assume one of the geometric conditions in Theorem~\ref{thm:main},
and let $f:\C\to\C$ satisfy Assumption~\ref{ass:lip} with
Lipschitz constant $L$. Fix $T>0$ and a measurable function
$u:Q\to\C$ that is finite almost everywhere.
For every $z_0,z_1\in H$, there is a control
$h\in L^2(Q_\omega;\C)$ such that the Dirichlet problem
\[
 Pz=f(u+z)-f(u)+\one_\omega h,\qquad z(0)=z_0,
\]
has a unique mild solution $z\in C([0,T];H)$ satisfying
$z(T)=z_1$. For every $T>0$, the control and solution can be
chosen so that
\begin{equation}\label{eq:shifted-lipschitz-cost}
 \|h\|_{L^2(Q_\omega)}+\|z\|_{C([0,T];H)}
 \leq C e^{C(T^{-1}+L^{2/3})}
              \bigl(\|z_0\|_H+\|z_1\|_H\bigr).
\end{equation}
Here $C$ depends only on the fixed spatial geometry, the control
region and, where applicable, the phase data. In particular, it is
independent of $u$, $f(0)$, $T$, $L$ and the endpoint states.
\end{lemma}

\begin{proof}
The pointwise inequalities
\[
 \begin{aligned}
 |f(u+v)-f(u+w)|&\leq L|v-w|,\\
 |f(u+v)-f(u)|&\leq L|v|
 \end{aligned}
\]
show that the nonlinear difference defines a strongly measurable
$H$-valued forcing along every continuous $H$-valued trajectory.
This forcing is globally Lipschitz in the trajectory and vanishes at zero.
The Duhamel contraction on short time intervals can therefore be
iterated to give unique forward and backward uncontrolled trajectories.
No integrability of $u$ is needed.

As in the preceding proof, we first work on an interval of length
$0<\tau\leq1$ and translate its left endpoint to zero.
For the prescribed endpoints, set
$d=\|z_0\|_H+\|z_1\|_H$.
Let $v^0$ and $v^1$ be the uncontrolled trajectories starting
from $z_0$ and ending at $z_1$, respectively. By Gronwall's inequality,
\[
 \|v^0\|_{C H}+\|v^1\|_{C H}\leq e^{L\tau}d.
\]
Choose the cutoff $\eta$ as in \eqref{eq:bridge-cutoff}, with
$T$ replaced by $\tau$, and set $r=\eta v^0+(1-\eta)v^1$.
Its residual is
\[
 \begin{aligned}
 R={}&f(u+r)-f(u)
       -\eta\bigl[f(u+v^0)-f(u)\bigr]\\
    &-(1-\eta)\bigl[f(u+v^1)-f(u)\bigr]
       -i\eta'(v^0-v^1).
 \end{aligned}
\]
The mild equations imply
$Pr=f(u+r)-f(u)-R$. Moreover, $R$ vanishes outside
$[\tau/3,2\tau/3]$ and
\[
 \|r\|_{C H}\leq e^{L\tau}d,\qquad
 \|R\|_{L^2((0,\tau)\times\Om)}
 \leq C\sqrt\tau(L+\tau^{-1})e^{L\tau}d.
\]
Indeed, the nonlinear part of $R$ is bounded pointwise by
$2L\eta(1-\eta)|v^0-v^1|$.
The support property also gives $R\in\Ss$ for every
admissible Carleman parameter.

For $(v_F,h_F)=\Lambda_sF$, define the nonlinear correction map
\[
 \mathcal N_s(F)=R+f(u+r+v_F)-f(u+r).
\]
By \eqref{eq:weighted-embeddings} and \eqref{eq:linearbound},
its Lipschitz constant on $\Ss$ is at most
$L\sqrt{C_{\rm Car}}m^{-3/2}$, where $m=4s/\tau^2$.
Choose $s$ exactly as in the proof of
\eqref{eq:uniform-segment-cost}, so that this constant is at
most $1/2$ and $m\leq C(\tau^{-1}+L^{2/3})$.
The fixed point gives $z=r+v_F$ with the required endpoints.
Using the residual bound above together with
\eqref{eq:Bs-cost-bound}, \eqref{eq:weighted-R-cost-bound}
and \eqref{eq:linear-continuous-bound}, we obtain
\[
 \|h\|_2+\|z\|_{C H}
 \leq C e^{C(\tau^{-1}+L^{2/3})+L\tau}d.
\]
The nonlinear terms are estimated using only the two displayed
Lipschitz bounds. Thus the estimates are independent of $u$ and $f(0)$.

To pass to arbitrary $T>0$, set
$J=\max\{1,\lceil T(1+L)^{1/3}\rceil\}$ and $\tau=T/J$.
Then $0<\tau\leq1$, $\tau^{-1}\leq T^{-1}+(1+L)^{1/3}$ and
$L\tau\leq(1+L)^{2/3}$.
If $J\geq2$, join $z_0$ to zero on the first interval and
zero to $z_1$ on the last, taking both the solution and the
control to be zero on the intervening intervals.
The concatenated trajectories satisfy the mild equation on $[0,T]$.
Taking the Euclidean norm of the segment control norms and the
supremum of the state norms gives \eqref{eq:shifted-lipschitz-cost}.
For $J=1$, the single-interval estimate gives the same conclusion.
Uniqueness for the prescribed control follows from the
Lipschitz difference bound and Gronwall's inequality.
\end{proof}

\section{Logarithmic nonlinearities}
\label{sec:logarithmic}

Throughout this section, let $\Om=(0,\ell)$ and let $\omega$ be a
nonempty open subset of $\Om$. We first assume that
Assumption~\ref{ass:radial-log-class} holds and write
\[
 f(z)=F(z)+\mathfrak g(z),\qquad F(z)=z\phi(|z|),\quad F(0)=0,
 \qquad \operatorname{Lip}(\mathfrak g)\leq L_g.
\]
The exponent $0\leq\gamma<1/2$ controls the derivative of the
radial coefficient with respect to logarithmic amplitude. The
critical case $\gamma=1/2$ is treated separately in
Subsection~\ref{subsec:critical-log}. Neither $f(0)=0$ nor a
conservation law is assumed for the full nonlinearity.
To reach an arbitrary terminal state, we first construct an
uncontrolled reference trajectory ending at that state. We then
successively control the difference from this reference. At each
step, we compare the true trajectory with the truncated controlled
trajectory through a common truncation.

\subsection{Truncation and the forced equation}

For the Cauchy theory of the logarithmic Schr\"odinger equation, see
\cite{CazenaveHaraux1980,CHO2024,HayashiOzawa2025}.
The truncation and well-posedness lemmas in this subsection
remain valid for $0\leq\gamma<1$ under the same derivative
bound. In particular, they apply at $\gamma=1/2$.
The restriction $\gamma<1/2$ is needed for the control
construction in the next subsection.

The stability estimate below extends the imaginary-part
cancellation for the canonical logarithmic nonlinearity; see,
for example, \cite[Lemma~2.1 and Section~2.2]{CarlesGallagher2018}.
We keep the dependence on the truncation level explicit for its
later use in the control construction.

\begin{lemma}\label{lem:log-truncation}
For $K\geq1$, we define
\[
 \rho_K(r)=\min\{e^K,\max\{e^{-K},r\}\},\qquad
 \phi_K(r)=\phi(\rho_K(r)),\qquad
 f_K(z)=z\phi_K(|z|)+\mathfrak g(z).
\]
There is a constant $C$, independent of $K$, such that
\begin{align}
 \operatorname{Lip}(f_K)&\leq C(1+K)^{1+\gamma},
                         \label{eq:log-lip-truncation}\\
 |f(z)-f_K(z)|&\leq C(1+K)^\gamma e^{-K}(1+|z|^2),
                         \label{eq:log-tail-error}\\
 |f(z)|+|f_K(z)|&\leq C(1+|z|^{3/4}+|z|^{5/4}).
                         \label{eq:log-subpower-growth}
\end{align}
Moreover, with $\Gamma_K=L_g+D(1+K)^\gamma$,
\begin{equation}\label{eq:log-one-sided}
 \left|\operatorname{Im}\bigl((f_K(z)-f_K(w))
                                      \overline{z-w}\bigr)\right|
 \leq\Gamma_K|z-w|^2.
\end{equation}
Let $Pu=f_K(u)+H$ and $Pv=f_K(v)+J$ be Dirichlet mild
solutions with integrable $L^2$ right-hand sides. Then
\begin{equation}\label{eq:log-forced-stability}
 \|u(t)-v(t)\|_2\leq e^{\Gamma_K t}
 \left(\|u(0)-v(0)\|_2+
              \int_0^t\|H(r)-J(r)\|_2\,\dd r\right).
\end{equation}
Every mild solution of $Pu=f_K(u)+H$, and every mild solution
with $f$ in place of $f_K$ and integrable nonlinearity, satisfies
\begin{equation}\label{eq:log-mass}
 \|u(t)\|_2\leq e^{L_g t}
 \left(\|u(0)\|_2+t\ell^{1/2}|\mathfrak g(0)|
                       +\int_0^t\|H(r)\|_2\,\dd r\right).
\end{equation}
If $\gamma=0$, \eqref{eq:log-one-sided} and
\eqref{eq:log-forced-stability} also hold for $f$ with the
constant $L_g+D$.
\end{lemma}

\begin{proof}
Let $\Phi(q)=\phi(e^q)$. By local absolute continuity,
$|\Phi'(q)|\leq D(1+|q|)^\gamma$ almost everywhere, and hence
\[
 |\phi(r)|\leq |\phi(1)|+C D(1+|\log r|)^{1+\gamma}.
\]
The truncation satisfies
$r|\phi_K'(r)|\leq D(1+K)^\gamma$ almost everywhere.
The associated radial map is linear near zero and locally
Lipschitz elsewhere. The derivative bound
$|\phi_K(r)|+r|\phi_K'(r)|\leq C(1+K)^{1+\gamma}$
proves \eqref{eq:log-lip-truncation}.

For the lower tail, write $r=e^{-K-b}$, $b\geq0$. Then
\[
 r|\phi(r)-\phi(e^{-K})|
 \leq D e^{-K-b}\int_K^{K+b}(1+q)^\gamma\,\dd q
 \leq C D(1+K)^\gamma e^{-K},
\]
because $b(1+b)^\gamma e^{-b}$ is bounded.
For the upper tail, write $r=e^{K+b}$ and apply the same bound
after dividing by $e^{-K}r^2$. This yields
\eqref{eq:log-tail-error}, since the perturbation cancels exactly.
For the growth estimate, every fixed power of
$1+|\log r|$ is dominated by $r^{-1/4}+r^{1/4}$, and the
truncation reduces the absolute logarithmic amplitude.

For a real radial coefficient $\varphi$ with
$r|\varphi'(r)|\leq B$, write $r=|z|\geq s=|w|>0$.
Then $|\varphi(r)-\varphi(s)|\leq B(r-s)/s$ and
$|\operatorname{Im}(z\bar w)|\leq s|z-w|$. The identity
\[
 \operatorname{Im}\bigl((z\varphi(r)-w\varphi(s))
                                      \overline{z-w}\bigr)
 =\bigl(\varphi(s)-\varphi(r)\bigr)\operatorname{Im}(z\bar w)
\]
gives the one-sided bound. The case of a zero amplitude follows
directly. Applying this bound to $\phi_K$ and adding the
Lipschitz bound for $\mathfrak g$, we obtain the stated estimate.
The interaction variables $U(-t)u(t)$ are absolutely continuous
in $L^2$ when the right-hand side is integrable. Applying the
energy identity to the difference, regularizing the norm at zero,
and then using Gronwall's inequality, we obtain
\eqref{eq:log-forced-stability}.

The real radial term has zero contribution to the mass identity,
while $|\mathfrak g(z)|\leq L_g|z|+|\mathfrak g(0)|$.
This yields \eqref{eq:log-mass}. When $\gamma=0$,
the original coefficient satisfies the required uniform derivative
bound, so the final assertion follows.
\end{proof}

The following periodic restriction estimate is a basic case of
the Fourier estimates in \cite{Bourgain1993}. We include a proof
based on the resonance relations.

\begin{lemma}\label{lem:interval-l4}
For every fixed $T_*>0$, there exists $B=B(\ell,T_*)$ such that
for $0<\tau\le T_*$, $v\in L^2(0,\ell)$, and
$F\in L^1(0,\tau;L^2(0,\ell))$,
\begin{equation}\label{eq:interval-l4}
 \left\|U(t)v-i\int_0^tU(t-s)F(s)\,\dd s\right\|_{L^4((0,\tau)\times(0,\ell))}
 \le B\left(\norm v_2+\norm{F}_{L^1(0,\tau;L^2)}\right).
\end{equation}
\end{lemma}
\begin{proof}
We extend $v$ oddly to the circle of length $2\ell$. The periodic
Schr\"odinger flow has time period $P_*=2\ell^2/\pi$.
For a Fourier polynomial, integrating the fourth power of the absolute
value of its evolution over one spatial and one temporal period leaves only
the resonant relations
\[
 n_1+n_2=n_3+n_4,\qquad n_1^2+n_2^2=n_3^2+n_4^2.
\]
These imply $\{n_1,n_2\}=\{n_3,n_4\}$. In fact, for
$V(x)=\sum_n c_ne^{in\pi x/\ell}$ and its periodic flow
$U_{\rm per}$, the exact identity is
\[
 \int_0^{P_*}\!\int_{-\ell}^{\ell}
       |U_{\rm per}(t)V|^4\,\dd x\,\dd t
 =2\ell P_*\left[2\left(\sum_n|c_n|^2\right)^2
                         -\sum_n|c_n|^4\right].
\]
Parseval's identity and a density argument give the free estimate
over one period. The odd extension has twice the squared $L^2$
norm of $v$, and the fourth power of its evolution is even in
space. Covering $[0,T_*]$ by finitely many periods gives a
constant independent of $\tau\le T_*$. Restricting the flow to
$(0,\ell)$ yields the Dirichlet estimate. For each fixed $s$, the
same bound applies to $\one_{\{s<t<\tau\}}U(t-s)F(s)$.
An application of Minkowski's inequality then gives the estimate
for the Duhamel term.
\end{proof}

\begin{lemma}
\label{lem:log-wellposed}
Let $T>0$, $v\in L^2(0,\ell)$ and
$H\in L^1(0,T;L^2(0,\ell))$. Then there exists a unique Dirichlet
mild solution of $Pu=f(u)+H$, $u(0)=v$, in the class
\[
 u\in C([0,T];L^2)\cap L^4((0,T)\times(0,\ell)).
\]
Its nonlinearity belongs to $L^2$ on the cylinder, and
\eqref{eq:log-mass} holds. For $0<T\leq1$, both this solution
and all its truncation approximants satisfy
\begin{equation}\label{eq:log-uniform-l4-polynomial}
 \|u\|_{L^4}\leq C(1+M^{3/2}),\qquad
 M=e^{L_g T}\bigl(\|v\|_2+T\ell^{1/2}|\mathfrak g(0)|
                                      +\|H\|_{L^1L^2}\bigr),
\end{equation}
where $C$ depends only on the fixed interval and nonlinearity
parameters. For $\gamma=0$, uniqueness also holds among
$C([0,T];L^2)$ mild solutions with nonlinearity in $L^1L^2$.
\end{lemma}

\begin{proof}
Let $u_K$ solve the globally Lipschitz equation with $f_K$.
For each $K$, its nonlinearity belongs to $L^1(0,T;L^2)$.
Lemma~\ref{lem:interval-l4} therefore gives $u_K\in L^4$.
The mass bound is uniform in $K$. Set $Z_K=\|u_K\|_{L^4}$.
Combining Lemma~\ref{lem:interval-l4} with the growth bound
and interpolation, we obtain, with constants independent of $K$,
\[
 Z_K\leq B\bigl[M+CT\ell^{1/2}
       +CT\ell^{1/8}M^{3/4}
       +CT^{7/8}M^{3/4}Z_K^{1/2}\bigr].
\]
Indeed, the inequality $\||u_K(t)|^{5/4}\|_2
\leq M^{3/4}\|u_K(t)\|_4^{1/2}$, followed by H\"older's inequality
in time, yields the last term. Young's inequality then gives a
uniform finite $L^4$ bound for fixed $T$, as well as
\eqref{eq:log-uniform-l4-polynomial} for $T\leq1$.

To compare two approximants with $J\geq K$, we rewrite both
equations with $f_K$ and treat $f_J(u_J)-f_K(u_J)$ as an
additional source in the equation for $u_J$.
The tail estimate gives
\[
 |f_J(z)-f_K(z)|
 \leq C(1+K)^\gamma e^{-K}(1+|z|^2),\qquad J\geq K,
\]
since $(1+K)^\gamma e^{-K}$ is decreasing for $K\geq1$.
Equations \eqref{eq:log-tail-error} and
\eqref{eq:log-forced-stability} imply
\[
 \|u_K-u_J\|_{C L^2}
 \leq C(1+K)^\gamma e^{-K+T\Gamma_K}
                       (T\ell^{1/2}+\sqrt T\,\|u_J\|_{L^4}^2).
\]
The factor on the right tends to zero because $\gamma<1$,
so the approximants converge in $C([0,T];L^2)$ to $u$.
Passing to a subsequence gives almost everywhere convergence,
and Fatou's lemma yields $u\in L^4$. To identify the limiting
nonlinearity, consider first a point where $u\ne0$. The
approximants eventually lie in a compact annulus on which the
radial truncation is inactive. At a point where $u=0$, write
$f_K(u_K)=F_K(u_K)+\mathfrak g(u_K)$, where
$F_K(z)=z\phi_K(|z|)$. The bound
$|F_K(z)|\leq C(|z|^{3/4}+|z|^{5/4})$, without a constant
term, gives $F_K(u_K)\to0$, while continuity gives
$\mathfrak g(u_K)\to\mathfrak g(0)=f(0)$.
Thus $f_K(u_K)\to f(u)$ almost everywhere.
For every measurable subset $E\subset(0,T)\times(0,\ell)$,
\[
 \int_E|f_K(u_K)|^2\,\dd x\,\dd t
 \leq C\bigl(|E|+|E|^{5/8}+|E|^{3/8}\bigr),
\]
by the uniform $L^4$ bound. Fatou's lemma and Vitali's theorem
therefore give $f(u)\in L^2$ and $f_K(u_K)\to f(u)$ strongly in
$L^2$. The Duhamel integrals converge in $C L^2$. Passing to the
limit in the mild formulation proves existence.

The same comparison gives uniqueness. Let $u,v$ belong to
$C L^2\cap L^4$ and solve the equation with the same initial
state and source. By \eqref{eq:log-subpower-growth}, their
nonlinearities belong to $L^2(Q)$. Rewriting both equations
with $f_K$, we obtain
\[
 \|u-v\|_{C L^2}
 \leq C(1+K)^\gamma e^{-K+T\Gamma_K}
       (T\ell^{1/2}+\sqrt T(\|u\|_{L^4}^2+\|v\|_{L^4}^2)).
\]
Letting $K\to\infty$, we obtain uniqueness without requiring a
stability constant uniform in $K$.
For $\gamma=0$, the last assertion follows from the one-sided
bound for the untruncated equation in Lemma~\ref{lem:log-truncation}.
\end{proof}

\subsection{Subcritical control near a reference trajectory}

We now apply the forced well-posedness theory to a truncated
control problem. The following lemma estimates the control cost
and the terminal error relative to the reference trajectory.

\begin{lemma}
\label{lem:log-short-control}
Let $u\in C([0,T];L^2)\cap L^4(Q)$ solve $Pu=f(u)$.
There exist $C>0$ and $B\geq1$ such that, for every interval
$I=[t_*,t_*+\tau]\subset[0,T]$ with $0<\tau\leq1$, every
$K\geq1$ and every $v\in L^2$, one can choose a control $h_K$
with the following properties. Set
\[
 q=\frac{2(1+\gamma)}3<1,\qquad
 E_K(\tau)=C(\tau^{-1}+K^q),\qquad M=\|v\|_2.
\]
The solution of the shifted truncated problem
\[
 Pz_K=f_K(u+z_K)-f_K(u)+\one_\omega h_K,
 \qquad z_K(t_*)=v,
\]
satisfies $z_K(t_*+\tau)=0$ and
\begin{align}
 \|h_K\|_{L^2(I\times\omega)}&\leq e^{E_K(\tau)}M,
                         \label{eq:log-truncated-cost}\\
 \|z_K\|_{C(I;L^2)}+\|z_K\|_{L^4(I\times\Om)}
                  &\leq e^{E_K(\tau)}M.
                         \label{eq:log-truncated-l4}
\end{align}
For the solution $Py_K=f(y_K)+\one_\omega h_K$ of the
untruncated equation with $y_K(t_*)=u(t_*)+v$,
\begin{equation}\label{eq:log-terminal-error}
 \|y_K(t_*+\tau)-u(t_*+\tau)\|_2
 \leq e^{-K+E_K(\tau)}(B+M^3).
\end{equation}
The constant $C$ depends only on the fixed interval, control
region and nonlinearity parameters. One may choose
\begin{equation}\label{eq:reference-constant}
 B=C\bigl(1+\|u\|_{C([0,T];L^2)}^3+\|u\|_{L^4(Q)}^2\bigr),
\end{equation}
after increasing $C$ to ensure $B\geq1$.
\end{lemma}

\begin{proof}
The function $f_K$ is globally Lipschitz with constant at most
$C(1+K)^{1+\gamma}$. Applying Lemma~\ref{lem:shifted-control}
to $f_K(u+z)-f_K(u)$ on $I$ after a time translation, we obtain
a null control and the $C L^2$ estimate with exponent
$C(\tau^{-1}+K^q)$. To estimate the space--time norm, we use
the interval $L^4$ bound and obtain
\[
 \|z_K\|_{L^4}
 \leq B_\ell\bigl(M+\tau\operatorname{Lip}(f_K)
                      \|z_K\|_{C L^2}+\sqrt\tau\|h_K\|_2\bigr).
\]
Absorbing the polynomial factor into the exponential, we obtain
\eqref{eq:log-truncated-cost}--\eqref{eq:log-truncated-l4}
after enlarging $C$; this constant remains independent of the
reference.

Lemma~\ref{lem:log-wellposed} gives the solution of the
untruncated equation with the same control. Combining its mass
and $L^4$ estimates, we obtain
\[
 \|y_K\|_{L^4(I\times\Om)}^2
 \leq C\left[1+
   \bigl(\|u(t_*)\|_2+M+\sqrt\tau\|h_K\|_2\bigr)^3\right].
\]
Here the constants absorb the fixed quantities $L_g$ and
$\ell^{1/2}|\mathfrak g(0)|$, since $\tau\leq1$.
It remains to compare $y_K$ with $u+z_K$. Writing both
equations with $f_K$ leaves a difference between their source
terms given by
\[
 [f(y_K)-f_K(y_K)]-[f(u)-f_K(u)].
\]
Their initial states agree. It therefore follows from
\eqref{eq:log-forced-stability} and the tail bound that
\[
 \|y_K-u-z_K\|_{C(I;L^2)}
 \leq C(1+K)^\gamma e^{-K+\tau\Gamma_K}
  \left[\tau\ell^{1/2}+\sqrt\tau
        (\|y_K\|_{L^4(I\times\Om)}^2+\|u\|_{L^4(I\times\Om)}^2)\right].
\]
Since $\gamma<q$ and $\tau\leq1$, we absorb the factors
$(1+K)^\gamma$ and $e^{\tau\Gamma_K}$, together with the
preceding state bound, into $e^{C(\tau^{-1}+K^q)}$.
Evaluating at the terminal time, using $z_K(t_*+\tau)=0$ and
enlarging $C$, we obtain
\eqref{eq:log-terminal-error} with \eqref{eq:reference-constant}.
\end{proof}

\begin{proof}[Proof of Theorem~\ref{thm:log-main}\textup{(i)}]
The class of nonlinearities is invariant under
$f^\sharp(z)=\overline{f(\bar z)}$: the real radial part is
unchanged and $\mathfrak g$ is replaced by the globally
Lipschitz function $\overline{\mathfrak g(\bar z)}$.
We solve the uncontrolled $f^\sharp$ equation forward from
$\overline{y_1}$ and conjugate its time reversal. This gives
an uncontrolled reference $u\in C([0,T];L^2)\cap L^4(Q)$
for $f$, with $u(T)=y_1$.

Set $S=\min\{T,1\}$. We construct a control joining $y_0$
to $u(S)$ by time $S$. To this end, put
\[
 a=\frac{6S}{\pi^2},\qquad \tau_j=\frac{a}{(j+1)^2},\qquad
 t_0=0,\quad t_{j+1}=t_j+\tau_j,
 \qquad K_j=A2^j.
\]
Then $t_j\uparrow S$. Let $C,B$ be as in
Lemma~\ref{lem:log-short-control}, and put
$M_0=\|y_0-u(0)\|_2$. Choose $A\geq1$ so that
\begin{equation}\label{eq:log-cascade-parameters}
 C(\tau_j^{-1}+K_j^q)\leq K_j/8\quad(j\geq0),\qquad
 \log\bigl(B+\max\{1,M_0^3\}\bigr)\leq A/8.
\end{equation}
Such a choice is possible because $q<1$. More explicitly,
$\sup_{j\geq0}(j+1)^2/2^j=9/4$, so it suffices that
\[
 A\geq\max\left\{1,\frac{36C}{a},
       (16C)^{1/(1-q)},
       8\log\bigl(B+\max\{1,M_0^3\}\bigr)\right\}.
\]

We apply the preceding reference-trajectory lemma on each
successive interval and use the resulting control in the true
equation. Set
$M_j=\|y(t_j)-u(t_j)\|_2$ and
$H_j=\|h_j\|_{L^2((t_j,t_{j+1})\times\omega)}$.
Then
\[
 H_j\leq e^{K_j/8}M_j,\qquad
 M_{j+1}\leq e^{-7K_j/8}(B+M_j^3).
\]
The parameter choice gives $M_1\leq e^{-3K_0/4}$.
Induction yields
\[
 M_{j+1}\leq e^{-K_j/2},\qquad
 H_j\leq e^{K_j/8-K_{j-1}/2}=e^{-K_j/8}\quad(j\geq1).
\]
Indeed, once $M_j<1$,
$\log(B+M_j^3)\leq\log(B+1)\leq A/8\leq K_j/8$.
Consequently
\[
 \sum_{j\geq0}H_j^2
 \leq e^{A/4}M_0^2+\sum_{j\geq1}e^{-A2^j/4}<\infty.
\]
The concatenated control therefore belongs to $L^2$. We extend
it by zero on $(S,T)$ and apply Lemma~\ref{lem:log-wellposed}
on the full interval. Uniqueness identifies this global solution
with each successive local solution. Its continuity and
$y(t_j)-u(t_j)\to0$ imply $y(S)=u(S)$. Uniqueness on $[S,T]$
then gives $y(T)=u(T)=y_1$. The global well-posedness lemma
also gives the asserted $L^4$ and nonlinear $L^2$ regularity.
This continuity argument identifies the accumulated endpoint
without requiring a mass estimate for the difference $y-u$.

We next prove the quantitative assertion. Let $0<T\leq1$ and
$d=\|y_0\|_2+\|y_1\|_2$. The backward reference satisfies
\[
 \|u\|_{C L^2}\leq C(1+d),\qquad
 \|u\|_{L^4}^2\leq C(1+d^3).
\]
Thus $B\leq C(1+d^3)$ and $M_0\leq C(1+d)$. Put
$W=B+\max\{1,M_0^3\}$. All requirements in
\eqref{eq:log-cascade-parameters} hold for
\[
 A=C_0/T+C_1+8\log W
\]
with sufficiently large fixed $C_0,C_1$. In particular,
\[
 A\leq C\bigl(T^{-1}+\log(2+d)\bigr).
\]
We retain the original control estimate on the first interval,
rather than replacing its exponent by $A/8$. Since
$\tau_0=6T/\pi^2$, we obtain
\[
 \begin{aligned}
 H_0&\leq \exp\!\left(C(\tau_0^{-1}+A^q)\right)M_0\\
 &\leq C e^{C/T}(1+d)
                  \exp\!\left(C[\log(2+d)]^q\right).
 \end{aligned}
\]
Here we used $(a+b)^q\leq a^q+b^q$ for $a,b\geq0$, and
$T^{-q}\leq T^{-1}$. The sum of the squared norms of the
remaining controls is bounded independently of $d$ and $T$.
The mass estimate \eqref{eq:log-mass} gives the same bound for
the supremum norm of the state. Thus \eqref{eq:log-main-cost}
holds for $0<T\leq1$.

For $T>1$, let $J=\lceil T\rceil$ and
$\tau=T/J\in(1/2,1]$. We apply the preceding estimate on the
$J$ successive intervals, joining $y_0$ to zero on the first,
zero to $y_1$ on the last, and zero to zero on each intervening
interval. Since the interval lengths satisfy
$\tau^{-1}<2$, the constants are uniform over these controls.
The squared control norms add, and the state norm is the
maximum of the segment norms. Hence
\[
 \begin{aligned}
 \|h\|_{L^2(Q_\omega)}^2
 &\leq C(1+d)^2\exp\!\left(2C[\log(2+d)]^q\right)+CJ,\\
 \|y\|_{C([0,T];L^2)}
 &\leq C(1+d)\exp\!\left(C[\log(2+d)]^q\right).
 \end{aligned}
\]
Since $J\leq T+1\leq2T$, these bounds give
\eqref{eq:log-main-cost} for $T>1$. The states match at every
junction, so their finite concatenation is a mild solution and
retains the stated $L^4$ and nonlinear $L^2$ regularity.
When $f(0)\ne0$, the intermediate trajectories and controls
need not vanish; their contribution gives the $\sqrt T$ term.

Finally, for every $\varepsilon>0$, the inequality $q<1$ gives
\[
 \sup_{r\geq0}\bigl(Cr^q-\varepsilon r\bigr)<\infty.
\]
Applying this with $r=\log(2+d)$ and using
$2+d\leq2(1+d)$, we obtain
\eqref{eq:log-almost-linear-cost}.

\end{proof}

\begin{proof}[Proof of Theorem~\ref{thm:log-main}\textup{(ii)}]
For $0<\sigma\leq1$, Theorem~\ref{thm:log-main}\textup{(i)}
gives a uniform null-control cost
$C_\lambda e^{C_\lambda/\sigma}$ on the unit ball. For $a>0$,
the transformation
\[
 (y,h)\longmapsto
 (ae^{-i\lambda\log(a^2)t}y,
  ae^{-i\lambda\log(a^2)t}h)
\]
preserves this equation and multiplies all relevant norms by $a$.
Applying this scaling to normalized nonzero initial states gives
a null control with cost linear in the initial norm. For the zero
state, we use the zero control. For arbitrary $T>0$, set
$\sigma=\min\{T/2,1/2\}$. Join $y_0$ to zero on $[0,\sigma]$,
and use the conjugate time reversal of a null trajectory from
$\overline{y_1}$ on $[T-\sigma,T]$. On the intervening interval,
take the state and control to be zero, which is possible since
$f(0)=0$. The equation has the conjugation symmetry required
for this time reversal. The squared control norms add, and
the state norm is the maximum over the two controlled intervals.
Since $\sigma^{-1}=2\max\{T^{-1},1\}\leq2(T^{-1}+1)$,
increasing $C_\lambda$ gives
\eqref{eq:canonical-log-fast-cost} for every $T>0$.
\end{proof}

\begin{remark}
The comparison between the control cost and the truncation
error relies on $K^q=o(K)$, where
$q=2(1+\gamma)/3<1$. The nonuniform stability factor
$e^{C\tau K^\gamma}$ also has a sublinear exponent and is
included in the error estimate. The argument therefore covers
logarithmic powers strictly below $3/2$, including those with a
nonconservative Lipschitz perturbation. It does not apply to
polynomial nonlinearities. The critical logarithmic exponent is
treated below under a smallness condition; the case of arbitrary
critical coefficients remains open. For $\gamma>0$, the
uniqueness class includes the $L^4$ condition needed for the
tail comparison. Uniqueness in general $C_tL^2_x$ mild classes
is not asserted.
\end{remark}

\subsection{The critical logarithmic exponent}\label{subsec:critical-log}

At $\gamma=1/2$, the exponent in the truncated control cost is
linear in $K$. We therefore keep its leading coefficient
explicit and show that smallness of this coefficient allows the
truncation error to decay fast enough for the cascade. We first
treat the pure critical derivative bound and then include the
arbitrary lower-order term in
\eqref{eq:critical-log-two-coefficients}.

\begin{proof}[Proof of Theorem~\ref{thm:log-critical-small}]
Choose $C_*\geq1$ so that \eqref{eq:shifted-lipschitz-cost}
holds with $C_*$ in both occurrences of $C$, and set
\[
 D_*:=\frac35\left(\frac1{64C_*}\right)^{3/2},\qquad
 \mathfrak a(D):=C_*\left(\frac{5D}{3}\right)^{2/3},\qquad
 c_*:=24C_*\left(\frac53\right)^{2/3}.
\]
\emph{Step 1. A pure critical derivative bound.}
Assume first that $D_0=0$, and write $D=D_1$. Thus
\begin{equation}\label{eq:critical-log-smallness}
 r|\phi'(r)|\leq D\sqrt{1+|\log r|},\qquad 0\leq D\leq D_*.
\end{equation}
The truncation, mass, and growth bounds in
Lemma~\ref{lem:log-truncation} remain valid when $\gamma=1/2$.
The proof of Lemma~\ref{lem:log-wellposed} also applies here:
the Cauchy and uniqueness comparisons require only $\gamma<1$,
since $(1+K)^\gamma e^{-K+T(L_g+D(1+K)^\gamma)}\to0$.
The lemma therefore gives the global forced solution and the
uniform bound $\|y\|_{L^4}\leq C(1+M^{3/2})$. We keep
$L_g=\operatorname{Lip}(\mathfrak g)$ fixed.

To compare the control cost with the truncation error, we retain
the dependence of the control exponent on $D$.
Integration of \eqref{eq:critical-log-smallness} gives
\[
 |\phi_K(r)|\leq |\phi(1)|
                 +\frac{2D}{3}(1+K)^{3/2},
 \qquad r|\phi_K'(r)|\leq D(1+K)^{1/2}.
\]
Consequently, with $L_0=|\phi(1)|+L_g$,
\[
 L_K:=\operatorname{Lip}(f_K)
 \leq L_0+\frac{5D}{3}(1+K)^{3/2},
 \qquad
 L_K^{2/3}\leq L_0^{2/3}
                   +\left(\frac{5D}{3}\right)^{2/3}(1+K).
\]
Let $u$ be an uncontrolled reference trajectory. On any interval
$I=[t_*,t_*+\tau]$, $0<\tau\leq1$, we apply
Lemma~\ref{lem:shifted-control} to
$f_K(u+z)-f_K(u)$, with initial state $v$ and terminal state
zero. This shifted nonlinearity has Lipschitz constant at most
$L_K$ and vanishes at zero. Writing $a=\mathfrak a(D)$, we obtain
\begin{equation}\label{eq:critical-truncated-control}
 \|h_K\|_{L^2(I\times\omega)}+\|z_K\|_{C(I;L^2)}
 \leq e^{c_0/\tau+aK}M,\qquad M=\|v\|_2,
\end{equation}
where, for example, one may take
\[
 c_0=C_*+\log C_*+C_*L_0^{2/3}+a.
\]
The coefficient of $K$ in this estimate depends only on
$D$ and the geometry. In particular, increasing $c_0$ does not
increase $a$. The shifted nonlinearity belongs to $L^1(I;L^2)$,
since its modulus is at most $L_K|z_K|$.
Lemma~\ref{lem:interval-l4} therefore gives
$z_K\in L^4(I\times\Om)$.

Let $y_K$ solve the true equation with this control and initial
state $u(t_*)+v$. Put
\[
 R=\|u\|_{C([0,T];L^2)},\qquad
 B=1+R^3+\|u\|_{L^4(Q)}^2\geq1.
\]
The mass estimate and the uniform $L^4$ bound imply
\[
 \|y_K\|_{L^4(I\times\Om)}^2
 \leq C\left(1+R^3+e^{3c_0/\tau+3aK}M^3\right).
\]
We now estimate the truncation error. As in
Lemma~\ref{lem:log-short-control}, we compare $y_K$ and
$u+z_K$ using the same nonlinearity $f_K$. Their initial states
agree, and their source terms differ by
$[f(y_K)-f_K(y_K)]-[f(u)-f_K(u)]$. Thus
\[
 \begin{aligned}
 \|y_K-u-z_K\|_{C(I;L^2)}
 &\leq C(1+K)^{1/2}
 e^{-K+\tau(L_g+D\sqrt{1+K})}\\
 &\quad\times
 \left[\tau\ell^{1/2}+\sqrt\tau
  \bigl(\|y_K\|_{L^4(I\times\Om)}^2
             +\|u\|_{L^4(I\times\Om)}^2\bigr)\right].
 \end{aligned}
\]
For every $\eta>0$,
\[
 (1+K)^{1/2}e^{D\sqrt{1+K}}
       \leq C_{D,\eta}e^{\eta K}\qquad(K\geq1).
\]
Since $\tau\leq1$, we absorb all factors independent of $K$
into $e^{c_\eta/\tau}$, with $c_\eta\geq c_0$.
Evaluating at the terminal time, where $z_K=0$, gives
\begin{equation}\label{eq:critical-terminal-error}
 \|y_K(t_*+\tau)-u(t_*+\tau)\|_2
 \leq e^{c_\eta/\tau-(1-3a-\eta)K}(B+M^3).
\end{equation}
Here $c_\eta$ depends only on the fixed geometry, the displayed
nonlinearity parameters and $\eta$; it is independent of $u$,
$K$, $\tau$, and $v$.

We construct the backward reference from $y_1$ using the
conjugate nonlinearity, as in the proof of
Theorem~\ref{thm:log-main}. Set $S=\min\{T,1\}$ and
\[
 b=\frac{6S}{\pi^2},\qquad
 \tau_j=\frac{b}{(j+1)^2},\qquad K_j=A2^j,
 \qquad M_0=\|y_0-u(0)\|_2.
\]
Set $t_0=0$ and $t_{j+1}=t_j+\tau_j$, so that $t_j\uparrow S$.
On each interval, we choose the preceding control and write
\[
 M_j=\|y(t_j)-u(t_j)\|_2,\qquad
 H_j=\|h_j\|_{L^2((t_j,t_{j+1})\times\omega)}.
\]
Fix $\eta=1/64$. Our smallness assumption gives $a\leq1/64$.
Choose
\begin{equation}\label{eq:critical-cascade-parameter}
 A=1+\frac{36c_\eta}{b}
          +8\log\bigl(B+\max\{1,M_0^3\}\bigr).
\end{equation}
Since $\sup_{j\geq0}(j+1)^2/2^j=9/4$, this ensures
$c_\eta/\tau_j\leq K_j/16$. Moreover,
$3a+\eta\leq1/16$. Equations
\eqref{eq:critical-truncated-control} and
\eqref{eq:critical-terminal-error} therefore give the recursive
bounds used in the subcritical cascade:
\[
 H_j\leq e^{K_j/8}M_j,\qquad
 M_{j+1}\leq e^{-7K_j/8}(B+M_j^3).
\]
The logarithmic term in \eqref{eq:critical-cascade-parameter}
gives $M_1\leq e^{-3K_0/4}$. Once $M_j\leq1$,
$\log(B+M_j^3)\leq A/8\leq K_j/8$.
Induction therefore yields
\[
 M_{j+1}\leq e^{-K_j/2},\qquad
 H_j\leq e^{-K_j/8}\quad(j\geq1).
\]
Thus $\sum_jH_j^2<\infty$. We extend the concatenated control
by zero on $(S,T)$ and apply Lemma~\ref{lem:log-wellposed}
on $[0,T]$. By uniqueness, this global forced solution agrees
with each local trajectory. Its continuity gives $y(S)=u(S)$,
and uniqueness on $[S,T]$ then yields $y(T)=y_1$. The
global well-posedness lemma also gives the asserted $L^4$ and
nonlinear $L^2$ regularity.

To estimate the cost, suppose $T\leq1$ and set
$d=\|y_0\|_2+\|y_1\|_2$. The reference bounds imply
\[
 B\leq C(1+d^3),\qquad M_0\leq C(1+d),\qquad
 W:=B+\max\{1,M_0^3\}\leq C(1+d)^3.
\]
We retain the coefficient $a$ in the first control estimate.
Since $\tau_0=b=6T/\pi^2$, we obtain
\[
 \begin{aligned}
 H_0&\leq e^{c_0/b+aA}M_0
       \leq C e^{C/T}W^{8a}M_0\\
    &\leq C e^{C/T}(1+d)^{1+24a}.
 \end{aligned}
\]
The sum of the squared norms of the remaining controls is
uniformly bounded. The mass estimate gives the same bound for
the supremum norm of the state. We therefore obtain
\begin{equation}\label{eq:critical-log-cost}
 \|h\|_{L^2(Q_\omega)}+\|y\|_{C([0,T];L^2)}
 \leq C e^{C/T}(1+d)^{1+24\mathfrak a(D)}.
\end{equation}

\medskip\noindent
\emph{Step 2. An arbitrary lower-order term.}
Set $a=\mathfrak a(D_1)$ and $L_0=|\phi(1)|+L_g$.
We use the backward reference construction from Step~1, with
$M_0=\|y_0-u(0)\|_2$ and
$B=1+\|u\|_{C L^2}^3+\|u\|_{L^4(Q)}^2$.
The forced well-posedness and tail estimates hold with the
unrestricted constant $D_0+D_1$ and exponent $\gamma=1/2$.
To estimate the control cost, we keep the separate contributions
in the Lipschitz bound:
\[
 L_K\leq L_0+D_0(1+K)+\frac{5D_1}{3}(1+K)^{3/2}.
\]
Subadditivity of the power $2/3$ gives, with
$b=C_*D_0^{2/3}$,
\[
 \|h_K\|_2+\|z_K\|_{C L^2}
 \leq e^{c_0/\tau+aK+bK^{2/3}}M.
\]
Here $c_0$ absorbs the constant terms from
$(1+K)^{2/3}\leq1+K^{2/3}$; the coefficients $a$ and $b$
remain unchanged.
The comparison with the true equation uses the same truncation
as in Step~1 and gives, for every $\eta>0$,
\[
 M_{\rm next}\leq
 e^{c_\eta/\tau-(1-3a-\eta)K+3bK^{2/3}}(B+M^3).
\]
Fix $\eta=1/64$. Since
$3bK^{2/3}\leq K/64+C_b$, the constant $C_b$ can be
absorbed into $\bar c/\tau$, with
$\bar c\geq c_\eta,c_0$ depending only on the fixed
parameters. The same choice of $\bar c$ also gives
\[
 \|h_K\|_2\leq e^{\bar c/\tau+(a+1/192)K}M,
 \qquad
 M_{\rm next}\leq
 e^{\bar c/\tau-(1-3a-1/32)K}(B+M^3).
\]
Set $s=6\min\{T,1\}/\pi^2$,
$\tau_j=s/(j+1)^2$, and $K_j=A2^j$. Choose the cascade parameter
\[
 A=1+\frac{72\bar c}{s}
       +8\log\bigl(B+\max\{1,M_0^3\}\bigr).
\]
Then $\bar c/\tau_j\leq K_j/32$, while $a\leq1/64$.
The control exponent is at most
$(1/32+1/64+1/192)K_j<K_j/8$, and the loss in the
terminal-error exponent is at most
$(1/32+3/64+1/32)K_j=7K_j/64<K_j/8$.
These bounds allow us to use the recursive estimates and the
infinite concatenation from Step~1.

For $T\leq1$, we have
$A\leq C(T^{-1}+\log(2+d))$. To estimate the first control,
we retain the unabsorbed $bA^{2/3}$ term and obtain
\[
 H_0\leq e^{c_0/\tau_0+aA+bA^{2/3}}M_0
 \leq C e^{C/T}(1+d)^{1+24a}
              e^{C[\log(2+d)]^{2/3}}.
\]
Estimating the remaining controls and the state as before, we obtain
\begin{equation}\label{eq:critical-log-lower-order-cost}
 \|h\|_{L^2(Q_\omega)}+\|y\|_{C([0,T];L^2)}
 \leq C e^{C/T}(1+d)^{1+24a}
       \exp\!\left(C[\log(2+d)]^{2/3}\right).
\end{equation}
Since $24a=c_*D_1^{2/3}$, we absorb the final factor into
$C_\varepsilon(1+d)^\varepsilon$ to obtain
\eqref{eq:critical-main-cost} for $0<T\leq1$.
For $T>1$, we use the finite concatenation from the proof of
Theorem~\ref{thm:log-main}\textup{(i)}, with
$J=\lceil T\rceil$ and $\tau=T/J\in(1/2,1]$.
The first and last intervals contribute at most
$C_\varepsilon(1+d)^{2(1+c_*D_1^{2/3}+\varepsilon)}$
to the squared control norm, and the remaining intervals
contribute at most $C_\varepsilon J$.
The state norm is bounded by the maximum of the segment norms.
Taking square roots and using $J\leq2T$, we obtain
\eqref{eq:critical-main-cost} for all $T>0$, with the same
smallness threshold $D_*$ and exponent $c_*D_1^{2/3}$.
This completes the proof.
\end{proof}

\appendix
\section{Verification of the planar phase condition}
\label{app:planar-phase}

In this appendix, we shall verify Assumption~\ref{ass:phase-geometry} for the control
region defined in \eqref{eq:planar-control-region}.
The uncontrolled compact set is $K_\omega=K_0\cup K_1$, where
\[
 K_0=\{x\in\overline\Om:d(x)\geq2\delta\},
 \qquad
 K_1=\{x\in\overline\Om:d(x)\leq\delta,\ \pi(x)\notin\Gamma_c\}.
\]
The two sets are separated by a positive distance.

Near $K_0$, take
\[
 \psi_0(x)=|x-z|^2,
 \qquad z\in\R^2\setminus\overline\Om.
\]
Then $D^2\psi_0=2I$, and $\nabla\psi_0$ is bounded away
from zero near $K_0$.

For the phase near $K_1$, cut the boundary at a point strictly
inside $\Gamma_c$. The resulting arclength coordinate $s$ is
single-valued on a neighborhood of
$\partial\Om\setminus\Gamma_c$.
Let $\gamma(s)$ be the boundary parametrization,
$\tau=\gamma'$, and $\nu$ the outward unit normal.
Define the signed curvature by $\nu'=k\tau$.
In the normal coordinates
\[
 X(s,r)=\gamma(s)-r\nu(s),
\]
where $r>0$ inside $\Om$, set
\[
 \psi_1(X(s,r))=\exp(as+br^2),
 \qquad
 a>0,\quad
 2b>\|k\|_{L^\infty(\partial\Om)}^2.
\]
This formula also defines $\psi_1$ smoothly for small negative
$r$. At $r=0$,
\[
 \partial_\nu\psi_1=0,
 \qquad
 \nabla\psi_1=ae^{as}\tau,
 \qquad
 D^2\psi_1=e^{as}
 \begin{pmatrix}
  a^2 & -ak\\
  -ak & 2b
 \end{pmatrix}
 \quad\text{in the frame }(\tau,\nu).
\]
The first diagonal entry is positive, and the determinant is
\[
 a^2e^{2as}(2b-k^2)>0.
\]
By compactness of $\partial\Om\setminus\Gamma_c$, the Hessian
is uniformly positive definite and the gradient is bounded
away from zero on this boundary arc.
With $a,b$ fixed, continuity preserves these properties near
$K_1$ when $\delta>0$ is sufficiently small.

We combine the two local phases using disjoint open neighborhoods $U_j\subset\R^2$ of
$K_j$, $j=0,1$, with $U_0\Subset\Om$ and $U_1$ contained
in the normal coordinate chart.
Take $\chi_j\in C_c^\infty(U_j)$ equal to one near $K_j$,
and define
\[
 \psi=\chi_0\psi_0+\chi_1\psi_1,
\]
where each product is extended by zero outside $U_j$.
Then $\psi$ is real-valued and smooth on $\R^2$, and
$\psi=\psi_j$ near $K_j$.

By compactness, there exist $c_0,c_1>0$ and a relatively
open neighborhood $\mathcal O$ of $K_\omega$ in
$\overline\Om$ such that
\[
 \begin{aligned}
 |\nabla\psi|\geq c_0,\qquad D^2\psi\geq c_1I
 &\quad\text{on }\mathcal O,\\
 \partial_\nu\psi=0
 &\quad\text{on }\mathcal O\cap\partial\Om.
 \end{aligned}
\]
The boundary part of $\mathcal O$ lies where $\psi=\psi_1$,
and hence the normal derivative vanishes.
This verifies Assumption~\ref{ass:phase-geometry}.

\section*{Acknowledgements}
This work was partially supported by the National Key Research and
Development Program of China (Grant No.~2023YFC2206100), the National
Natural Science Foundation of China (Grant No.~12231008), and the
Fundamental Research Funds for the Central Universities
(Grant No.~2026BRSXB003).

\section*{Conflict of interest}
The authors declare no conflict of interest.

\section*{Data availability}
No data were generated or used in this study.

\end{document}